\documentclass[12pt,a4paper]{amsart}

\makeatletter
\def\@tocline#1#2#3#4#5#6#7{\relax
\ifnum #1>\c@tocdepth 
  \else
    \par \addpenalty\@secpenalty\addvspace{#2}%
\begingroup \hyphenpenalty\@M
    \@ifempty{#4}{%
      \@tempdima\csname r@tocindent\number#1\endcsname\relax
 }{%
   \@tempdima#4\relax
 }%
 \parindent\z@ \leftskip#3\relax \advance\leftskip\@tempdima\relax
 \rightskip\@pnumwidth plus4em \parfillskip-\@pnumwidth
 #5\leavevmode\hskip-\@tempdima #6\nobreak\relax
 \ifnum#1<0\hfill\else\dotfill\fi\hbox to\@pnumwidth{\@tocpagenum{#7}}\par
 \nobreak
 \endgroup
  \fi}
  \def\l@subsection{\@tocline{1}{0pt}{30pt}{5pc}{}}
\makeatother

\usepackage{amsmath,amssymb,amsthm}
\usepackage{graphicx}
\usepackage{footnote}

\numberwithin{equation}{section}

\newcommand{\punt}{\mathbf{.}}
\DeclareMathAlphabet{\mathbbold}{U}{bbold}{m}{n}

\DeclareSymbolFont{rsfscript}{OMS}{rsfs}{m}{n}
\DeclareSymbolFontAlphabet{\mathrsfs}{rsfscript}
\DeclareFontFamily{OMS}{rsfs}{\skewchar\font'177}
\DeclareFontShape{OMS}{rsfs}{m}{n}{%
      <5> rsfs5
      <6> <7> rsfs7
      <8> <9> <10> rsfs10
      <10.95> <12> <14.4> <17.28> <20.74> <24.88> rsfs10
      }{}

\def\calA{\mathrsfs{A}}

\def\calX{\mathrsfs{X}}
\def\calY{\mathrsfs{Y}}

\theoremstyle{plain}

\newtheorem{theorem}{Theorem}

\newtheorem{corollary}[theorem]{Corollary}

\newtheorem{proposition}[theorem]{Proposition}

\theoremstyle{definition}

\newtheorem{definition}[theorem]{Definition}
\newtheorem{example}[theorem]{Example}

\newtheoremstyle{dotless}{}{}{\itshape}{}{\bfseries}{}{ }{}
\theoremstyle{dotless}
\newtheorem*{exampleno}{Example}

\theoremstyle{remark}
\newtheorem{remark}[theorem]{Remark}

\numberwithin{theorem}{section}

\newcommand{\cop}{\scriptscriptstyle <-1>}
\newcommand{\K}{\bar{\mathfrak{K}}}
\newcommand{\LL}{\mathfrak{L}}
\newcommand{\ab}{\bar{\alpha}_{\scriptscriptstyle D}^{\scriptscriptstyle <-1>}}

\newcommand{\ibs}{\boldsymbol{i}}
\newcommand{\kbs}{\boldsymbol{k}}
\newcommand{\jbs}{\boldsymbol{j}}
\newcommand{\etabs}{\boldsymbol{\eta}}
\newcommand{\tbs}{\boldsymbol{t}}

\newcommand{\xbs}{\boldsymbol{x}}
\newcommand{\Xbs}{\boldsymbol{X}}
\newcommand{\zbs}{\boldsymbol{z}}
\newcommand{\mubs}{\boldsymbol{\mu}}
\newcommand{\gammabs}{\boldsymbol{\gamma}}

\newcommand{\nubs}{\boldsymbol{\nu}}
\newcommand{\lambdabs}{\boldsymbol{\lambda}}

\newcommand{\mmodels}{\boldsymbol{\vdash}}
\newcommand{\X}{\mathbf{X}}

\newcommand{\E}{\mathbb E}
\newcommand{\bio}{\boldsymbol{\iota}}
\newcommand{\ind}{\mathbf{1}}
\newcommand{\mbs}{\boldsymbol{m}}
\newcommand{\pbs}{\boldsymbol{p}}
\newcommand{\vbs}{\boldsymbol{v}}

\begin{document}

\title{Symbolic Calculus in Mathematical Statistics: a review}
\author{Elvira Di Nardo}
\address{Dept. of Mathematics, Computer Science and Economics,
         University of Basilicata,
         Campus Macchia Romana,
         Via dell'Ateneo Lucano 10,
         85100 Potenza, Italy}
\email{elvira.dinardo@unibas.it}

\begin{abstract}
In the last ten years, the employment of symbolic methods has substantially extended both the theory and the applications of statistics and probability. This survey reviews the development of a symbolic technique arising from classical umbral calculus, as introduced by Rota and Taylor in $1994.$ The usefulness of this symbolic technique is twofold. The first is
to show how new algebraic identities drive in discovering insights among topics apparently very far from each other and related to probability and statistics. One of the main tools is a formal generalization of the convolution of identical probability distributions, which allows us to employ compound Poisson random variables in various topics that are only somewhat interrelated. Having got a different and deeper viewpoint, the second goal is to show how to set up algorithmic processes performing efficiently algebraic calculations. In particular, the challenge of finding these symbolic procedures should lead to a new method, and it poses new problems involving both computational and conceptual issues. Evidence of efficiency in applying this symbolic method will be shown within statistical inference, parameter estimation, L\'evy processes, and, more generally, problems involving multivariate functions. The symbolic representation of Sheffer polynomial sequences allows us to carry out a unifying theory of classical, Boolean and free cumulants. Recent connections within random matrices have extended the applications of the symbolic method.
\end{abstract}

\maketitle

\thispagestyle{myheadings}
\font\rms=cmr8
\font\its=cmti8
\font\bfs=cmbx8

\markright{\its S\'eminaire Lotharingien de
Combinatoire \bfs 67 \rms (2015), Article~B67a \hfill}
\def\thepage{}

\tableofcontents

\pagenumbering{arabic}
\addtocounter{page}{1}
\markboth{\SMALL ELVIRA DI NARDO}{\SMALL
SYMBOLIC CALCULUS IN MATHEMATICAL STATISTICS}

\section{Introduction}

Most practical problems do not have exact solutions, but numerical computing helps in finding solutions to desired precision.  Numerical methods give data that can be analyzed to reveal trends or approximations. In contrast to numerical methods, symbolic methods treat objects that are either formal expressions, or are algebraic in nature. In many cases, numerical methods will not give sufficient information about the nature of the problem; this is where symbolic methods fit in. When used properly, they can give us more insight into the problem we are trying to solve. Another advantage is that, while numerical methods may simply fail to compute correct results,  symbolic methods yield closed or explicit formulas. Although the term {\lq\lq symbolic method\rq\rq}  is also used in different contexts (see for example~\cite{Flajolet}), here we refer to a set of manipulation techniques aimed at performing algebraic calculations, preferably through an algorithmic approach, in order to find efficient mechanical processes that may be passed to a computer. This is usually called \lq\lq symbolic computation\rq\rq.
Despite the fact that its development took place later than that of  numerical or graphical algorithms, symbolic computation has had a comparable impact, especially in theoretical and applied statistical research. Topics range from asymptotic expansions to Edgeworth series, from likelihood functions to saddlepoint approximations.

   Since symbolic computation is exact, its accuracy is not a fundamental issue, as it is for numerical methods. Replacing  pencil and paper by a computer is not the only aim of symbolic methods.
    Computational cost is a question naturally arising when problems of complex nature are addressed. The advantage of symbolic methods is that, by approaching a problem from a different mathematical point of view, efficiency is sometimes a by-product.

   The symbolic method we are going to review arises from the so-called \lq\lq classical umbral calculus\rq\rq, a creation of Gian-Carlo Rota, who spent most of his scientific life  looking for an elegant    and as general as possible formalization of this calculus.

   It is certainly due to Rota's efforts that the heuristic method, invented and largely employed by Rev.~John Blissard between $1861$ and $1868$ (see references in \cite{Dinardo1}), now has a solid foundation and wide-spread application within the mathematical community. Despite the impressive amount of scientific literature following the introduction of this calculus, Rota, in $1994$, had the courage to turn his umbral calculus upside down, outlining what he considered to be a new and correct syntax for this matter. The classical umbral calculus we use consists essentially in a moment calculus, since its basic device is to represent a unital sequence of numbers by  symbols $\alpha,$ called umbra, i.e., to associate the sequence $1, a_1, a_2, \ldots$ to the sequence $1, \alpha, \alpha^2, \ldots$ of powers of $\alpha$ via an operator $\E$ that looks like the expectation of random variables (r.v.'s)  \cite{eleventh,SIAM}. The $i$-th element of the sequence is then called the {\it $i$-th moment\/} of $\alpha.$ This symbolic approach shares with free probability \cite{Nica,Speicher} the use of moments as a tool to characterize r.v.'s. An analog of freeness\footnote{In free probability, free r.v.'s correspond to classical independent r.v.'s.} is defined in a manner resembling independence of tensors. In contrast to free probability, where different linear operators may be applied to the same non-commutative r.v., in the symbolic method of moments just one operator is employed, but the same sequence of moments may correspond to more than one symbol. For a pair of commutative r.v.'s, freeness is equivalent to at least one of them having vanishing variance. So freeness is a pure non-commutative device,  and the symbolic approach we are going to review can be considered as the natural analog of a non-commutative field.

   \lq\lq As sometimes happens in the practice of the mathematical investigation, the subject we deal with here does not develop the original idea from which our research started in the spring of 1997, but this paper is closely related to it. In that period, Gian-Carlo Rota was visiting professor at the University of Basilicata and, during a conversation just before he left, he shared with us his keen interest in a research project: to develop a combinatorial random variable theory. The delicate question arising on the underlying  foundational side and the short time available led us to continue the discussion via email, weaving it with different activities over a period of several months. The following year, Gian-Carlo Rota held his last course in Cortona; we seized the opportunity to spend some time with him. We resumed the thread of our conversations and presented him with the doubts that gradually took hold of us. As usually, his contribution disclosed new horizons that have led us to write these pages.\rq\rq$\,$ This quotation is a part of the introduction of the paper {\sl ``Umbral nature of the Poisson random variables{}"} written by the author together with Domenico Senato and published in a volume dedicated to Gian-Carlo Rota after his death \cite{Dinardo1}. This paper was the first of a substantial list of articles devoted to the applications of classical umbral calculus in probability and statistics. In these papers, more definitions have been added, and the method has been restyled according to the needs faced when tackling some particular problems. The more these new tools have been developed, the closer the method has come to the method of moments employed in the study of random matrices, which the author believes to be a natural milestone of this theory and at the same time a good starting point for further developments and applications. This is why the version here reviewed is called {\it symbolic method of moments} in a talk \cite{talk} on new challenges and future developments of mathematical statistics given by the author in Dublin in $2011.$ Two papers have to be mentioned in support of this personal belief: a paper published in the {\sl Annals of Statistics} devoted to the computation of spectral statistics for random matrices \cite{Spectral} and a second one, published in  the {\sl Journal of Multivariate Analysis}, where the symbolic method of moments allows us to compute very general joint moments of non-central Wishart distributions \cite{Wishart}. The author chose to not include these topics in this review, since the overall setting is not yet completely settled, and a serious rethinking of the matter in terms of cumulants%
   \footnote{For the definition of cumulants of a number sequence, see Section~4 and in particular Equation (\ref{(cumRota)}).} instead of moments seems necessary.

   The first goal of this paper is to give a general overview of how the method has been developed and where it has been used with success.
While sticking to the original idea, the current version of the method
offers a much richer set of tools.
Therefore the introduction of these tools is done gradually along the paper according to the outline of applications.
The second goal is to prove its effectiveness in some applications that over the years have contributed to refining the theory. An example particularly meaningful is the computation of $k$-statistics (Section~6.1), a challenging problem since their first introduction in 1929  as unbiased estimators of cumulants \cite{Fisher}. Many authors have proposed different techniques aiming at performing this computation in a reasonable time. By using the symbolic method here reviewed, an efficient computation of $k$-statistics can be performed by means of a suitable generalization of randomized compound Poisson r.v.'s.

 The paper is organized as follows. Notation is introduced in Section~2 using the semantics of r.v.'s. In this way, the intention is to permit the reader with no prior knowledge of symbolic calculations to become comfortable with them. Section~2.1 shows how to perform these computations in dealing with moments of sampling distributions, reducing the underlying combinatorics of symmetric functions to a few relations covering a great variety of calculations. Section~2.2 is devoted to a special statistical device, called Sheppard's correction, useful in correcting estimations of moments when sample data are grouped into classes. The resulting closed-form formula gives an efficient algorithm to perform computations. In the same section, its multivariate version is sketched, showing what we believe is one of the most interesting  feature of this symbolic method: its multivariate generalization. More details are given in Section~7, where a symbolic version of the multivariate Fa\`a di Bruno formula is also introduced. The umbral calculus is particularly suited not only in dealing with number sequences, but also with polynomial sequences. Time-space harmonic polynomials are an example in which the symbolic method allows us to simplify the proofs of many properties, as well as their computation. Introduced in Section~3, these polynomials generate stochastic processes that result in a special family of martingales employed in mathematical finance.  Time-space harmonic polynomials rely on the symbolic representation of L\'evy processes, as given in Section~4.2,  obtained by using a generalization of compound Poisson processes, as given in Section~4.1. A L\'evy process is better described by its cumulant sequences than by its moments, as happens for many other r.v.'s, the Gaussian or the Poisson r.v.'s.  Section~4 is devoted to introducing umbrae whose moments are cumulants of a given sequence of numbers. Besides classical cumulants, quite recently \cite{Speicher} new families of cumulants have been introduced,  the Boolean and the free cumulants by using the algebra of multiplicative functions on suitable lattices. This link is exploited in Section~6.2, where the computational efficiency  of the symbolic method of moments is highlighted. In particular, by means of the symbolic representation of Abel polynomials, all families of cumulants can be parametrized by the same formula. This parametrization allows us to write  an algorithm giving moments in terms of classical, Boolean and free cumulants, and vice-versa. Other topics on Sheffer sequences are treated in Section~5, including Lagrange inversion and Riordan arrays, and solutions of difference equations. Section~6.1 shows how to compute unbiased estimators of cumulants by using exponential polynomials, which are special Sheffer sequences. The resulting algorithm is faster than others given in the literature, and it has been generalized to the multivariate case. In the last section, the multivariate umbral calculus is introduced, and various applications are given involving special families of multivariate polynomials.

This text is the written version of lectures given at the 67-th Seminaire
Lotharingien de Combinatoire in September 2011; joint session with XVII Incontro Italiano di Combinatoria Algebrica. Due to the length and to avoid making dull reading, the formalism is kept to the minimum, and references are given
step by step for those wishing to go into details and to understand the technical proofs of formal matters.

\section{The symbolic method of moments}
The method we are going to introduce in this section consists of shifting powers $a^i$ to indexes
$a_i$ by using a linear functional. The idea of dealing with subscripts as they were powers was extensively used by
the mathematical community since the nineteenth century, although with no rigorous foundation. The method, sometimes called
{\it Blissard's symbolic method}, is often attributed to J. Sylvester. He was the first who introduced the term {\it umbral\/} from latin to denote the shadowy techniques used to prove certain polynomial equations. Roman and Rota
have equipped these {\lq\lq tricks\rq\rq} with the right formalism in $1978,$ by using the theory of linear operators \cite{Roman}.
Despite the impressive amount of applications following this rigorous foundation (see references in \cite{Dinardo1}), in $1994$ Rota
proposed a new version of the method in a paper entitled \emph{``The classical umbral calculus{}"} \cite{SIAM}. The reason of this
rethinking is well summarized in the following quotation \cite{SIAM}: \lq\lq Although the notation of Hopf algebra
satisfied the most ardent advocate of spic-and-span rigor, the translation of {\it classical\/} umbral calculus into
the newly found rigorous language made the method altogether unwieldy and unmanageable. Not only was the eerie feeling
of witchcraft lost in the translation, but, after such a translation, the use of calculus to simplify computation and
sharpen our intuition was lost by the wayside.\rq\rq $\,$The term {\it classical\/} is strictly related to which sequence
is chosen as the {\it unity} of the calculus. Rota has always referred to the sequence $\{1\}$, and the same is done
in the version we are going to introduce, that from now on will be called {\it symbolic method of moments}.

Let ${\mathbb R}$ be the real field\footnote{The umbral calculus given in \cite{SIAM} considers a commutative integral domain whose quotient field is of characteristic zero.
For the applications given in the following, the real field  ${\mathbb R}$ is sufficient.}.
The symbolic method of moments consists of a set $\calA =\{\alpha, \gamma, \ldots\}$ of elements called {\it umbrae}, and
a linear functional $\E: {\mathbb R}[\calA] \rightarrow {\mathbb R},$ called \textit{evaluation},
such that $\E[1]=1$ and
\begin{equation}
\E[\alpha^{i} \gamma^{j} \cdots] = \E[\alpha ^{i}] E[\gamma^{j}] \cdots \,\, \hbox{(uncorrelation property)}
\label{(uncorrelation)}
\end{equation}
for non-negative integers $i,j,\ldots$.
\begin{definition}\cite{Dinardo1}
The sequence $\{1, a_1, a_2, \ldots \}\in {\mathbb R}$ is said to be {\it umbrally represented\/} by an umbra $\alpha$
if $\E[\alpha^i]=a_i$ for all positive integers~$i.$
\end{definition}
In the following, in place of $\{1, a_1, a_2, \ldots\}$ we will use the notation $\{a_i\}_{i \geq 0}$ assuming
$a_0=1.$ Special umbrae are given in Table~1.
\begin{savenotes}
\begin{table}[ht]
\caption{Special umbrae} 
\centering 
\begin{tabular}{l l} 
\hline\hline 
Umbrae & Moments for positive integers $i$ \\ [0.5ex] 
\hline 
Augmentation umbra & $\E[\varepsilon^i]=0$  \\ 
Unity umbra        & $\E[u^i]=1$  \\
Singleton umbra    & $\E[\chi^i]=\delta_{i,1},$ the Kronecker delta \\
Bell umbra         & $\E[\beta^i]={\mathfrak B}_i,$ the $i$-th {\it Bell number}\footnote{The $i$-th Bell number is the number of partitions of a finite non-empty set with $i$ elements or the $i$-th coefficient times $i!$ in the
Taylor series expansion of the function $\exp(e^t - 1).$}  \\
Bernoulli umbra    & $\E[\iota^i]=B_i,$ the $i$-th {\it Bernoulli number}\footnote{Many
characterizations of Bernoulli numbers can be found in the
literature, and each one may be used to define these numbers. Here we
refer to the sequence of numbers satisfying $B_0=1$ and
$\sum_{k=0}^i {\binom {i+1}  k} B_k = 0$ for $i=1,2,\ldots.$}  \\
Euler umbra        & $\E[\xi^i]={\mathfrak E}_i,$ the $i$-th {\it Euler number}\footnote{The Euler numbers are the coefficients of the formal power series $2 e^z/ [e^{2 z} + 1].$} \\ [1ex] 
\hline 
\end{tabular}
\label{table:uno} 
\end{table}
\end{savenotes}

Associating a measure with a sequence of numbers is familiar in probability theory, when the $i$-th term of a sequence can be considered as the $i$-th moment of a r.v., under suitable hypotheses (the so-called Hamburger moment problem \cite{Kendall}). As Rota underlines in {\it Problem 1: the algebra of probability} \cite{Rota}, all of probability theory could be done in terms of r.v.'s alone by taking an ordered commutative algebra over the reals, and endowing it with a positive linear functional.
Following this analogy, $a_i$ is called the {\it$i$-th moment\/} of the umbra $\alpha.$
\begin{definition} \label{refrv1}
A r.v.\ $X$ is said to be {\it represented\/} by an umbra $\alpha$ if its sequence of moments $\{a_i\}_{i \geq 0}$ is umbrally represented
by $\alpha.$
\end{definition}
In order to avoid misunderstandings, the expectation of a r.v.\ $X$ will be denoted by $E$, and its $i$-th moment by $E[X^i].$
\begin{example} \label{representation}
If $P(X=0)=1,$ the r.v.\ $X$ is represented by the augmentation umbra $\varepsilon.$ If $P(X=1)=1,$
the r.v.\ $X$ is represented by the unity umbra $u.$ The Poisson
r.v.\
${\rm Po}(1)$ of parameter $1$ is represented by the Bell umbra $\beta.$
More examples of umbrae representing classical r.v.'s can be found in
\cite{DiNardoOliva}.
\end{example}

Not all r.v.'s can be represented by umbrae. For example, the Cauchy r.v.\ does not
admit moments. Moreover, not all umbrae represent r.v.'s.
For example, the singleton umbra in Table~\ref{table:uno} does not represent a r.v.\ since its variance is negative
$\E[\chi^2]-\E[\chi]^2=-1,$ even though this umbra will play a fundamental role in the whole theory.
If $\{a_i\}_{i \geq 0}, \{g_i\}_{i \geq 0}$ and $\{b_i\}_{i \geq 0}$ are sequences umbrally represented by the umbrae
$\alpha, \gamma$ and $\zeta$, respectively, then the sequence $\{h_i\}_{i \geq 0}$ with
$$h_i=\sum_{k_1 + k_2 + \cdots + k_m = i} \binom {i }{ k_1, k_2, \ldots, k_m}   a_{k_{1}} \, g_{k_{2}} \cdots b_{k_{m}}$$
is umbrally represented by $\underbrace{\alpha + \gamma + \cdots + \zeta}_m,$ that is,
\begin{equation}
\E[(\underbrace{\alpha + \gamma + \cdots + \zeta}_m)^i]= \sum_{k_1+k_2+\cdots+k_m=i} \binom {i }{ k_1, k_2, \ldots, k_m}
a_{k_{1}} \, g_{k_{2}} \cdots b_{k_{m}}.
\label{(sum1)}
\end{equation}
The right-hand side of Equation~\eqref{(sum1)} represents the $i$-th moment of a summation of independent r.v.'s.

The second fundamental device of the symbolic method of moments is to represent the same sequence of
moments by different umbrae.
\begin{definition} \label{similar}
Two umbrae $\alpha$ and $\gamma$ are said to be {\it similar} if and only if $\E[\alpha^i]=\E[\gamma^i]$ for all positive integers $i.$
In symbols $\alpha \equiv \gamma.$
\end{definition}
If we replace the set $\{\alpha, \gamma, \ldots, \zeta\}$ by a set of $m$ distinct and similar umbrae $\{\alpha, \alpha^{\prime}, \ldots, \alpha^{\prime \prime}\},$ Equation~\eqref{(sum1)} gives
\begin{equation}
\E[(\underbrace{\alpha + \alpha^{\prime} + \cdots + \alpha^{\prime \prime}}_m)^i]= \sum_{\lambda \vdash i}
(m)_{\nu_{\lambda}}
d_{\lambda} a_{\lambda},
\label{(sum2)}
\end{equation}
where the summation is over all partitions\footnote{Recall that a partition of an integer $i$ is a sequence
$\lambda=(\lambda_1,\lambda_2,\ldots,\lambda_t)$, where
$\lambda_j$ are weakly decreasing integers and $\sum_{j=1}^t
\lambda_j = i$. The integers $\lambda_j$ are called {\it parts} of
$\lambda$. The {\it length} of $\lambda$ is the number of its
parts and will be indicated by $\nu_{\lambda}$. A different
notation is $\lambda=(1^{r_1},2^{r_2},\ldots)$, where $r_j$ is the
number of parts of $\lambda$ equal to $j$ and $r_1 + r_2 + \cdots
= \nu_{\lambda}$. We use the classical notation $\lambda \vdash
i$ to denote that \lq\lq $\lambda$ is a partition of $i$\rq\rq.}
 $\lambda$ of the integer $i,$
$a_{\lambda} = a_1^{r_1} \, a_{2}^{r_2} \cdots$ and
$$d_{\lambda} = \frac{i!}{(1!)^{r_1} (2!)^{r_2} \cdots r_1!\, r_2! \cdots}$$
is the number of $i$-set partitions
with block sizes given by the parts of $\lambda$. The new symbol $m {\boldsymbol .}  \alpha$ denotes the summation $\alpha + \alpha^{\prime} + \cdots + \alpha^{\prime \prime}$ and is called {\it dot-product\/} of $m$ and $\alpha.$ This new symbol is referred to as {\it auxiliary umbra}, in order to underline that it is not in $\calA.$  So Equation~\eqref{(sum2)} may be rewritten as
\begin{equation}
\E[(m {\boldsymbol .} \alpha)^i]= \sum_{\lambda \vdash i} (m)_{\nu_{\lambda}} d_{\lambda} a_{\lambda}.
\label{(sum33)}
\end{equation}
The introduction of new symbols is not only a way of lightening notation, but also to represent
sequences of moments obtained by suitable generalizations of the right-hand side of Equation (\ref{(sum33)}).
This will become clearer as we move along further. According to Definition~\ref{refrv1}, the auxiliary umbra $m {\boldsymbol .}  \alpha$ represents a summation of $m$ independent and identically distributed (i.i.d.) r.v.'s. Indeed, identically distributed r.v.'s share the same sequence of moments, if they have.

If we denote by $\calX$ the alphabet of  auxiliary umbrae, the evaluation $\E$ may be defined using the alphabet $\calA \cup \calX$, and so we can deal with auxiliary umbrae as they were elements of the base alphabet \cite{SIAM}. In the following, we still refer to the alphabet ${\calA}$ as the overall alphabet of umbrae, assuming that auxiliary umbrae are included.

Summations of i.i.d.r.v.'s appear in many sample statistics. So the first application we give in statistics concerns sampling distributions.
\subsection{Moments of sampling distributions}
It is recognized that an appropriate choice of language and notation can simplify and clarify many statistical calculations~\cite{McCullagh}. In managing algebraic expressions such as the variance of sample means, or, more generally, moments of sampling distributions, the main difficulty is the manual computation. Symbolic computation has removed many of such difficulties.
The need of efficient algorithms has increased the attention on how to speed up the computational time.
Many of these algorithms rely on the algebra of symmetric functions~\cite{Andrews}: $U$-statistics are an example
which arise naturally when looking for minimum-variance unbiased estimators. A $U$-statistic of a random sample has the form
\begin{equation}
U=\frac{1}{(n)_k}\sum \Phi(X_{j_1}, X_{j_2}, \ldots, X_{j_k}),
\label{Ustat}
\end{equation}
where $X_1, X_2, \ldots, X_n$ are $n$ independent r.v.'s, and the sum ranges
over the set of all permutations $(j_1,j_2,\ldots,j_k)$ of $k$ integers chosen in $[n]:=\{1,2,\dots,n\}.$
If $X_1,X_2, \ldots, X_n$ are identically distributed r.v.'s with common cumulative distribution
function $F(x),$ $U$ is an unbiased estimator of the population
parameter
$$\Theta(F)=\int \int \cdots \int \Phi(x_1, x_2,\ldots,x_k) \, {\rm d} F(x_1)\,  {\rm d}F(x_2)\cdots \, {\rm d}F(x_k).$$
In this case, the function $\Phi$ may be assumed to be a symmetric function of its arguments. Often, in applications, $\Phi$ is a
polynomial in the $X_j$'s, so that $U$-statistics are symmetric polynomials. By virtue of the fundamental theorem of
symmetric polynomials, $U$-statistics can be expressed in terms of elementary symmetric polynomials. The symbolic method
proposed here provides a way to find this expression \cite{CompStat}.

The starting point is to replace ${\mathbb R}$ by ${\mathbb R}[x_1, x_2, \ldots , x_n],$
where $x_1, x_2,$  $\ldots , x_n$ are indeterminates. This device allows us to deal with multivariable umbral polynomials $p \in {\mathbb R}[x_1, x_2, \ldots , x_n].$ The uncorrelation property \eqref{(uncorrelation)} is then updated to
$$
\E[x_i^{k_1} x_j^{k_2} \cdots \alpha^{k_3} \gamma^{k_4} \cdots] = x_i^{k_1} x_j^{k_2} \cdots  \E[\alpha^{k_3}] \E[\gamma^{k_4}] \cdots
$$
for any set of distinct umbrae in $\calA,$
for $i, j \in [n],$ with $i \ne j,$ and for all positive integers $k_1, k_2, k_3, k_4.$

For umbral polynomials $p,q \in {\mathbb R}[x_1, x_2, \ldots, x_n][\calA],$
we may relax the similarity equivalence and introduce another umbral
equivalence:
$$p \simeq q \quad \hbox{if and only if} \quad \E[p]=\E[q].$$
A polynomial sequence $\{p_i\}_{i \geq 0} \in {\mathbb R}[x_1, x_2, \ldots , x_n],$ with $p_0=1$ and
de\-gree$(p_i)=i,$ is represented by an umbra $\psi$, called {\it polynomial umbra\/},
if $\E[\psi^i]=p_i$ for all positive integers $i.$ The classical bases of the algebra of
symmetric polynomials can be represented by polynomial umbrae \cite{Bernoulli}. For example, the $i$-th
elementary symmetric polynomial $e_i(x_1, x_2, \ldots, x_n)$ satisfies
\begin{equation}
e_i(x_1, x_2, \ldots, x_n) \simeq \frac{(\chi_1 x_1 + \chi_2 x_2 + \cdots + \chi_n x_n)^i}{i!}, \quad \hbox{for} \,\, i \leq n,
\label{(sim1)}
\end{equation}
with $\chi_1,\chi_2, \ldots, \chi_n$ uncorrelated singleton umbrae. If the indeterminates $x_1, x_2, \ldots , x_n$
are replaced by $n$ uncorrelated umbrae $\alpha_1, \alpha_2,$ $\ldots ,\break \alpha_n$ similar to an umbra $\alpha,$ then Equivalence~\eqref{(sim1)} becomes
\begin{equation}
e_i(\alpha_1, \alpha_2, \ldots, \alpha_n) \simeq \frac{[n {\boldsymbol .}
 (\chi \alpha)]^i}{i!}, \quad \hbox{for} \,\, i \leq n.
\label{(sim1bis)}
\end{equation}
The umbral polynomial $e_i(\alpha_1, \alpha_2, \ldots, \alpha_n)$ is called
the {\it $i$-th umbral elementary symmetric polynomial}. The umbrae $\{\alpha_1, \alpha_2, \ldots, \alpha_n\}$ in \eqref{(sim1bis)} may also be replaced by some powers such as $\{\alpha_1^j, \alpha_2^j, \ldots, \alpha_n^j\}.$ Let us consider the monomial symmetric polynomial
$${\mathfrak m}_{\lambda}(x_1, x_2, \ldots, x_n) = \sum x_1^{\lambda_1} x_2^{\lambda_2} \cdots x_n^{\lambda_n}$$
with $\lambda \vdash i \leq n,$ where the notation is symbolic and must be interpreted in the way that all different images of the monomial $x_1^{\lambda_1} x_2^{\lambda_2} \cdots x_n^{\lambda_n}$ under permutations of the variables have to be summed.  If the indeterminates
$x_1, x_2, \ldots , x_n$ are replaced by $n$ uncorrelated umbrae $\alpha_1, \alpha_2, \ldots , \alpha_n$
similar to an umbra $\alpha,$ then
\begin{equation}
{\mathfrak m}_{\lambda}(\alpha_1, \alpha_2,
\ldots, \alpha_n) \simeq \frac{[n {\boldsymbol .} (\chi \alpha)]^{r_1}}{r_1!} \frac{[n {\boldsymbol .}
(\chi \alpha^2)]^{r_2}}{r_2!} \cdots,
\label{(moneq)}
\end{equation}
which is a product of umbral elementary polynomials in \eqref{(sim1bis)}.
In many statistical calculations, the so-called augmented monomial symmetric polynomials \cite{Kendall}
are involved:
\begin{multline*}
\tilde{{\mathfrak m}}_{\lambda}(x_1, x_2, \ldots, x_n)
\\
=
\sum_{j_1 \ne \ldots \ne j_{r_1} \ne
j_{r_1+1} \ne \ldots \ne j_{r_1+r_2} \ne \cdots } \!\!\! x_{j_1} \cdots
x_{j_{r_1}} \, x^2_{j_{r_1+1}} \cdots \, x^2_{j_{r_1+r_2}} \cdots.
\end{multline*}
These polynomials are obtained from ${\mathfrak m}_{\lambda}(x_1, x_2, \ldots, x_n)$ as follows:
\begin{equation}
{\mathfrak m}_{\lambda}(x_1, x_2,\ldots, x_n) = \frac{\tilde{{\mathfrak m}}_{\lambda}(x_1, x_2, \ldots, x_n)}{r_1! r_2! \cdots}.
\label{(monaug)}
\end{equation}
Then, from Equivalence~\eqref{(moneq)} and Equation~\eqref{(monaug)}, we have
\begin{equation}
\tilde{{\mathfrak m}}_{\lambda}(\alpha_1, \alpha_2, \ldots, \alpha_n) \simeq [n {\boldsymbol .} (\chi \alpha)]^{r_1} [n {\boldsymbol .}(\chi \alpha^2)]^{r_2} \cdots,
\label{(sim2)}
\end{equation}
which is again a product of umbral elementary polynomials \eqref{(sim1bis)}.
\begin{example}
{\rm If $\lambda =(1^2,3)$ and $n \geq 5$ then
\begin{align*}
& [n {\boldsymbol .} (\chi \alpha)]^{2} [n  {\boldsymbol .}  (\chi
 \alpha^3)] = (\chi_1 \alpha_1 + \cdots + \chi_n \alpha_n)^2
(\chi_1 \alpha_1^3 + \cdots + \chi_n \alpha_n^3) \\
&=  \left( \sum_{1 \leq j_1 \ne j_2 \leq n} \chi_{j_1}
\alpha_{j_1} \chi_{j_2} \alpha_{j_2} \right) \left( \sum_{1 \leq
j_3  \leq n} \chi_{j_3} \alpha_{j_3}^3 \right)
\simeq \sum_{1 \leq
j_1 \ne j_2 \ne j_3 \leq n} \alpha_{j_1} \alpha_{j_2}
\alpha_{j_3}^3.
\end{align*}
This last equivalence follows by observing that $\E[\chi_{j_1} \chi_{j_2} \chi_{j_3}]$ vanishes whenever there is at least one pair
of equal indices among $\{j_1, j_2, j_3\}.$}
\end{example}
More details on symmetric polynomials and polynomial umbrae are given in \cite{Bernoulli}. Due to Equivalence~\eqref{(sim2)}, the fundamental expectation result, at the core of unbiased estimation of moments, may be restated as follows.
\begin{theorem} \label{ttt} \cite{Bernoulli} If $\lambda= (1^{r_1}, 2^{r_2}, \ldots) \vdash i \leq n,$ then
$$\E\left\{[n {\boldsymbol .} (\chi \alpha)]^{r_1} [n {\boldsymbol .} (\chi \alpha^2)]^{r_2} \cdots \right\} = (n)_{\nu_{\lambda}} a_{\lambda}.$$
\end{theorem}
From Equation~\eqref{(sum33)} and Theorem~\ref{ttt}, the next corollary follows.
\begin{corollary} \label{3.3} If $i \leq n$, then
\begin{equation}
(m {\boldsymbol .} \alpha)^i \simeq \sum_{\lambda \vdash i} \frac{(m)_{\nu_{\lambda}}}{(n)_{\nu_{\lambda}}} \, d_{\lambda}\, [n {\boldsymbol .} (\chi \alpha)]^{r_1} [n {\boldsymbol .} (\chi \alpha^2)]^{r_2} \cdots.
\label{(sum3)}
\end{equation}
\end{corollary}
From Equivalence \eqref{(sum3)}, if $m=n$, then
\begin{equation}
(n {\boldsymbol .} \alpha)^{i} \simeq \sum_{\lambda \vdash i} d_{\lambda} [n {\boldsymbol .} (\chi \alpha)]^{r_1} [n {\boldsymbol .} (\chi \alpha^2)]^{r_2} \cdots.
\label{(single)}
\end{equation}
Theorem~\ref{ttt} discloses a more general result:
products of moments are moments of products of umbral elementary symmetric polynomials \eqref{(sim1bis)}.

By analogy with Equation~\eqref{Ustat}, the umbral symmetric polynomial on the right-hand side of Equivalence \eqref{(sum3)},
$$
\frac{1}{(n)_{\nu_{\lambda}}} [n {\boldsymbol .} (\chi \alpha)]^{r_1} [n {\boldsymbol .} (\chi \alpha^2)]^{r_2} \cdots,
$$
is called the $U$-{\it statistic} of uncorrelated and similar umbrae
$\alpha_1,\alpha_2,\!\ldots,\break\alpha_n.$ Then Theorem~\ref{ttt} says that $U$-statistics can be expressed as
products of umbral elementary symmetric polynomials, and Corollary~\ref{3.3} says that any statistic involving
a summation of i.i.d.r.v.'s can be expressed as a linear combination of $U$-statistics.

A more general problem consists in computing population moments of sample statistics. These
are symmetric functions of not independent r.v.'s. Within the symbolic method of moments,
the dependence between r.v.'s corresponds to working with umbral monomials with
non-disjoint supports. Indeed, the support of an umbral polynomial $p$ is the set of all umbrae occurring
in $p.$ If $\mu$ and $\nu$ are umbral monomials  with non-disjoint
supports, then $\E[\mu^i \, \nu^j] \ne \E[\mu^i] \E[\nu^j]$ for all positive
integers $i$ and $j.$ In the following, unless otherwise specified,
we work with umbral monomials with non-disjoint supports and
give an example of how to use these monomials to perform computations with not independent r.v.'s.
\begin{example}\cite{CompStat} \label{3.4}
Let us consider a bivariate random sample of r.v.'s $(X_1,Y_1)$, \ldots, $(X_n, Y_n),$ that is,
$(X_1, \ldots, X_n)$ is a random sample with parent distribution $X$ not necessarily
independent from the parent distribution $Y$ of the random sample $(Y_1, \ldots, Y_n).$
In order to compute the expectation of
\begin{equation}
\left(\sum_{i=1}^n X_i\right)^2 \left(\sum_{i=1}^n Y_i\right)  \,\, \hbox{or} \,\,  \left(\sum_{i \ne j}^n X_i^2 X_j\right) \left(\sum_{i=1}^n X_i^2 Y_i\right)^2
\label{(vrbick1)}
\end{equation}
by the symbolic method of moments, the main tool is to represent the sample statistics in (\ref{(vrbick1)}) as a suitable linear combination
of umbral monomials $\mu_1$ and $\mu_2$ representing the populations  $X$ and $Y$, respectively. For the former sample statistic \eqref{(vrbick1)}, we have
$$E\left[\left(\sum_{i=1}^n X_i\right)^2\left(\sum_{i=1}^n Y_i\right)\right]
 =  \E[(n {\boldsymbol .} \mu_1)^2 \,\, (n {\boldsymbol .} \mu_2)].$$
For the latter, let us observe that $E[(\sum_{i\ne j}^n X_{i}^2 X_{j})] = \E[(n {\boldsymbol .} \chi \mu_1^2)
(n {\boldsymbol .} \chi \mu_1)]$ and $E[(\sum_{i=1}^n X_{i}^2 Y_i)] = \E [n {\boldsymbol .}(\chi \mu_1^2$ $\mu_2)],$
so that the overall expectation of the product is given by
\begin{equation}
\E[n {\boldsymbol .} (\chi_1 \mu_1^2) \, n {\boldsymbol .}(\chi_1 \mu_1) \, n {\boldsymbol .}(\chi_2 \mu_1^2 \mu_2) \, n {\boldsymbol .} (\chi_3 \mu_1^2 \mu_2)].
\label{(overall)}
\end{equation}
Both evaluations can be performed by using a suitable generalization of Theorem~\ref{ttt} involving multisets of umbral monomials.
\end{example}
Let us recall that a multiset of umbral monomials $\{\mu_1, \mu_2, \ldots, \mu_k\}$ is
\begin{equation}
M = \{\underbrace{\mu_1, \ldots, \mu_1}_{f(\mu_1)}, \underbrace{\mu_2, \ldots, \mu_2}_{f(\mu_2)}, \ldots, \underbrace{\mu_k, \ldots,
\mu_k}_{f(\mu_k)}\}.
\label{(typemult)}
\end{equation}
The length of the multiset $M$ in \eqref{(typemult)} is $|M|=f(\mu_1) + f(\mu_2) + \cdots \break + f(\mu_k).$
The {\it support\/} of the multiset $M$ is $\bar{M} = \{\mu_1,\dots,\mu_k\}.$
A {\it subdivision} of a multiset $M$ is a multiset $S$ of $l$ non-empty submultisets $M_i=(\bar{M}_i, h_i)$ of $M$ with $l \leq |M|$ that satisfy \cite{Bernoulli}
\begin{enumerate}
\item[{\it a)}] $\bigcup_{i=1}^l \bar{M}_i = \bar{M};$
\item[{\it b)}] $\sum_{i=1}^l h_i(\mu) = f(\mu)$ for every $\mu \in \bar{M}.$
\end{enumerate}
Roughly speaking, a subdivision is a multiset which plays the same role for $M$ as a partition for a set.
A subdivision of the multiset $M$ in \eqref{(typemult)} is denoted by
$$
S=\{\underbrace{M_1 \ldots, M_1}_{g(M_1)}, \underbrace{M_2,\ldots,M_2}_{g(M_2)}, \ldots,
\underbrace{M_l,\ldots,M_l}_{g(M_l)}\}.
$$
\begin{example}
For the multiset $M=\{\mu_1, \mu_1, \mu_2\},$ a subdivision is $S=\{\{\mu_1\},
\{\mu_1,\mu_2\}\}$ or  $S=\{\{\mu_1\},\{\mu_1\},\{\mu_2\}\}.$ In the latter case, we write
$S = \{\underbrace{\{\mu_1\}}_{2}, \underbrace{\{\mu_2\}}_{1}\}.$
\end{example}
A {\tt MAPLE} algorithm
to find all multiset subdivisions is described in \cite{Statcomp}. This procedure is faster
than the iterated full partitions of Andrews and Stafford \cite{Andrews}, given that it takes
into account the multiplicity of all elements of $M.$ The higher this multiplicity is,
the more the procedure is efficient. The number of set partitions of $[|M|]$ corresponding to the same subdivision $S$
is
\begin{equation}
d_S=\frac{1}{g(M_1)! \, g(M_2)! \cdots g(M_l)!}\prod_{i=1}^k \frac{f(\mu_i)!}{h_{1}(\mu_i)! \, h_{2}(\mu_i)! \cdots h_{l}(\mu_i)!},
\label{(ds)}
\end{equation}
where $h_{j}(\mu_i)$ is the multiplicity of $\mu_i$ in the submultiset $M_j.$

\begin{example}
For $M=\{\mu_1, \mu_1, \mu_2\}$ let us consider the subdivision $S=\{\{\mu_1\}, \{\mu_1,\mu_2\}\}.$
If we label the elements of $M=\{\mu_1, \mu_1, \mu_2\}$ by the elements
of $\{1,2,3\},$ the partitions of $\{1,2,3\}$ corresponding to $S=\{\{\mu_1\}, \{\mu_1,\mu_2\}\}$ are
$1|23$ and $2|13,$ so $d_S=2.$ The same result arises from computing \eqref{(ds)}, since for $M_1=\{\mu_1\}$ and $M_2=\{\mu_1,\mu_2\}$ we have
$g(M_1)=1$, $g(M_2)=1$ and $h_{1}(\mu_1)=1$, $h_{2}(\mu_1)=h_{2}(\mu_2)=1.$ Moreover
$f(\mu_1)=2$ and $f(\mu_2)=1.$
\end{example}
If we set $(n {\boldsymbol .} \mu)_M = \prod_i (n {\boldsymbol .} \mu_i)^{f(\mu_i)},$
then, from Equivalence \eqref{(single)}, we obtain
\begin{equation}
(n {\boldsymbol .} \mu)_M  \simeq  \sum_{S} d_S \, [n {\boldsymbol .} (\chi\mu)]_{S} ,
\label{(mult)}
\end{equation}
where the summation is over all subdivisions $S$ of the multiset $M,$ the integer $d_S$ is given in
\eqref{(ds)}, and
\begin{equation}
[n {\boldsymbol .} (\chi \mu)]_{S} =  \prod_{M_i \in \bar{S}}
[n {\boldsymbol .} (\chi \mu_{M_i})]^{\,g(M_i)} \quad \hbox{with} \quad \mu_{M} =  \prod_{\mu \in \bar{M}} \mu^{f(\mu)}.
\label{(notation)}
\end{equation}

Equivalence~\eqref{(mult)} allows us to generalize Theorem~\ref{ttt} to umbral monomials with non-disjoint supports.
\begin{theorem} \cite{Bernoulli} \label{ttt1}
We have
$$\E\left([n {\boldsymbol .} (\chi \mu)]_S \right) = (n)_{|S|} \prod_{j=1}^l m_{h_j(\mu_1), h_j(\mu_2),\ldots, h_j(\mu_k)}^{g(M_j)},$$
where
$$m_{h_j(\mu_1), h_j(\mu_2),\ldots, h_j(\mu_k)} = \E\left[\mu_1^{h_j(\mu_1)} \mu_2^{h_j(\mu_2)} \cdots \mu_k^{h_j(\mu_k)}\right] = \E[\mu_{M_j}].$$
\end{theorem}
\begin{exampleno}{\bf 2.8 (continued).}
{\rm In order to compute $\E[(n {\boldsymbol .}\mu_1)^2 (n {\boldsymbol .} \mu_2)],$ we use Equivalence~\eqref{(mult)}.
For $M=\{\mu_1, \mu_1, \mu_2\}$ we have
\begin{multline*}
(n {\boldsymbol .}\mu_1)^2 (n {\boldsymbol .} \mu_2)  \simeq    [n \punt (\chi \mu_1^2 \mu_2)] +
2 \, [n \punt (\chi \mu_1)] \, [n \punt (\chi \mu_1 \mu_2)]\\
+ [n \punt (\chi \mu_1^2)]
   [n \punt (\chi \mu_2)] + [n \punt (\chi \mu_1)]^2 [n \punt (\chi \mu_2)],
\end{multline*}
and, from Theorem~\ref{ttt1} with $S=\{\{\mu_1,\mu_1\},\{\mu_2\} \}$,
we get
$$\E[(n {\boldsymbol .}\mu_1)^2 (n {\boldsymbol .} \mu_2)]= n \, m_{2,1} + 2 (n)_2 m_{1,0}  m_{1,1} + (n)_2  m_{2,0}  m_{0,1} +  (n)_3  m_{1,0}^2 m_{0,1},$$
with $m_{i,j} = E[X^i \,Y^j] = \E[\mu_1^i \,\mu_2^j].$
In order to compute \eqref{(overall)}, Equivalence~\eqref{(mult)} and Theorem~\ref{ttt1} have to be applied to
$$M=\{ \chi_1 \mu_1^2, \chi_1 \mu_1, \chi_2 \mu_1^2 \mu_2, \chi_3 \mu_1^2 \mu_2\}.$$
Since $M$ is a set, the subdivisions $S$ are just set partitions. In difference to the strategy proposed by
\cite{Vrbick} to solve the same problem, the singleton umbrae in $M$ allow us
to speed up the computation. For example, consider the partition
$\pi_1=\left\{ \{ \chi_1 \mu_1^2, \chi_1 \mu_1\}, \{\chi_2 \mu_1^2 \mu_2,
\chi_3 \mu_1^2 \mu_2\}\right\}$.
We have
$$[n {\boldsymbol .} (\chi \mu)]_{\pi_1} = [n {\boldsymbol .} (\chi \, \chi_1^2 \, \mu_1^3)] \, [n {\boldsymbol .} (\chi \, \chi_2 \chi_3 \, \mu_1^4 \mu_2^2)].$$
Since $\chi \,, \chi_2 \,, \chi_3$ are uncorrelated umbrae, we have
$\E[\chi \, \chi_2 \chi_3 \, \mu_1^4 \mu_2^2] = \E[\mu_1^4 \mu_2^2],$
but $\E[\chi \, \chi_1^2 \, \mu_1^3] = 0$
since $\E[\chi_1^2]=0.$ Therefore we have
$$\E[n {\boldsymbol .} (\chi \mu)]_{\pi_1}]=0.$$
This is the trick to speed up the computation: whenever, in one --- or more --- blocks of the subdivision
$S$, there are at least two umbral monomials involving
correlated singleton umbrae, the auxiliary umbra $[n {\boldsymbol .} (\chi
\mu)]_{S}$ has evaluation equal to zero. In the iteration procedure employed to build the subdivision $S,$ this allows us
to discard these blocks of the subdivision in the subsequent steps, see \cite{CompStat} for more details.}
\end{exampleno}

Further applications of Theorems~\ref{ttt} and \ref{ttt1}
can be found in \cite{CompStat}. See Table~$3$ of Appendix~$1$ for a
comparison of computation times with procedures existing in the literature.
%
\subsection{Sheppard's corrections}
In the real world, continuous variables are observed and recorded
in finite precision through a rounding or coarsening operation, i.e.,
a grouping rule. A compromise between the desire to know and the cost of knowing
is then a necessary consequence. The literature on grouped data
spans different research areas, see for example
\cite{Heitjan}. Due to the relevance of the method of moments as a tool
to learn about a distribution, great attention has been paid in the
literature to the computation of moments when data are grouped into classes.
The moments computed by means of the resulting grouped frequency
distribution are a first approximation to the estimated moments of
the parent distribution, but they suffer from the error committed
during the process of grouping.
The correction for grouping is a sum of two terms,
the first depending on the length of the grouping interval, the
second being a periodic function of the position. For a continuous
r.v., this very old problem was first discussed by
Thiele \cite{Thiele}, who studied the second term without taking
into consideration the first,
and then by Sheppard \cite{Sheppard} who studied
the first term, without taking into consideration the second.
Compelling reasons have been given in the literature
for neglecting the second term and thus for using Sheppard's corrections, that are nowadays
still employed. Quite recently the appropriateness of Sheppard's
corrections was re-examined in connection with some applications,
see, for example, \cite{Dempster} and \cite{Vardeman}.

Let us consider a r.v.\ $X$ with probability density function $f,$
representing the continuous parent distribution underlying data.
Denote by $a_i$ its $i$-th moment (also called {\it raw moment\/} to distinguish it
from the estimated one),
\begin{equation}
a_i = \int_{-\infty}^{\infty} x^i f(x) \,{\rm d}x.
\label{(momc)}
\end{equation}
Let $\tilde{a}_i$ be the $i$-th moment of the grouped distribution,
\begin{equation}
\tilde{a}_i = \frac{1}{h} \int_{-\infty}^{\infty} x^i \int_{- {\frac 1
2} h}^{{\frac 1 2} h} f(x + z) \,{\rm d}z \,{\rm d}x,
\label{(infinite)}
\end{equation}
with $h$ the length of the grouping intervals. Then the $i$-th raw moment
$a_i$ can be reconstructed via Sheppard's corrections as
\begin{equation}
a_i = \sum_{j=0}^{i} {\binom i j} \left( 2^ {1-j} -1 \right) B_j
\, h^j \, \tilde{a}_{i - j}, \label{(Sh1)}
\end{equation}
where $B_j$ is the $j$-th Bernoulli number.

The derivation of Sheppard's corrections was a popular topic in the
first half of the last century, see \cite{Hald} for a historical account.
This is because Sheppard deduced Equation~\eqref{(Sh1)} by using the
Euler--Maclaurin summation formula and by assuming that the density function had high order contact with
the $x$-axis at both ends. So there was a considerable controversy on
the set of sufficient conditions to be required in order to use formula \eqref{(Sh1)}.
If the rounding lattice is assumed to be random, these sufficient conditions
can be removed.

Grouping includes also censoring or splitting data into categories
during collection or publication, and so it does not only involve
continuous variables. The derivation of corrections for raw moments
of a discrete parent distribution followed a different path. They were first given in the Editorial of
Vol.~$1$, no.~$1$, of {\it Annals of Mathematical Statistics}
(page~111).
The method used to develop a general formula was extremely
laborious. Some years later, Craig \cite{Craig} considerably
reduced and simplified the derivation of these corrections without requiring any other
condition and stating these corrections on the average.  Except for the papers of Craig \cite{Craig} and Baten \cite{Baten},
no attention has been paid to multivariate generalizations of Sheppard's
corrections, probably due to the complexity of the resulting formulas.

The symbolic method has allowed us to find such corrections both
for continuous and discrete parent distributions in closed-form formulas,
which have been naturally extended to multivariate parent distributions \cite{DinardoShep}. The result
is an algorithm to implement all these formulas in any symbolic package \cite{DinardoShepMAPLE}.
A similar employment of the symbolic method can be performed within wavelet theory.
Indeed, the reconstruction of the full moment information $a_i$ in \eqref{(Sh1)} through the grouped moments
$\tilde{a}_i$ can be seen as some kind of multilevel analysis, see \cite{Saliani, Shen}.

If the sequence $\{\tilde{a}_i\}_{i \geq 0}$ in \eqref{(infinite)} is umbrally
represented by the umbra $\tilde{\alpha}$ and the sequence $\{a_i\}_{i \geq 0}$
in \eqref{(momc)} is umbrally represented by the umbra $\alpha,$
Sheppard's corrections \eqref{(Sh1)} can be obtained by expanding the $i$-th moment of
\begin{equation}
\alpha \equiv \tilde{\alpha} + h \left( \iota + \frac{1}{2} \right),
\label{(shc)}
\end{equation}
where $\iota$ is the Bernoulli umbra. Equivalence~\eqref{(shc)} still holds, if the moments refer to
a parent distribution over a finite interval.  When $a_i$ is the $i$-th moment of a discrete
population, and $\tilde{a}_i$ is the $i$-th moment calculated from the grouped frequencies,
Equivalence~\eqref{(shc)} needs to be modified to
\begin{equation}
\alpha \equiv \tilde{\alpha} + h \left( \iota + {\frac 1 2} \right) + {\frac h m} \left( -1 {\boldsymbol .} \iota - {\frac 1 2} \right),
\label{(iii)}
\end{equation}
where we assume that $m$ consecutive values of the discrete population are grouped
in a frequency class of width $h.$ In Equivalence~\eqref{(iii)}, $-1 {\boldsymbol .} \iota$ denotes one more auxiliary
umbra, the inverse umbra of the Bernoulli umbra.\footnote{Given two
umbrae $\alpha$ and $\gamma,$ they are said to be \textit{inverse}
to each other if $\alpha+\gamma\equiv\varepsilon.$ We denote the
inverse of the umbra $\alpha$ by $-1 {\boldsymbol .}  \alpha.$ An umbra and its inverse are uncorrelated.
The inverse of an umbra is not unique, but any two umbrae inverse to any given umbra are similar.} Sheppard's corrections
\eqref{(shc)} can be recovered from \eqref{(iii)} letting $m \rightarrow \infty$ in a suitable way.
\par
As done in Section~2.1, the generalization of Equivalences \eqref{(shc)} and \eqref{(iii)} to the multivariate setting follows by introducing suitable umbral monomials.

Let $(\mu_1,\mu_2, \ldots,\mu_k)$ be a $k$-tuple of umbral monomials with non-disjoint supports.
\begin{definition} \label{4.1}
The sequence $\{m_{t_1, \, t_2,\ldots, \, t_k}\}$ is umbrally represented by the $k$-tuple $(\mu_1, \mu_2,  \ldots,\mu_k)$
if $E[\mu_1^{t_1} \mu_2^{t_2} \cdots \mu_k^{t_k}] = m_{t_1, \, t_2, \ldots,\, t_k}$ for all positive integers $t_1, t_2, \ldots, t_k.$
\end{definition}
Following the analogy with random vectors, $m_{t_1, \, t_2, \ldots, \, t_k}$ is called the multivariate moment of the
$k$-tuple $(\mu_1, \mu_2,  \ldots,\mu_k)$  of order ${\boldsymbol t}=\left(t_1, \, t_2, \right.$ $\left.\ldots, \, t_k\right).$

According to the latter notation in \eqref{(notation)}, if $M$ is a multiset with support $\bar{M}=\{\mu_1,\mu_2, \ldots,\mu_k\}$ and
multiplicities $f(\mu_j)=t_j$ for $j=1,2,\ldots,k,$ then we may write $E[\mu_M]=m_{t_1, \, t_2, \ldots, \, t_k}.$
Moreover, assume there exists a multivariate random vector ${\boldsymbol X}=(X_1, X_2, \ldots, X_k)$
with joint density  $f_{\boldsymbol X}({\boldsymbol x})$ over ${\mathbb R}^k$ such that
\begin{equation}
m_{t_1, \, t_2, \ldots, \, t_k} = \int_{{\mathbb R}^k} {\boldsymbol x}^{\boldsymbol t} f_{\boldsymbol X}({\boldsymbol x}) \, {\rm d}{\boldsymbol x}
\label{(mulmom)}
\end{equation}
is its multivariate moment of order ${\boldsymbol t}.$ Note that, as in the univariate case, without loss of generality
we may consider joint densities with
range\footnote{The range of a joint (probability) density (function) $f_{\boldsymbol X}({\boldsymbol x})$ is the subset of ${\mathbb R}^k$ where $f_{\boldsymbol X}({\boldsymbol x})$ is different from $0.$} of bounded rectangle type.
The moments calculated from the grouped frequencies are
\begin{equation}
\tilde{m}_{t_1,\, t_2, \ldots, \, t_k} = \frac{1}{h_1 h_2 \cdots h_k} \int_{{A}_k} \int_{{\mathbb R}^k}
\prod_{j=1}^k (x_j + z_j)^{t_j} f_{\boldsymbol X}({\boldsymbol x}) \, {\rm d}{\boldsymbol x} \, {\rm d}{\boldsymbol z},
\label{(mulgr)}
\end{equation}
where
\begin{multline*}
A_k = \left\{ {\boldsymbol z}= (z_1, z_2, \ldots, z_k) \in {\mathbb R}^k
:\right.\\
\left.
 z_j \in \left(-{\frac 1 2} h_j, {\frac 1 2} h_j \right)
\text{ for all }j \in
\{1,2, \ldots,k\} \right\},
\end{multline*}
and $h_j \in {\mathbb R} - \{0\}$ are the window widths for any component. Then we have the following theorem.
\begin{theorem}
If the sequence $\{m_{t_1, \, t_2, \ldots, \, t_k}\}$ in \eqref{(mulmom)} is umbrally represented
by the $k$-tuple $(\mu_1, \mu_2,  \ldots,\mu_k)$ and the sequence $\{\tilde{m}_{t_1, \, t_2, \ldots, \, t_k}\}$ in
\eqref{(mulgr)} is umbrally represented by the $k$-tuple $(\tilde{\mu}_1, \tilde{\mu}_2,  \ldots, \tilde{\mu}_k),$ then
\begin{equation}
\mu_M \equiv  \left[\tilde{\mu} + h \left( \iota + \frac{1}{2} \right) \right]_M,
\label{(shcmul)}
\end{equation}
where $M$ is the multiset given in \eqref{(typemult)},
$$\left[\tilde{\mu} + h \left( \iota + \frac{1}{2} \right) \right]_M = \prod_{j=1}^k \left[ \tilde{\mu}_j + h_j \left( \iota_j
+ \frac{1}{2} \right) \right]^{t_j},$$
and $\{\iota_j\}$ are uncorrelated Bernoulli umbrae.
\end{theorem}
Equivalence~\eqref{(shcmul)} is implemented in \cite{DinardoShepMAPLE} by means of the following steps. In order to recover expressions of raw multivariate moments in terms of grouped moments, we need to multiply summations like
$$\sum_{s_j=1}^{i} \binom {i }{s_j} \tilde{\mu}_j^{s_j} h_j^{i - s_j} \left( 2^{1-i+s_j} - 1\right) B_{i-s_j},$$
corresponding to the $i$-th power of $\tilde{\mu}_j + h_j \left( \iota_j
+ \frac{1}{2} \right),$ and then to replace occurrences of products like $\tilde{\mu}_1^{s_1}
\tilde{\mu}_2^{s_2} \cdots \tilde{\mu}_k^{s_k}$ by $\tilde{m}_{s_1, \, s_2,
  \ldots, \, s_k}.$
If the multivariate parent distribution is discrete, Equivalence \eqref{(shcmul)} has to be updated to
$$
\mu_M \equiv \left[ \tilde{\mu} + h \left( \iota + {\frac 1 2} \right) + {\frac h m} \left( -1 {\boldsymbol .} \iota
- {\frac 1 2} \right) \right]_M
$$
where the symbol on the right-hand side denotes the product
$$ \prod_{j = 1}^k  \left[ \tilde{\mu}_j + h_j \left( \iota_j + {\frac 1 2} \right) + \frac {h_j }{ m_j} \left( -1 {\boldsymbol .} \iota_j
- {\frac 1 2} \right) \right]^{t_j},$$ and $m_j$ denotes the number of consecutive values grouped in a frequency class of width $h_j.$
\section{Compound Poisson processes}
In $1973$, Rota \cite{VIII} claimed that compound Poisson processes
are related to polynomial sequences $\{p_i(x)\}_{i \geq 0}$ of binomial type.
Such a sequence of polynomials is defined by $p_i(x)$ being
of degree $i$ for all positive integers $i,$ $p_0(x)=1$, and
\begin{equation}
p_i(x+y) = \sum_{j=0}^i  {\binom i j} p_j(x) p_{i-j}(y),
\quad \quad i=0,1,\dots.
\label{(binomprop)}
\end{equation}
This connection has ramifications into several other areas of analysis and probability: Lagrange expansions, renewal theory,
exponential family distributions and infinitely divisible processes.
Due to its relevance, different authors have tried to formalize
this connection. Stam \cite{Stam} studied polynomial sequences
of binomial type in terms of an integer-valued compound Poisson process.
He was especially interested in the asymptotic behavior
of the probability generating function $p_i(x)/p_i(1)$ for $i \rightarrow \infty.$
Partial results are available on the radius of convergence under suitable conditions,
involving the power series with coefficients $\{p_i(x)\}_{i \geq 0}$. The resulting theory relies on a complicated system
of notations.   Constantine and Savit \cite{Costantine} have derived a
generalization of Dobi\'nsky's formula by means of compound Poisson processes.
Pitman \cite{Pitman} has investigated some probabilistic aspects of Dobi\'nsky's formula.

The theory of Bell umbrae gives a natural way to relate compound Poisson processes
to polynomial sequences of binomial type. Historically, this connection was the
first application in probability of classical umbral calculus as proposed in $1994.$
More details are given in \cite{Dinardo1}.

So let us come back to the Bell umbra $\beta,$ introduced in Section~2. Instead of using
Bell numbers, we may characterize this umbra by using its factorial moments. The
factorial moments of an umbra are
\begin{equation}
a_{(i)} = \E[(\alpha)_i] = \begin{cases} 1, & \mbox{if } i = 0, \\ E[\alpha (\alpha - 1) \cdots (\alpha - i + 1)], & \mbox{if } i > 0.
\end{cases}
\label{(momfact)}
\end{equation}
\begin{theorem} \cite{Dinardo1}
The Bell umbra $\beta$ is the unique umbra (up to similarity) with all
factorial moments equal to $1.$
\end{theorem}
In Equation~\eqref{(sum33)}, set $\E[(m  {\boldsymbol .} \alpha)^i]=q_i(m)$ and observe that $q_i(m)$ is a polynomial of degree $i$ in $m.$ Suppose that we replace $m$ by an indeterminate\footnote{We replace ${\mathbb R}$ by ${\mathbb R}[x,y].$
The indeterminate $y$ is added because of~\eqref{(binomprop)}.} $x.$
The symbol having sequence $\{q_i(x)\}_{i \geq 0}$ as moments is denoted by $x {\boldsymbol .} \alpha$ and  called
the dot-product of $x$ and $\alpha,$ that is,
\begin{equation}
q_i(x) = \E[(x {\boldsymbol .} \alpha)^i] = \sum_{\lambda \vdash i} (x)_{\nu_{\lambda}} d_{\lambda} a_{\lambda}
\label{(3lala)}
\end{equation}
for all positive integers $i.$  A Poisson r.v.\ $\hbox{Po}(x)$ is represented by $x {\boldsymbol .} \beta,$ which
is called the {\it Bell polynomial umbra}.

By using the exponential (partial) Bell polynomials $ B_{i,k}$ (cf.\ \cite{Riordan}),
a different way to write $q_i(x)$ in \eqref{(3lala)} is
\begin{equation}
q_i(x) = \E\left[(x {\boldsymbol .} \alpha)^i \right] = \sum_{k=1}^i (x)_{k} B_{i,k}(a_1, a_2, \ldots, a_{i-k+1}).
\label{(4lala)}
\end{equation}
Equation~\eqref{(4lala)} with $\alpha$ replaced by the Bell umbra
$\beta$ returns the so-called {\it exponential polynomials}
$\Phi_i(x)$,
\begin{equation}
\E\left[(x {\boldsymbol .} \beta)^i \right] = \Phi_i(x) = \sum_{k=1}^i (x)_k S(i,k),
\label{(xpart)}
\end{equation}
where $S(i,k)$ are Stirling numbers of the second kind. According
to \cite{Berndt}, these polynomials were introduced by S. Ramanujan in his
unpublished notebooks. Later, exponential polynomials were studied by Bell  \cite{Bell}
and Touchard \cite{Touchard}.
Rota, Kahaner and Odlyzko \cite{VIII} have stated their basic properties via umbral operators. If, in~\eqref{(xpart)}, we replace $x$ by
a non-negative integer $n,$ the resulting umbra $n {\boldsymbol .} \beta$ is the sum of $n$ similar and uncorrelated Bell umbrae,
likewise in probability theory where a Poisson r.v.\ $\hbox{Po}(n)$ is the sum of $n$ i.i.d.\ Poisson r.v.'s $\hbox{Po}(1).$ More
generally, the closure under convolution of Poisson distributions $F_t$, that is, $F_s \star F_t = F_{s+t},$
is simply reformulated as\footnote{According to property {\it iv)} of Corollary 1 in \cite{Dinardo1}, the umbra $(t+s) {\boldsymbol .} \beta$ is similar to $t {\boldsymbol .} \beta + s {\boldsymbol .} \beta^{\prime},$ with $\beta$ and $\beta^{\prime}$ uncorrelated Bell umbrae. If $s \ne t$ then $s {\boldsymbol .} \beta$ and $t {\boldsymbol .} \beta$ are two distinct symbols, and so uncorrelated. Then we can use $s {\boldsymbol .} \beta$ instead of  $s {\boldsymbol .} \beta^{\prime}.$}
$$
(t+s) {\boldsymbol .} \beta \equiv t {\boldsymbol .} \beta+s {\boldsymbol .} \beta
$$
with $t,s \in {\mathbb R}.$ If, in~\eqref{(xpart)}, we replace $x$ by a generic umbra $\alpha,$
the auxiliary umbra $\alpha {\boldsymbol .} \beta$ satisfies
$$
(\alpha {\boldsymbol .} \beta)^i \simeq \Phi_i(\alpha) = \sum_{k=1}^i (\alpha)_k S(i,k).
$$
The umbra $\alpha {\boldsymbol .} \beta$ represents a random sum of independent Poisson r.v.'s
$\hbox{Po}(1)$ indexed by an integer r.v.\ $Y$ represented by $\alpha,$ that is, a randomized Poisson r.v.\ $\hbox{Po}(Y)$ with parameter $Y.$

If we swap the umbrae $\alpha$ and $\beta$, the resulting
umbra is the {\it $\alpha$-partition umbra} $\beta {\boldsymbol .} \alpha$
representing a compound Poisson r.v.\ with parameter $1.$ Let us recall that
a compound Poisson r.v.\ with parameter $1$ is a random sum $S_N = X_1 + X_2 + \cdots + X_N,$
where $N \sim \hbox{Po}(1)$ and $\{X_j\}$ are i.i.d.r.v.'s. In the $\alpha$-partition umbra  $\beta {\boldsymbol .} \alpha,$ $\beta$
represents $N$ and $\alpha$ represents one of $\{X_j\}.$ Moreover, since the Poisson r.v.\
$\hbox{Po}(x)$ is umbrally represented by the Bell polynomial umbra $x {\boldsymbol .} \beta,$
a compound Poisson r.v.\ with parameter $x$ is represented by the
{\it polynomial $\alpha$-partition} umbra $x {\boldsymbol .} \beta {\boldsymbol .}
\alpha.$ This encodes the connection between compound Poisson processes and polynomial
sequences of binomial type, which is what we have claimed at the beginning of this section.
The name ``partition umbra" has a probabilistic background.
Indeed, the parameter of a Poisson r.v.\ is usually denoted
by $x=\lambda t,$ with $t$ representing time. When the (time) interval in which $t$
ranges is partitioned into non-overlapping ones, their contributions
are stochastically independent and add up to $S_N.$ This circumstance
is umbrally expressed by the relation
\begin{equation}
(x+y) {\boldsymbol .} \beta {\boldsymbol .} \alpha \equiv x {\boldsymbol .} \beta {\boldsymbol .} \alpha + y {\boldsymbol .} \beta {\boldsymbol .} \alpha,
\label{(eq:somma)}
\end{equation}
paralleling the property~\eqref{(binomprop)} for the binomial sequences
$p_i(x) =\break \E[(x {\boldsymbol .} \beta {\boldsymbol .} \alpha)^i].$
If $\{a_i\}_{i \geq 0}$ is umbrally represented by the umbra $\alpha,$ moments of $x {\boldsymbol .} \beta {\boldsymbol .} \alpha$ are
\begin{equation}
\E[(x {\boldsymbol .} \beta {\boldsymbol .} \alpha)^i] = \sum_{k=1}^i x^k B_{i,k}(a_1,a_2,\ldots,
a_{i-k+1})= \sum_{\lambda \vdash i} d_{\lambda} \, x^{\nu_{\lambda}} a_{\lambda},
\label{(alfapartition)}
\end{equation}
where $B_{i,k}$ are the partial  Bell exponential polynomials. For $x=1$, we infer from~\eqref{(alfapartition)} that
\begin{equation}
\E[(\beta {\boldsymbol .} \alpha)^i] = \sum_{k=1}^i B_{i,k}(a_1,a_2,\ldots,
a_{i-k+1})=  {\calY}_i(a_1,a_2,\ldots,a_i),
\label{(comp)}
\end{equation}
where ${\calY}_i = {\calY}_i(a_1,a_2,\ldots,a_i)$ is the $i$-th complete Bell exponential polynomial \cite{Riordan}.
Then the $\alpha$-partition umbra represents the binomial sequence $\{{\calY}_i\}.$

The following example proves that partition umbrae are employed in representing r.v.'s with distribution different from the Poisson distribution.
\begin{example} \label{gaus} {\rm A Gaussian r.v.\ ${\mathcal N}(m,s)$ is represented by the umbra
$m + \beta  {\boldsymbol .} (s \eta),$ where $\eta$ is an umbra
satisfying $E[\eta^i]= \delta_{i,2},$ for all positive integers $i.$
Indeed, the $i$-th moment of $m + \beta  {\boldsymbol .} (s \eta)$ is
$$
\E\big\{[m + \beta  {\boldsymbol .} (s \eta)]^i\big\}=
\sum_{k=0}^{\lfloor {i / 2} \rfloor}\bigg(\frac {s^2 }{ 2}\bigg)^k \frac{(i)_{2k}}{k!} m^{i-2k},
$$
which gives the $i$-th element of a well-known sequence of orthogonal polynomials, the \emph{Hermite polynomials} $H_{i}^{(\nu)}(x)$ \cite{Riordan}, with $m=x$ and $s^{2}=-\nu.$}
\end{example}
\begin {theorem} \cite{Dinardo1} \label{recur}
The $\alpha$-partition umbra satisfies the equivalence
\begin{equation}
(\beta {\boldsymbol .} \alpha)^{i} \simeq \alpha \, (\beta {\boldsymbol .} \alpha+\alpha)^{i-1},
\quad i=1,2,\ldots,
\label{(prpart)}
\end{equation}
and conversely every umbra satisfying Equivalence~\eqref{(prpart)} is
an $\alpha$-part\-it\-ion umbra.
\end{theorem}

One more property of the $\alpha$-partition umbra involves {\it disjoint sums}.
The disjoint sum $\alpha \dot{{}+{}} \gamma$ of $\alpha$ and $\gamma$ is an auxiliary umbra representing the sequence
$\{a_i + g_i\}_{i \geq 0}$ with $\{a_i\}_{i \geq 0}$ and $\{g_i\}_{i \geq 0}$ umbrally represented by the umbra $\alpha$  and
$\gamma$, respectively.
\begin{theorem} \label{labdisj} \cite{Dinardoeurop} $\beta{\boldsymbol .} (\alpha \dot{{}+{}} \gamma) \equiv \beta {\boldsymbol .} \alpha  + \beta {\boldsymbol .} \gamma.$
\end{theorem}

If in the $\alpha$-partition umbra $\beta {\boldsymbol .} \alpha,$ we replace the umbra $\beta$ by a different
umbra $\gamma,$ we get a new auxiliary umbra, the dot-product $\gamma {\boldsymbol .} \alpha.$ More formally, let us consider
the polynomial $q_i(x)$ in~\eqref{(3lala)}, and suppose to replace $x$ by an umbra $\gamma$
with moments $\{g_i\}_{i \geq 0}.$ The polynomial
\begin{equation}
q_i(\gamma) =\sum_{\lambda \vdash i} (\gamma)_{\nu_{\lambda}} d_{\lambda} a_{\lambda}
\label{dotproduct1}
\end{equation}
is an umbral polynomial with support $\{\gamma\}.$ The auxiliary umbra $\gamma {\boldsymbol .} \alpha$ is an auxiliary umbra
with moments
\begin{equation}
\E[(\gamma {\boldsymbol .} \alpha)^i] = \E[q_i(\gamma)],
\label{dotproduct2}
\end{equation}
for all positive integers $i.$ A very special dot-product umbra is $\chi {\boldsymbol .} \alpha,$ with $\chi$ the singleton umbra.
This umbra is the keystone in dealing with sequences of cumulants. More details will be given
in Section~4. Indeed we will show  that any umbra is the partition umbra of $\chi {\boldsymbol .} \alpha,$
that is, the symbolic method of moments is like a calculus of measures on Poisson algebras.
In the following, we give an example of how to use the umbra $\chi {\boldsymbol .}
\alpha$ to disclose a not well-known connection between  partition polynomials and Kailath--Segall polynomials.
\begin{example}[\sc The Kailath--Segall formula] \label{6.3}

Let $\{X_t\}_{t \geq 0}$ be a centered L\'evy process\footnote{A centered L\'evy process is a zero mean stochastic process with independent and stationary increments~\cite{Sato}, see also Section~3.1.} with moments of all orders. The variations of the process $\left\{X_t^{(i)}\right\}_{t \geq 0}$
are defined as $X_t^{(1)} = X_t$ and
$$
X_t^{(2)} =
[X,X]_t = \sigma^2 t + \sum_{0 < s \leq t}
(\Delta X_s)^2, \, X_t^{(i)} = \sum_{0 < s \leq t}
(\Delta X_s)^i,
\;\; \hbox{for} \,\, i \geq 3,
$$
where $\Delta X_s=X_s - X_{s{\text{-}}}.$ The iterated stochastic integrals
$$
P_t^{(0)} = 1, \,\,\, P_t^{(1)}=X_t, \,\,\, P_t^{(i)}=\int_0^t P_{s\text{-}}^{(i-1)} {\rm d}X_s, \;\; \hbox{for} \,\,  i \geq 2,
$$
are related to the variations $\left\{X_t^{(i)}\right\}_{t \geq 0}$ by the
Kailath--Segall for-\break mula~\cite{KS}
\begin{equation}
P_t^{(i)} = \frac{1}{i} \left( P_t^{(i-1)} X_t^{(1)} - P_t^{(i-2)} X_t^{(2)} + \cdots +
(-1)^{i+1} P_t^{(0)} X_t^{(i)} \right).
\label{(KS)}
\end{equation}
It follows by induction that
$P_t^{(i)}$ is a polynomial in $X_t^{(1)}, X_t^{(2)}, \ldots, X_t^{(i)},$ called the $i$-th \emph{Kailath--Segall polynomial}.
\begin{theorem} \label{Annali} \cite{AnnaliImma}
If $\{\Upsilon_t\}_{t \geq 0}$ and $\{\psi_t\}_{t \geq 0}$ are two families of umbrae with
$\E[\Upsilon_t^i]= i! E\left[P_t^{(i)}\right]$  and $\E[\psi_t^i] = E\left[X_t^{(i)} \right],$ for all positive integers~$i,$ then
\begin{equation}
\Upsilon_t \equiv \beta {\boldsymbol .} [(\chi {\boldsymbol .} \chi) \psi_t] \quad \hbox{and} \quad (\chi {\boldsymbol .} \chi) \psi_t \equiv \chi {\boldsymbol .} \Upsilon_t.
\label{(KStheor)}
\end{equation}
\end{theorem}
Equivalences~\eqref{(KStheor)} give the umbral version of the Kailath--Segall
formula (the former) and its inversion (the latter).
The proof essentially relies on Theorem~\ref{recur}, since
Equation~\eqref{(KS)} is equivalent to $\E[\Upsilon_t^i] = \E\left[\psi_t (\Upsilon_t + \psi_t)^{i - 1}\right].$
By the former Equivalence~\eqref{(KStheor)}, the Kailath--Segall polynomials result in
the complete Bell exponential polynomials~\eqref{(comp)}
in the variables
$$\left\{(-1)^{i-1} (i-1)! E\left[X_t^{(i)}\right]\right\}.$$
The latter similarity in Equivalence~\eqref{(KStheor)} is a generalization of
$$\chi {\boldsymbol .} (\chi_1 x_1 + \cdots + \chi_n x_n) \equiv (\chi {\boldsymbol .} \chi) \sigma,$$
which gives the elementary symmetric polynomials in terms of power sum symmetric polynomials, umbrally represented by the umbra $\sigma,$
see \cite{Bernoulli}. That is, if we replace the jumps $\{\Delta X_s\}$ in $X_t^{(i)}$ by suitable
indeterminates $\{x_s\},$ then the Kailath--Segall polynomials reduce to the polynomials given in \cite{Taqqu1}.
These polynomials deserve further investigations in connection with stochastic integrals \cite{TaqquPeccati, RotaWallstrom}, for which we
believe that an approach via the symbolic method of moments is fruitful.
\end{example}
Theorem \ref{Annali} gives the symbolic representation of L\'evy process variations. The next subsection is devoted to the symbolic representation of
L\'evy processes. This representation extends the claimed connection between binomial sequences and compound Poisson processes to the more general class of infinite divisible distributions.

\subsection{Time-space harmonic polynomials}
L\'evy processes are stochastic processes playing a fundamental role in various fields such as physics, engineering and actuarial
science \cite{BN}. They are particularly popular in mathematical finance as models for jumps or spikes of asset prices in financial markets.

The martingale property is one of the fundamental mathematical properties which underlies many
finance models, when no knowledge of the past price process can help to predict the future asset price.
In particular, a market is said to be efficient if the price process is modeled by a martingale
(for more details see \cite{Schoutens}).

L\'evy processes are not martingales if they have an increasing or decreasing drift.
One tool to compensate this fault is to use
Wald's martingale, which has various applications \cite{Lyuu}.
If $\{X_t\}_{t \geq 0} = \{X_t\}$ is a L\'evy process,\footnote{When no misunderstanding is possible, we simply use $\{X_t\}$ to denote a stochastic process.}  then the Wald exponential martingale is
defined as
\begin{equation}
M_t =\frac{\exp\{z X_t\}}{E[\exp\{z X_t\}]}.
\label{expmart}
\end{equation}
Definition~\eqref{expmart} requires the existence of the moment generating function
$E[\exp\{z X_t\}]$ in a suitable neighborhood of the origin. A way to overcome this
gap is to look for polynomials $P(x,t),$ with the property that,
when the indeterminate $x$ is
replaced by a L\'evy process $\{X_t\},$ the resulting stochastic process is a martingale.
These polynomials have been introduced first by Neveu \cite{neveu} for
random walks and then by Sengupta \cite{GS95} for L\'evy processes. They are
called {\it time-space harmonic polynomials}.
\begin{definition}\label{TSH}
A family of polynomials $\{P(x,t)\}_{t \geq 0}$ is called {\it time-space harmonic} with
respect to a stochastic process $\{X_t\}_{t \geq 0}$ if
$$E[P(X_t,t) \,\, |\; \mathfrak{F}_{s}] =P(X_s,s),$$ for $0 \leq s \leq t,$
where $\mathfrak{F}_{s}=\sigma\left( X_\tau : 0 \leq \tau \leq s\right)$
is the natural filtration\footnote{A natural filtration
$\mathfrak{F}_{t}$ is the $\sigma$-algebra generated by the
pre-images $X_s^{-1}(B)$ for Borel subsets $B$ of ${\mathbb R}$
and times $s$ with $0 \leq s \leq t.$} associated with $\{X_t\}_{t \geq 0}.$
\end{definition}
Wald's martingale \eqref{expmart} has been recently used
in order to characterize time-space harmonic polynomials \cite{Sengupta08}, but without reaching a closed expression.

The aim of this subsection is to show how to recover a very general class of time-space harmonic polynomials
by using a symbolic representation of L\'evy processes. In particular,
our attention focuses on special properties of L\'evy processes, namely \cite{Sato}
\begin{enumerate}
\item[{\it a)}] $X_0 = 0$ a.s. (almost surely);
\item[{\it b)}] for all $n \geq 1$ and for all $0 \leq t_1 \leq t_2 \leq \ldots \leq t_n < \infty,$ the r.v.'s
$X_{t_2} - X_{t_1}, X_{t_3} - X_{t_2}, \ldots, X_{t_{n}} - X_{t_{n - 1}}$ are independent;
\item[{\it c)}] for $0 \leq s \leq t,$ $X_{t + s} - X_s$ and $X_t$ are identically distributed (stationarity).
\end{enumerate}
Due to these properties, at time $t$ the information of the process $\{X_t\}$ may be split into i.i.d.\ subprocesses:
\begin{equation}
X_t \stackrel{d}{=} \underbrace{\Delta X_{t/n} + \Delta X_{t/n} + \cdots + \Delta X_{t/n}}_n.
\label{(additive)}
\end{equation}
Property \eqref{(additive)} characterizes the so-called {\it infinitely divisible r.v.'s}, which are
in one-to-one correspondence with  L\'evy processes. This is also why L\'evy himself refers to these processes as a sub-class of additive processes. Starting from \eqref{(sum33)},
a symbolic version of a L\'evy process can be obtained by means of the same steps as employed in the construction of $q_i(x)$
at the beginning of Section~3. If the positive integer $m$ is replaced by $t \in {\mathbb R}$  in \eqref{(sum33)}, then the auxiliary umbra
$t  {\boldsymbol .} \alpha$ satisfies
\begin{equation}
\E[(t {\boldsymbol .} \alpha)^i] = q_i(t) = \sum_{\lambda \vdash i} (t)_{\nu_{\lambda}} \, d_{\lambda} \, a_{\lambda},
\label{(momentlevy)}
\end{equation}
for all positive integers $i$. The umbra $t {\boldsymbol .} \alpha$ is the {\it dot-product\/} of $t$ and $\alpha.$
Since $m {\boldsymbol .} \alpha$ is a summation of $m$ uncorrelated and similar umbrae, $t {\boldsymbol .} \alpha$
parallels the decomposition \eqref{(additive)}. Its properties are the same as the ones of L\'evy processes when
they admit finite moments. Here there are some examples.

{\it Binomial property.} The $i$-th moment $q_i(t) = E(X_t^i)$ of a
L\'evy process is a polynomial function of $t$ which satisfies
the property \eqref{(binomprop)} of binomial sequences.  In terms of umbrae,
this property is equivalent to
\begin{equation}
(t+s) {\boldsymbol .} \alpha \equiv t {\boldsymbol .} \alpha + s {\boldsymbol .}\alpha.
\label{(binomimomentumbral)}
\end{equation}
We cannot miss the analogy between Equivalences~\eqref{(eq:somma)} and \eqref{(binomimomentumbral)}.
Indeed, an umbra representing a L\'evy process is similar to an umbra representing a compound Poisson process.
This is because for any umbra $\alpha$ there exists an umbra $\gamma$ satisfying $\alpha \equiv \beta {\boldsymbol .}
\gamma,$ so that $ t {\boldsymbol .} \alpha \equiv t {\boldsymbol .} \beta {\boldsymbol .} \gamma.$ The proof
of this statement is given in Section~4.

{\it Additivity property.} If $\{W_t\}$ and $\{Z_t\}$ are two independent L\'evy processes, then the process $\{X_t\}$ with $X_t = W_t + Z_t$
is a L\'evy process. In terms of umbrae, the additivity property corresponds to
$$
t {\boldsymbol .} (\alpha + \gamma) \equiv t {\boldsymbol .} \alpha + t {\boldsymbol .} \gamma, \quad
\hbox{with $\alpha,\gamma \in \calA.$}
$$

{\it Homogeneity property.} If $\{X_t\}$ is a L\'evy process and $c \in {\mathbb R},$ then $\{c \, X_t\}$ is a L\'evy process.
This homogeneity property is equivalent to
$$
t {\boldsymbol .} (c \alpha) \equiv c (t {\boldsymbol .} \alpha), \quad \hbox{with $\alpha \in \calA.$}
$$

{\it Nested processes.} If $\{X_t\}$ is a L\'evy process, then
$\{(X_t)_s\}$ is a L\'evy process, that is, a nested L\'evy process is
again
a L\'evy process. The analog for $t {\boldsymbol .} \alpha$ is $t
{\boldsymbol .} (s {\boldsymbol .} \alpha) \equiv s {\boldsymbol .} (t
{\boldsymbol .} \alpha) \equiv (st) {\boldsymbol .} \alpha,$
with $t,s \in {\mathbb R}.$

According to Definition~\ref{TSH}, in order to verify that an umbral polynomial is a time-space harmonic polynomial,
a suitable notion of conditional evaluation has been introduced in the polynomial
ring ${\mathbb R}[x][\calA],$ see \cite{AnnaliImma}. Denote by ${\mathcal X}$ the set ${\mathcal X} = \{\alpha\}.$
\begin{definition}\label{condeval1}
The linear operator $\E(\;\cdot\; \vline \,\, \alpha):\, {\mathbb R}[x][\calA]
\;\longrightarrow\; {\mathbb R}[{\mathcal X}]$
satisfying
\begin{itemize}
\item[{\it i)}] $\E(1 \,\, \vline \,\,\alpha)=1$;
\item[{\it ii)}] $\E(x^m \alpha^n \gamma^i \varsigma^j\cdots \, \,  \vline \,\, \alpha)=x^m \alpha^n
\E(\gamma^i)\E(\varsigma^j)\cdots$ for uncorrelated umbrae
$\alpha, \gamma, \varsigma, \ldots$ and for positive integers $m,n,i,j,\ldots$
\end{itemize}
is called \emph{conditional evaluation} with respect to $\alpha.$
\end{definition}

In other words,  the conditional evaluation with respect to $\alpha$
deals with the umbra $\alpha$ as an indeterminate.  As it happens for r.v.'s, the conditional evaluation is an element of
${\mathbb R}[x][\calA]$, and the overall evaluation of $\E(p\,\, \vline \,\, \alpha)$ gives $\E(p),$ with
$p \in {\mathbb R}[x][\calA].$

The conditional evaluation needs to be carefully handled when dot products are involved. Let us observe that the conditional evaluation
with respect to $n \punt \alpha$ satisfies
$$\E[(n+1) \punt \alpha \, | \, n \punt \alpha] = \E ( n \punt \alpha + \alpha^{\prime}  \, | \, n \punt \alpha) = n \punt \alpha + \E(\alpha^{\prime}),$$
with $\alpha^{\prime}$ an umbra similar to $\alpha.$ By similar arguments, for all positive integers $n$ and $m,$ since
$\E\{ [(n + m) \punt \alpha]^i  \, | \, n \punt \alpha\} = \E \{ [ n \punt \alpha + m \punt \alpha^{\prime}]^i \, | \, n \punt \alpha \},$  we have
\begin{equation}
\E \{ [(n + m) \punt \alpha]^i  \, | \, n \punt \alpha\} = \sum_{j=0}^i \binom{i}{j} (n \punt \alpha)^j \E[(m \punt \alpha^{\prime})^{i-j}].
\label{(condeval)}
\end{equation}
Equation (\ref{(condeval)}) suggests how to define the conditional evaluation with respect to the auxiliary umbra $s \punt \alpha.$
\begin{definition}
The conditional evaluation of $t \punt \alpha$ with respect to the auxiliary umbra $s \punt \alpha$ is
$$\E[ (t \punt \alpha)^i  \, | \, s \punt \alpha] = \sum_{j=0}^i \binom{i}{j} (s \punt \alpha)^j \E\{[(t-s) \punt \alpha^{\prime}]^{i-j}\}.$$
\end{definition}

\begin{theorem} \cite{DinardoShep} \label{UTSH2}
For all non-negative integers $i,$ the family of polynomials
\begin{equation}
Q_i(x,t)=\E[(- t {\boldsymbol .} \alpha + x)^i] \in {\mathbb R}[x]
\label{(tshumbral)}
\end{equation}
is time-space harmonic with respect to $\{t {\boldsymbol .} \alpha\}_{t \geq 0},$ that is,
\begin{equation}
\E\left[Q_i(t {\boldsymbol .} \alpha,t) \, \, \vline \, \, s {\boldsymbol .} \alpha \right]
=  Q_i(s {\boldsymbol .} \alpha,s) \qquad \hbox{for}\,\, s,t \geq 0.
\label{(TSHformula)}
\end{equation}
\end{theorem}
The umbra $-t  {\boldsymbol .} \alpha$ in~\eqref{(tshumbral)} is a special auxiliary umbra  with the remarkable property
\begin{equation}
-t {\boldsymbol .} \alpha + t {\boldsymbol .} \alpha^{\prime} \equiv \varepsilon,
\label{(inverse)}
\end{equation}
where $\alpha^{\prime} \equiv \alpha.$ Due to Equivalence~\eqref{(inverse)}, the umbra $-t {\boldsymbol .} \alpha$ is
the {\it inverse}\footnote{Since $-t {\boldsymbol .} \alpha$ and $t
  {\boldsymbol .} \alpha$ are two distinct symbols, they can be
  considered uncorrelated. When no confusion arises, we will use the symbol
$t {\boldsymbol .} \alpha$ instead of $t   {\boldsymbol .} \alpha^{\prime}.$} umbra of $t {\boldsymbol .}
\alpha$. Its moments can be obtained via Equation~\eqref{(momentlevy)} by replacing $t$ by $-t.$
The family $\{Q_i(x,t)\}_{i \geq 0}$ is a basis for the
space of time-space harmonic polynomials \cite{commstat} so that many families of
known time-space harmonic polynomials can be recovered as a linear
combination of $\{Q_i(x,t)\}_{i \geq 0}.$ For example, Hermite polynomials  are
time-space harmonic with respect to Wiener processes, and Poisson--Charlier polynomials
are time-space harmonic with respect to Poisson processes. More examples are given
in Appendix~2.
\begin{remark}
Theorem~\ref{UTSH2} may be proved when the L\'evy process admits moments up to a finite order $m.$
In this case, the polynomials $Q_i(x,t)$ in Equation~\eqref{(tshumbral)} are defined up to $i \le m.$ Then formal
power series are replaced by polynomials of degree $m$, and operations like summation and
multiplication among formal power series can be
performed in terms of polynomials \cite{Taqqu}.
\end{remark}
\begin{example}[\sc Random walks] \label{3.13}
When the parameter $t$ is replaced by a positive integer $n,$ Theorem~\ref{UTSH2} still holds.
The sequence\break $\{n {\boldsymbol .} \alpha\}_{n\geq 0}$ represents a random walk
$\{X_n\}_{n \geq 0}$ with $X_0=0, X_n=M_1+M_2+\cdots+M_n$, and $\{M_n\}_{n \geq 0}$ a sequence of i.i.d.r.v.'s.
Note that the sequence $\{M_n\}_{n \geq 0}$ satisfies $M_0=X_0=0$ and $M_n=X_n-X_{n-1},$ for all positive integers $n.$
\end{example}

To recover the connection between Wald's martingale \eqref{expmart} and the polynomials $Q_i(x,t),$ we have to introduce the \emph{generating function} of an umbra $\alpha,$ that is, the formal power series \cite{Dinardoeurop}
$$
f(\alpha, z)= 1 + \sum_{i \geq 1} a_i \frac{z^i}{i!} \in {\mathbb R}[x][[z]],
$$
whose coefficients are the moments of the umbra $\alpha.$ Formal power series allow
us to work  with generating functions \cite{Wilf} which do not have a positive radius
of convergence or have indeterminate coefficients  \cite{Stanley}. Similar umbrae $\alpha \equiv \gamma$
have the same generating function $f(\alpha,z)=f(\gamma,z).$ This property allows us to recover the
generating functions of the auxiliary umbrae. Table~\ref{tablegenfun} gives some examples.
\begin{table}[ht]
\caption{Special generating functions} 
\centering 
\begin{tabular}{l l} 
\hline\hline 
Umbrae & Generating functions  \\ [0.5ex] 
\hline 
Unity umbra $u$                        & $f(u,z)=1$  \\
Singleton umbra $\chi$                 & $f(\chi,z)=1+z$  \\
Bell umbra  $\beta$                    & $f(\beta,z)= \exp[e^z - 1]$  \\
Summation $\alpha + \gamma$            & $f(\alpha+\gamma,z) = f(\alpha,z)f(\gamma,z)$ \\ [1ex]
Dot product $n {\boldsymbol .} \alpha$ & $f(n {\boldsymbol .} \alpha, z) = f(\alpha,z)^n$ \\ [1ex]
Dot product $t {\boldsymbol .} \alpha$ & $f(t {\boldsymbol .} \alpha, z) = f(\alpha, z)^t$ \\ [1ex]
Inverse $- t {\boldsymbol .} \alpha$   & $f(- t {\boldsymbol .} \alpha, z) = 1/f(\alpha, z)^t$ \\ [1ex]
\hline 
\end{tabular}
\label{tablegenfun} 
\end{table}

In particular, the generating function of the time-space harmonic polynomial umbra $-t {\boldsymbol .} \alpha + x$ is
\begin{equation}
f(- t {\boldsymbol .} \alpha +x , z)=\frac{\exp\{xz\}}{f(\alpha,z)^t}=1 + \sum_{i \geq 1} Q_i(x,t) \frac{z^i}{i!}.
\label{(wald)}
\end{equation}
By replacing $x$ by $t  {\boldsymbol .} \alpha$ in~\eqref{(wald)}, we can compare Equation~\eqref{(wald)} with
the following series expansion for Wald's exponential martingale~\eqref{expmart}
$$
\frac{\exp\{z X_t\}}{E[\exp\{z X_t\}]} = 1 + \sum_{i \geq 1}
R_i(X_t,t)\frac{z^i}{i!}.
$$

The algebraic structure of formal power series is isomorphic to sequences endowed with the convolution product, each series corresponding to the sequence of its coefficients \cite{Stanley}. Thus, equality of two formal
power series is interpreted as equality of their corresponding
coefficients, that is,
$E[R_i(X_t,t)] = \E[Q_i(t {\boldsymbol .} \alpha,t)].$

Since $\E[Q_i(t {\boldsymbol .} \alpha,t)]=0,$ we have $E[R_i(X_t,t)] = 0$ for $i \geq 1$  and
$$1 + \sum_{i \geq 1} E[R_i(X_t,t)] \frac{z^i}{i!}=1$$
in agreement with Wald's identity \cite{neveu}. Therefore, the sequence\break $\{E[R_i(X_t,t)]\}_{t \geq 0}$
is umbrally represented by the augmentation umbra $\varepsilon,$ with $f(\varepsilon,z)=1.$  But this
is exactly what happens if in $-t {\boldsymbol .} \alpha + x$ we replace $x$ by $t {\boldsymbol .} \alpha,$ due to~\eqref{(inverse)}.

Time-space harmonic polynomials  with respect to $\{t  {\boldsymbol .} \alpha\}_{t \geq 0}$ are characterized by the following theorem.

\begin{theorem} \cite{DinardoShep}
A polynomial $P(x,t)=\sum_{j=0}^k p_j(t) \, x^j,$ of degree $k$ for all $t \geq 0,$ is  time-space harmonic
with respect to a L\'evy process $\{X_t\},$  represented by $t {\boldsymbol .} \alpha,$ if and only if
$$p_j(t) = \sum_{i=j}^k  \binom{i}{j} \, p_i(0) \, \E[(-t {\boldsymbol .} \alpha)^{i-j}], \qquad \hbox{for}\,\,  j=0,\ldots,k.$$
\end{theorem}
An open problem is the problem of extending the symbolic method of moments to matrix-valued stochastic processes by using
L\'evy processes and time-space harmonic polynomials.
Matrix-value time-space harmonic polynomials include the matrix-valued counterparts of classical polynomials, as the Hermite, Laguerre or Jacobi polynomials, which have gained increasing interest in multivariate statistics. Different methods have been proposed in the literature, see \cite{Lawi} and references therein,
by using eigenvalues of random matrices or hypergeometric functions of matrix argument. It would be interesting to check  if
dealing with hypergeometric functions by means of the symbolic method leads to more feasible expressions as happens for the multivariate generating functions, see
Section~7.
\section{Cumulants}
A probability measure ${\mathbb P}$ is infinitely divisible if for any positive integer $n$ there is a
probability measure ${\mathbb P}_n$ such that $ {\mathbb P} = ({\mathbb P}_n)^{n*},$ where $({\mathbb P}_n)^{n*}$
denotes the $n$-fold convolution of ${\mathbb P}_n.$ This property parallels Equation~\eqref{(additive)} for L\'evy processes $\{X_t\}$ and can be generalized.
Indeed, if we denote by ${\mathbb P}_t$ the probability measure of $X_t$, then ${\mathbb P}_t=({\mathbb P}_1)^{t*}$
with ${\mathbb P}_1$ the probability measure of $X_1.$ Since the convolution of probability measures corresponds
to a summation of r.v.'s, the r.v.\ $X_t$ can be expressed as a summation of $t$ times the r.v.\ $X_1.$ The symbolic method
of moments allows us to generalize this decomposition by replacing $t \in {\mathbb R}$ by the umbra $\gamma.$
The resulting auxiliary umbra is the dot-product of $\alpha$ and $\gamma$ in \eqref{dotproduct1}, that represents
a summation of $\gamma$ times the umbra $\alpha.$ A special dot-product is the umbra $\chi \punt \alpha,$
where $\chi$ is the singleton umbra, since its moments are cumulants of the sequence $\{a_i\}_{i \geq 0}$
umbrally represented by $\alpha.$

Among the number sequences related to r.v.'s, cumulants play a
central role characterizing many r.v.'s occurring in classical
stochastic processes.  For example, a Poisson r.v.\ is the unique
probability distribution for which all cumulants are equal to
the parameter. A Gaussian r.v.\ is the unique probability distribution
for which all cumulants vanish beyond the second.

(Formal) cumulants $\{c_j\}$ of a sequence $\{a_i\}$ are usually defined by the identity of formal power series
\cite{Rota}
\begin{equation}
1 + \sum_{i \geq 1} a_i \frac{t^i}{i!} =
\exp{\left({\sum_{j \geq 1}  c_j \frac{t^j}{j!}}\right)},
\label{(cumRota)}
\end{equation}
which uniquely determines the non-linear functionals $\{c_j\}$ as
polynomials in $\{a_i\}.$ In this definition, any
questions concerning convergence of involved series may be
disregarded \cite{Stanley}. Moreover, many difficulties connected to the
so-called problem of cumulants\, smooth out. Here, with the problem of
cumulants, we refer to the characterization of sequences that are
cumulants of some probability distribution. The simplest example
is that the second cumulant of a probability distribution must
always be non-negative, and is zero if and only if all of the higher
cumulants are zero. Cumulants are not subject to such constraints
when they are analyzed from a symbolic point of view.
\begin{example}[\sc The variance]
The variance $\hbox{Var}(X) = E(X^2) - E(X)^2$ of a r.v.\ is the cumulant of second order, an index very useful in statistics.
It possesses remarkable properties that may be extended to cumulants of any order:
\begin{enumerate}
\item[{\it i)}] it is invariant under translation: $\hbox{Var}(X + a) = \hbox{Var}(X) $
for any constant $a;$
\item[{\it ii)}] $\hbox{Var}(X+Y) = \hbox{Var}(X) + \hbox{Var}(Y),$ if $X$ and $Y$ are independent r.v.'s;
\item[{\it iii)}]  $\hbox{Var}(X)$ is a polynomial in the moments of the r.v.\ $X$.
\end{enumerate}
\end{example}

A first approach to the theory of cumulants via the classical umbral calculus was
given by Rota and Shen in \cite{Shen2}. By using the symbolic method of moments, the formulas commonly used to express
moments of a r.v.\ in terms of cumulants, and vice-versa,
are encoded in a dot-product involving the singleton umbra \cite{Dinardoeurop}.
\begin{definition}\label{(cumulant)}
The $\alpha$-cumulant umbra $\kappa_{{\scriptscriptstyle \alpha}}$
is the auxiliary umbra satisfying $\kappa_{{\scriptscriptstyle \alpha}} \equiv \chi {\boldsymbol .} \alpha.$
\end{definition}
The three main algebraic properties of cumulants can be stated as follows:

\begin{enumerate}
\item[{\it i)}] {\it additivity property}:
$$\chi {\boldsymbol .} (\alpha + \gamma) \equiv \chi {\boldsymbol .} \alpha \dot{{}+{}} \chi {\boldsymbol .} \gamma,$$
that is, if $\{c_i(X+Y)\}$ is the sequence of cumulants of the
summation of two independent r.v.'s $X$ and $Y$,
then $c_i(X+Y)=c_i(X)+c_i(Y);$
\item[{\it ii)}] {\it homogeneity property}:
$$\chi {\boldsymbol .} (a \alpha) \equiv a (\chi {\boldsymbol .}
  \alpha), \,\,\text{for all } a \in {\mathbb R};
$$
that is, $c_i(a X)= a^i \, c_i(X)$ if $\{c_i(a X)\}$ denotes the sequence of cumulants of $a X.$
\item[{\it iii)}] {\it semi-invariance property}:
$$\chi {\boldsymbol .} (\alpha + a \, u) \equiv \chi {\boldsymbol .} \alpha \dot{{}+{}} a \chi,$$
that is, for all $a \in {\mathbb R},$ we have $c_1(X + a)  =  c_1(X) + a$ and $c_i(X + a)  =  c_i(X)$
for all positive integers $i$ greater than $1.$
\end{enumerate}

The next proposition follows from Equivalence~\eqref{(prpart)} and paves the way
to a symbolic handling of Abel sequences, as will become clearer
in Section~6.
\begin{proposition} \cite{dps} \label{corcum}
If $\kappa_{\alpha}$ is the $\alpha$-cumulant umbra, then, for all positive integers $i$,
\begin{equation}
\alpha^i \simeq \kappa_{{\scriptscriptstyle \alpha}} (\kappa_{{\scriptscriptstyle \alpha}}+\alpha)^{i-1} \quad  \hbox{and}
\quad  \kappa_{{\scriptscriptstyle \alpha}}^i \simeq \alpha (\alpha - 1 {\boldsymbol .} \alpha)^{i-1}.
\label{(ric)}
\end{equation}
\end{proposition}
The former Equivalence~\eqref{(ric)} has been assumed by Rota and Shen
as definition of the cumulant umbra \cite{Shen2}.
In terms of moments, this equivalence reads
$$a_i = \sum_{j=0}^{i-1} {\binom i j} a_j c_{i-j},$$
largely used in statistic framework \cite{Pistone, Smith}.

Since $\E[(\chi)_i]= (-1)^{i-1}(i-1)!$ for all positive integers $i$ \cite{Dinardoeurop}, the formula
expressing the cumulants $\{c_j\}$ in terms of the moments $\{a_i\}$ is recovered from Equations~\eqref{dotproduct1}
and~\eqref{dotproduct2}, that is,
\begin{equation}\label{id:cum vs mom}
c_j = \sum_{\lambda \vdash j} \mathrm{d}_{\lambda}(-1)^{\nu_{\lambda}-1}(\nu_{\lambda}-1)! \, a_{\lambda}.
\end{equation}
\begin{theorem}[\sc Inversion theorem] \cite{Dinardoeurop} \label{inv}
If  $\kappa_{\alpha}$ is the $\alpha$-cumulant umbra,  then
$\alpha \equiv \beta {\boldsymbol .} \kappa_{\alpha}.$
\end{theorem}
In particular, from~\eqref{(comp)} we have
\begin{equation}\label{id:cum vs Bell}
a_i={\calY}_i(c_1,c_2,\ldots,c_i) \quad \hbox{and} \quad a_i=\sum_{\lambda\vdash i} \mathrm{d}_{\lambda}c_{\lambda}.
\end{equation}
The inversion theorem in Theorem~\ref{inv} justifies Definition~\ref{(cumulant)} since in terms of generating functions we have
$f(\alpha,z) = \exp[f(\kappa_{\alpha},z)-1]$, corresponding to Equation~\eqref{(cumRota)}.
Moreover, since, from Inversion Theorem~\ref{inv}, any umbra $\alpha$ could be seen as the partition umbra of its cumulant
umbra $\kappa_{{\scriptscriptstyle \alpha}}$, it is possible to prove a more general result:  every polynomial sequence of binomial type is completely determined by the sequence of its (formal) cumulants. On the other hand, we have already seen in the previous section that the polynomial umbra $x {\boldsymbol .} \beta {\boldsymbol .} \kappa_{{\scriptscriptstyle \alpha}}$ admits a sequence of moments of binomial
type, see Equivalence \eqref{(eq:somma)}. Due to Inversion Theorem~\ref{inv}, any polynomial sequence of binomial
type is umbrally represented by a suitable polynomial umbra $x {\boldsymbol .} \alpha.$ This statement gives the connection
between L\'evy processes and compound Poisson processes. We will get back to this connection in Example~\ref{LevComp}.

If we swap the umbra $\alpha$ and the umbra $\chi$ in $\chi {\boldsymbol .} \alpha,$ the resulting auxiliary umbra has moments equal to the factorial moments $\{a_{(i)}\}_{i \geq 0}$ of $\alpha$ given in~\eqref{(momfact)}, that is, $E[(\alpha {\boldsymbol .} \chi)^i] = a_{(i)}.$  The umbra $\alpha {\boldsymbol .} \chi$ is called the {\it $\alpha$-factorial umbra}. Factorial moments provide very concise expressions for moments of some discrete distributions, such as the binomial distribution \cite{Dinardoeurop}, for example.
%
\subsection{Randomized compound Poisson r.v.'s}
The way to characterize the umbra representing a r.v.\ $X$ is to determine the
sequence of its moments $\{a_i\}$. When this sequence exists, this can be done by comparing
its moment generating function $E[\exp(z X)]$ with the generating
function of the umbra. Recall that, when $E[\exp(z X)]=f(z),$ for $z$ in a suitable neighborhood
of the origin,
then $f(z)$ admits an exponential expansion in terms of  moments, which are completely determined by the corresponding distribution function
(and vice-versa). In this case, the moment generating function encodes all the information of $X$, and the notion of identically distributed r.v.'s corresponds to the similarity among umbrae. In the symbolic method of moments, instead, the convergence of
$f(\alpha,z)$ is not relevant \cite{Stanley}. This means that we can define the umbra
whose moments are the same as the moments of a log-normal r.v.,\footnote{If $X$ is a standard Gaussian r.v.\
${\mathcal N}(0,1),$ then the log-normal r.v.\ is $Y=\exp(X).$} even if this r.v.\ does not admit a moment generating function.
In Section~3, generating functions of umbrae have been used to recover Wald's martingale. But generating functions
have more advantages. One of which is the encoding of symbolic operations among umbrae in terms of suitable operations among generating functions.
For example, Equation~\eqref{(cumRota)} is encoded in the equivalence $\alpha \equiv \beta \punt \kappa_{\alpha}$
of Inversion Theorem~\ref{inv}, and, more generally, the exponential of a generating function corresponds to the generating
function of the polynomial $\alpha$-partition umbra
\begin{equation}
f(x {\boldsymbol .}  \beta {\boldsymbol .} \alpha, z)=\exp \big[ x \, (f(\alpha,z)-1) \big].
\label{(genalfa)}
\end{equation}
\begin{example}[\sc L\'evy processes] \label{LevComp}
If $\{X_t\}$ is a L\'evy process, then its moment generating function is
$E[e^{z X_t}]=[\phi(z)]^t$, where $\phi(z)$ is the moment generating function of $X_1$ \cite{Sato}.
This property parallels the infinite divisibility property of ${\mathbb P}_t$ given at the beginning
of this section. In particular,
$$E[e^{z X_t}] = \exp[t \, \log\phi(z)] = \exp[t \, k(z)],$$
where $k(z)$ is the cumulant
generating function of $X_1,$ satisfying $k(0)=0.$ Comparing $\exp[t \, k(z)]$ with
$\exp[t(f(\kappa_{{\scriptscriptstyle \alpha}},z)-1)]$ given in Equation \eqref{(genalfa)}, with
$x$ replaced by $t,$ we infer that
\begin{enumerate}
\item[{\it i)}] a L\'evy process is umbrally represented by $t {\boldsymbol .} \beta {\boldsymbol .} \kappa_{{\scriptscriptstyle \alpha}}$
representing in turn a compound Poisson process of parameter $t;$
\item[{\it ii)}] the umbra $\kappa_{{\scriptscriptstyle \alpha}}$ represents the coefficients of $k(z),$
which are the cumulants of $X_1;$
\item[{\it iii)}] the compound Poisson process of parameter $t$ involves i.i.d.r.v.'s $\{X_i\}_{i \geq 0}$ whose moments
are the  cumulants of $X_1.$
\end{enumerate}
\end{example}
The composition of generating functions corresponds to an iterated dot-product. Indeed, moments of the auxiliary umbra $(\gamma {\boldsymbol .} \beta) {\boldsymbol .} \alpha$ can be obtained from~\eqref{(alfapartition)}, by replacing
$x$ by $\gamma$:
\begin{equation}
\E[\Phi_i(\gamma {\boldsymbol .} \beta)]= \sum_{\lambda \vdash i} d_{\lambda} \, g_{\nu_{\lambda}} a_{\lambda}.
\label{(momcomp)}
\end{equation}
The auxiliary umbra  $ (\gamma {\boldsymbol .} \beta)  {\boldsymbol .} \alpha $  is called the {\it composition umbra} of $\alpha$ and $\gamma,$ since its $i$-th moment corresponds to the $i$-th coefficient of the composition $f(\gamma, f(\alpha,z)-1).$
\begin{theorem} \cite{Dinardo1} \label{compass}
If $\alpha, \gamma \in {\calA}$, then
$(\gamma {\boldsymbol .} \beta)  {\boldsymbol .} \alpha  \equiv  \gamma {\boldsymbol .} (\beta  {\boldsymbol .} \alpha).$
\end{theorem}
Due to Theorem~\ref{compass}, we may write the composition umbra as $\gamma {\boldsymbol .} \beta {\boldsymbol .} \alpha.$
The composition umbra $(\gamma {\boldsymbol .} \beta)  {\boldsymbol .} \alpha$ represents a compound randomized  Poisson r.v.,
i.e., a random sum indexed by a randomized Poisson r.v.\ represented by the umbra $\gamma {\boldsymbol .} \beta.$
The composition umbra $\gamma {\boldsymbol .} (\beta {\boldsymbol .} \alpha)$ generalizes the random sum of i.i.d.\
compound Poisson r.v.'s with parameter~$1,$ indexed by an integer r.v.\ $X,$ i.e., a randomized compound Poisson r.v.\
with random parameter $X$ represented by the umbra $\gamma.$

Paralleling the composition of generating functions, the {\it compositional inverse} of $\alpha$ is the auxiliary umbra $\alpha^{\cop}$ having generating function $f(\alpha^{\cop},z)$ satisfying
$$f[\alpha^{\cop},f(\alpha,z)-1]=f[\alpha,f(\alpha^{\cop},z)-1]=1+z$$
or
$
\alpha {\boldsymbol .}  \beta {\boldsymbol .}  \alpha^{\cop} \equiv \alpha^{\cop} {\boldsymbol .}  \beta {\boldsymbol .}  \alpha \equiv \chi,
$
since $f(\chi,z)=1+z.$
\begin{theorem} \cite{Niederhausen}
An umbra $\alpha$ has a compositional inverse $\alpha^{\cop}$ if and only if $\E[\alpha] = a_1 \ne 0.$
\end{theorem}
A very special umbra is the compositional inverse $u^{\cop}$ of the
unity umbra $u$ which satisfies
\begin{equation}
u {\boldsymbol .}  \beta {\boldsymbol .}  u^{\cop} \equiv u^{\cop}
{\boldsymbol .}  \beta {\boldsymbol .}  u \equiv \chi.
\label{(compinvun)}
\end{equation}
The following result is very useful in performing computations with dot-products.
\begin{proposition} \cite{Dinardoeurop}
\label{(cumbell)}
$\chi  \equiv u^{\cop} {\boldsymbol .} \beta \equiv \beta {\boldsymbol .}  {u^{\cop}}$ and $\beta {\boldsymbol .}  \chi \equiv u \equiv \chi {\boldsymbol .}  \beta.$
\end{proposition}
The compositional inverse of an umbra is closely related to the Lagrange inversion formula, but we will use Sheffer polynomial
sequences in Section~5 to compute its moments.

We end this section by introducing the connection between composition umbrae and convolutions of multiplicative functions.
Multiplicative functions have been used to define free and Boolean cumulants by means of the lattice of non-crossing and interval partitions, respectively \cite{Nica, Speicher}. This is an issue that we will return to again in Section~6.

\begin{example} \cite{dps} \label{algmul}
Let $\Pi_n$ be the set of all partitions of $[n]$ equipped with the usual refinement order $\leq$,
and denote the minimum and the maximum by ${\boldsymbol 0}_n$ and ${\boldsymbol 1}_n$,
respectively. Denote by $|\pi|$ the number of blocks of $\pi \in \Pi_n.$
If $\sigma\in \Pi_n$ and $\sigma \leq \pi,$ then there exists a unique sequence of non-negative
integers $(k_1, k_2, \ldots , k_n)$ with $k_1 + 2 \, k_2 + \cdots + n \, k_n = |\sigma|$
and $k_1 + k_2 + \cdots + k_n = |\pi|$ such that
$$[\sigma,\pi]= \Pi_1^{k_1} \times \Pi_1^{k_1} \times \cdots \times \Pi_n^{k_n},$$
with $[\sigma,\pi]=\{\tau \in \Pi_n: \sigma \leq \tau \leq \pi \}.$ The sequence $(k_1, k_2, \ldots , k_n)$
is called the {\it type} of the interval $[\sigma, \pi]$, and $k_i$ is the number of blocks of $\pi$ that are
the union of $i$ blocks of $\sigma.$ A function ${\mathfrak f}: \Pi_n \times \Pi_n \rightarrow {\mathbb C}$
is said to be {\it multiplicative} if ${\mathfrak f}(\sigma,\pi) = f_1^{k_1} f_2^{k_2} \cdots f_n^{k_n}$ with $\sigma \leq \pi$ and
$f_n:={\mathfrak f}({\boldsymbol 0}_n,{\boldsymbol 1}_n).$
The M\"obius function $\mu$ is multiplicative with $\mu_n=(-1)^{n-1} (n-1)!.$ Since the
sequence $\{(-1)^{n-1} (n-1)!\}$ is represented by the compositional inverse of the unity
umbra $u,$ the umbral counterpart of the M\"obius function $\mu$ is $u^{\cop}.$
The zeta function $\zeta$ is multiplicative with $\zeta_n=1.$ The umbral
counterpart of the zeta function $\zeta$ is $u.$ As the umbra $u^{\cop}$ is the
compositional inverse of $u,$ the M\"obius function corresponds to the inverse of the zeta function.
Indeed, the convolution $\star$ between two multiplicative
functions ${\mathfrak f}$ and ${\mathfrak g}$ is defined by
$$
({\mathfrak f} \star {\mathfrak g})(\sigma,\pi)
:= \sum_{\sigma\leq\tau\leq\pi} {\mathfrak f}(\sigma,\tau) \, {\mathfrak g}(\tau,\pi).
$$
The function ${\mathfrak h} = {\mathfrak f} \star {\mathfrak g}$ is also multiplicative with $h_n =
({\mathfrak f}\star {\mathfrak g})(\mathbf{0}_n, \mathbf{1}_n)$ and
\begin{equation}
h_n=\sum_{\pi\in\Pi_n}f_{\pi}\,g_{\ell(\pi)},
\label{id:h=f star g}
\end{equation}
where $f_{\pi}:= f_1^{k_1}f_2^{k_2}\cdots f_n^{k_n}$ and $k_i$ is
the number of blocks of $\pi$ of cardinality $i$.
The convolution of two multiplicative functions~\eqref{id:h=f star g} corresponds to a composition
umbra since Equation~\eqref{(momcomp)} can be indexed by set partitions \cite{Bernoulli}, that is,
$$\E[(\gamma {\boldsymbol .} \beta {\boldsymbol .} \alpha)^n] = \sum_{\pi\in\Pi_n} g_{\pi}\,a_{|\pi|}.$$
The identity with respect to the convolution $\star$ is the Delta function $\Delta$ with
$\Delta_n= \delta_{1,n}.$ Its umbral counterpart is the singleton
umbra $\chi.$ The M\"obius function $\mu$ and the zeta function $\zeta$ are inverses of each other with respect to
$\star$, that is, $\mu\star\zeta=\zeta\star\mu=\Delta$. So, this last equality parallels Equivalences~\eqref{(compinvun)}.
\end{example}
\subsection{L\'evy processes}
In Section~3.1, we have stressed the central role played by L\'evy processes in modeling
financial markets. They are preferable to Brownian motion since deviation from normality can be
explained by the jump or spike component. The symbolic representation of  L\'evy processes $\{X_t\}$ by means of the family
of auxiliary umbrae $\{t {\boldsymbol .} \alpha\}_{t \geq 0}$ takes into account the strong interplay between these processes and
infinitely divisible distributions, but does not reveal their jump-diffusion structure.
The same considerations hold when $\alpha$ is replaced by $\beta {\boldsymbol .} \kappa_{\alpha},$
which discloses the role played by cumulants of $X_1$ in the distribution of $X_t$, but does not add more information
on paths of a L\'evy process. Instead, L\'evy processes are also called {\it jump-diffusion processes} since
they result from a summation of a Brownian motion and
a compensated compound Poisson process,\footnote{A compensated Poisson process is $\{N_t - \lambda t\}$  with $\{N_t\}$ a Poisson process $\hbox{\rm Po}(\lambda t).$}
a property
that characterizes unexpected changes in the paths, if any. The following theorem gives the jump-diffusion
decomposition of a L\'evy process.
\begin{theorem} \cite{Sato} \label{decomposition}
If $\{X_t\}$ is a L\'evy process,  then
$$X_t = m \, t + s \, {\mathcal B}_t + \sum_{k=1}^{N_t} J_k - t \, a \, \lambda,$$ where
$m \in {\mathbb R}, s > 0, \{{\mathcal B}_t\}$ is a standard Brownian motion, $\{N_t\}$
is a Poisson process $\hbox{\rm Po}(\lambda t)$, and $\{J_k\}_{k \geq 1}$ are i.i.d.r.v.'s
with $E(J)=a < \infty.$
\end{theorem}
All sources of randomness are mutually independent. In Theo-\break rem~\ref{decomposition}, the diffusion component of
a  L\'evy process is represented by the Wiener process $m \, t + s \,  {\mathcal B}_t,$ while the jump component is represented
by the compensated compound Poisson process $\left(\sum_{k=1}^{N_t} J_k - t a \lambda\right).$ The symbolic method
reduces the jump-diffusion decomposition to a compound Poisson process. This symbolic representation is obtained
by taking into consideration the so-called L\'evy--Khintchine formula, which splits the increment $X_1$ in a summation
according to the jump-diffusion property.
\begin{theorem} [\sc L\'evy--Khintchine formula] \cite{Schoutens}
If $\{X_t\}$ is a L\'evy process, then $E[e^{z X_t}] = [\phi(z)]^t$,
with
\begin{equation}
\phi(z) = \exp\left\{ z \, m + \frac{1}{2}\, s^2 \, z^2 +
\int_{\mathbb R}\left(e^{zx} - 1 - zx{\textbf{1}}_{\{|x| \leq 1\}}\right) {\rm d}\,\nu(x)\right\}.
\label{(lkf)}
\end{equation}
\end{theorem}
The triple $(m, s^2, \nu)$ is called \emph{L\'evy triple}, and $\nu$ is the \emph{L\'evy measure}.
Since in Example~\ref{LevComp} we have seen that $f(t {\boldsymbol .} \beta {\boldsymbol .} \kappa_{\alpha},z) = [\phi(z)]^t,$
with $f(\kappa_{\alpha},z) = \phi(z),$ what we need is to characterize the umbra $\kappa_{{\scriptscriptstyle \alpha}}$
according to \eqref{(lkf)}. First, let us observe that the L\'evy measure $\nu$
contains much information on $X_t.$ If $\nu({\mathbb R}) < \infty,$ almost all paths of $X_t$ have a finite number of jumps on every compact interval.
Under suitable normalization, this corresponds to the condition $\int_{\mathbb R} {\rm d}\,\nu(x) = 1,$
in agreement with $\E[1]=1.$ If $\nu$ admits all moments, then Equation~\eqref{(lkf)} may be rewritten as
\begin{equation}
\phi(z) = \exp \left(c_0 z + \frac{1}{2} s^2 z^2\right) \,
\exp\left\{\int_{\mathbb R}\left(e^{zx} - 1 - zx\right){\rm
  d} \,\nu(x)\right\} \label{LK2}
\end{equation}
where $c_0= m + \int_{\{|x| \geq 1\}} x \, {\rm d}\,\nu(x).$  Equation~\eqref{LK2} allows us to split the umbra $\kappa_{{\scriptscriptstyle \alpha}}$ in a disjoint sum, as the following theorem shows.

\begin{theorem}\label{T2}
A L\'evy process $\{X_t\}$ is umbrally represented by the family of auxiliary umbrae
$$
\{t {\boldsymbol .} \beta {\boldsymbol .} [c_0 \chi \dot{{}+{}} s \eta \dot{{}+{}}\gamma]\}_{t \geq 0},
$$
where $\gamma$ is the umbra associated to the L\'evy measure, that is,
$f(\gamma,z) = 1 + \int_{\mathbb R}\left(e^{zx} - 1 - zx\right){\rm d} \, \nu(x),$
and $\eta$ is an umbra with $f(\eta,z) = 1 + z^2/2.$
\end{theorem}
The jump-diffusion decomposition in Theorem~\ref{decomposition} parallels the symbolic representation in Theorem~\ref{T2}, since from Theorem~\ref{labdisj} we have  $$t {\boldsymbol .} \beta {\boldsymbol .} [c_0 \chi \dot{{}+{}} s \eta \dot{{}+{}}\gamma]
\equiv t {\boldsymbol .} \beta  {\boldsymbol .} [c_0 \chi \dot{{}+{}} s \eta] + t {\boldsymbol .} \beta  {\boldsymbol .} \gamma,$$
and the umbra (see Example~\ref{gaus})
$$\beta {\boldsymbol .} [c_0 \chi \dot{{}+{}} s \eta] \equiv c_0 + \beta {\boldsymbol .}  (s \eta)$$
represents a Gaussian r.v.\ with mean $c_0$ (called {\it drift}) and variance $s^2.$ The compensated Poisson
process is instead umbrally represented by the auxiliary umbra $t  {\boldsymbol .}\beta  {\boldsymbol .} \gamma.$
Indeed, the umbra $\gamma$ can be further decomposed into a pure point part, corresponding to the moment generating function
$ \int_{\mathbb R} e^{zx} \, {\rm d} \, \nu(x),$ and a singular part, corresponding to the moment generating function
$\int_{\mathbb R} (1+z \, x) \, {\rm d} \, \nu(x).$ The singular part of the L\'evy measure is encoded in the umbra $\chi,$
which indeed does not have a probabilistic counterpart.

The singleton umbra also plays a special role with respect to the martingale property of a L\'evy process.
A centered L\'evy process is a process with $E[X_t]=0$ for all $t \geq 0$, and, in this case, almost all
paths of $X_t$ have finite variation.
\begin{theorem} \cite{Applebaum}
A L\'evy process is a martingale if and only if its drift is $c_0 = 0.$
\end{theorem}
A centered L\'evy process is umbrally represented by the auxiliary umbra $\{t  {\boldsymbol .} \beta  {\boldsymbol .}
(s \eta \dot{{}+{}}\gamma)\}_{t \geq 0},$ where the explicit contribution of the singleton umbra is not visible anymore.
\section{Sheffer polynomial sequences}
Many characterizations of Sheffer polynomial  sequences are given  in the
literature. We refer to the one involving generating functions.
\begin{definition}[\sc Sheffer sequences] \label{10.1}
A sequence $\{s_i(x)\}$ of polynomials is called a {\it Sheffer sequence} if and only if its generating function is given by
$$1 + \sum_{i \geq 1} s_i(x) \frac{t^i}{i!} = h(t) \exp\{ x \, g(t) \},$$
with $h(t) =  h_0 + h_1 t + h_2 t^2 + \cdots$,
$g(t) =  g_1 t + g_2 t^2 + \cdots$, and $h_0,g_1 \ne 0.$
\end{definition}
Applications of Sheffer sequences range from analysis to statistics, from combinatorics
to physics. A good account of fields where these polynomials are employed is given in \cite{DiBucchianico}.
Their importance stems from the fact that many polynomial sequences, such as Laguerre polynomials, Meixner polynomials of first and
second kind, Bernoulli polynomials, Poisson--Char\-lier polynomials, and Stirling polynomials are Sheffer polynomial sequences.
Moreover, the time-space harmonic polynomials introduced in Section~3 are special Sheffer sequences
as well.

Sheffer sequences can be considered as the core of the umbral calculus as introduced by Roman and Rota
in \cite{Roman}. There, one finds several chapters devoted to identities, properties
and recurrences involving Sheffer sequences. Following these properties, Taylor has
characterized their coefficients  by using the classical umbral calculus \cite{Taylor}.
In order to handle Sheffer polynomials by the symbolic method of moments, a special auxiliary umbra
has been introduced in \cite{Niederhausen}.
One of the main advantages of this symbolic representation is that
properties of Sheffer sequences are reduced to a few equivalences.

In order to keep the length of the paper within reasonable bounds, here we limit ourselves to presenting
some applications of the symbolic method involving Sheffer polynomial sequences.
Whenever necessary, we recall the umbral equivalences we need. For a more complete
treatment of this subject, the reader should consult  \cite{Niederhausen}.

The first application we deal with concerns the finding of solutions to linear recurrences.
In many special combinatorial problems,
the hardest part in their solution may be the discovery of an effective recursion. Once a
recursion has been established, Sheffer polynomials are often a simple and general tool for finding solutions in closed form.
The main contributions in this respect are due to Niederhausen \cite{Niederhausen1} (see references therein), by means of
finite operator calculus \cite{VIII}. Using the symbolic method of moments, computations are
simplified \cite{recur}. We give several examples below.

A second application involves the Lagrange inversion formula. This formula
gives the coefficients of the compositional inverse of a formal power series. Indeed, in Definition~\ref{10.1} the formal
power series $g(t)$ has a compositional inverse, since $g_1 \ne 0.$ There are many different proofs of this formula, many of them assuming that the involved power series are convergent and represent analytic functions on a disk. Using the symbolic representation
of Sheffer sequences, the proof reduces to a simple computation \cite{Niederhausen}.

A third and last application concerns Riordan arrays. A(n exponential) {\it Riordan array}
is a pair $(g(t),f(t))$ of (exponential) formal power series, where $g(t)$ is an
invertible series and $f(0)=0.$ The pair defines an infinite lower triangular array according to the rule
$$d_{i,k} = \hbox{$i$-th coefficient of \,\,} g(t) \frac{[f(t)]^k}{k!}, \,\,\,\, \hbox{for} \,\, 0 \leq k \leq i < \infty.$$
Riordan arrays are used in a constructive way to find the generating function of many combinatorial sums, to characterize families of orthogonal polynomials, and to compute determinants of Hankel or Toeplitz matrices. Using the symbolic method of moments,
we give a characterization of Riordan arrays. This characterization is very general since the same
holds for the coefficients of Sheffer polynomials and connection constants \cite{Niederhausen}.
 If $\{s_{i}(x)\}$ and $\{r_{i}(x)\}$ are Sheffer sequences,
the constants $b_{i,k}$ in the expression
\begin{equation}
r_{i}(x)= \displaystyle{\sum_{k=0}^{i}} b_{i,k} s_{k}(x)
\label{const}
\end{equation}
are known as {\it connection constants}.  Several changes of bases among polynomial sequences may be efficiently computed by using
connection constants. Connection constants, Riordan arrays, and coefficients of Sheffer polynomials
share the same expression in terms of {\it generalized Bell polynomials} \cite{Niederhausen}. These polynomials
are a generalization of complete Bell exponential polynomials \eqref{(comp)} and also play a fundamental
role in the multivariate version of the symbolic method.

The starting point of the symbolic representation of Sheffer sequences is the following definition.
\begin{definition} \label{shum}
A polynomial umbra $\varsigma_x$ is called a {\it Sheffer umbra for $(\alpha,\gamma)$} if and only if
$\varsigma_x \equiv \alpha+ x {\boldsymbol .} \gamma^{*},$
where $\gamma^{*} \equiv \beta {\boldsymbol .} \gamma^{\cop}$ is called the {\it adjoint umbra} of $\gamma.$
\end{definition}
Of course, to be well posed, in Definition~\ref{shum} the umbra $\gamma$ has to admit
the compositional inverse $\gamma^{\cop},$ that is, $E[\gamma] = g_{1}\ne 0.$
From Definition~\ref{shum}, the generating function of a Sheffer umbra is
$$
f(\varsigma_{x}, z) = f(\alpha,z) \, \exp\{x \, [f^{\cop}(\gamma,z)-1]\},
$$
in agreement with Definition~\ref{10.1}.
\subsection{Solving linear recurrences}
In order to show how to solve linear recurrences by using the symbolic method,
we present two examples, the former involving a uniform r.v., the latter Dyck paths
\cite{Niederhausen}.
\begin{example} \label{unif} Suppose we are asked to solve the difference
equation
\begin{equation}
q_{i}(x+1)=q_{i}(x)+q_{i-1}(x)
\label{recurr1}
\end{equation}
under the condition $\int_{0}^{1}q_{i}(x)\, {\rm d}x=1$ for all positive integers
$i$. This initial condition is equivalent to $\E[q_i(U)]=1,$ where $U$ is a uniform r.v.\
on the interval $(0,1).$ Set $q_i(x) = s_i(x) / i!.$ Then, from~\eqref{recurr1},
we have
\begin{equation}
s_{i}(x+1) = s_i(x) + i \, s_{i-1}(x).
\label{(aaa)}
\end{equation}
A polynomial umbra $\varsigma_{x}$ is a Sheffer umbra if and only if there exists an umbra $\gamma,$
equipped with compositional inverse, which satisfies $\varsigma_{\gamma+x  {\boldsymbol .} u}
\equiv\chi+ \varsigma_{x}$ \cite{Niederhausen}, that is, if and only if $s_{i}(\gamma+ x  {\boldsymbol .} u) = s_i(x) + i \,
s_{i-1}(x)$ for all positive integers $i.$ By comparing this last equality
with~\eqref{(aaa)}, we have to choose as umbra $\gamma$ the unity umbra $u.$ Since $u^{*} \equiv \chi,$ we have
$q_i(x) \simeq (x \boldsymbol{.} \chi + \alpha)^i/i!$
for $\alpha$ depending on the initial condition. By recalling that  $\int_{0}^{1} p(x) \, {\rm d} x =
\E[p(-1 {\boldsymbol .} \iota)]$ for any $p(x) \in {\mathbb R}[x]$ \cite{DinardoShep}, we infer
\begin{multline*}
\int_{0}^{1} q_{i}(x) \, {\rm d}x =1
\quad \text{if and only if}\quad
\E[q_i(-1 \boldsymbol{.} \iota)] = 1 \\
\text{if and only if}\quad  \E[s_i(- 1 \boldsymbol{.} \iota)]=i!.
\end{multline*}
Let the sequence $\{i!\}$ be represented by the umbra $\bar{u},$ called the {\it Boolean unity umbra}.\footnote{%
Its moments are called Boolean cumulants, introduced in Section~$6.$ For Boolean cumulants, the umbra $\bar{u}$ plays the same role as the unity umbra $u$ for classical cumulants.} Then, from $\E[q_i(-1 \boldsymbol{.} \iota)] = 1$, we have
$$-1 \boldsymbol{.} \iota \boldsymbol{.} \chi + \alpha \equiv \bar{u}
\quad
\text{if and only if}
\quad
\alpha \equiv \bar{u} + \iota \boldsymbol{.} \chi.$$
Solutions of the recursion~\eqref{recurr1} are moments of the Sheffer umbra\break $(\iota +\beta+x\boldsymbol{.} u)\boldsymbol{.}\chi$ divided by $i!$ {\hfill \raggedleft{$\qed$}}
\end{example}

\begin{example} \label{Dyck}
A ballot path takes up-steps (\textit{u}) and right-steps (\textit{r}), starting at the origin
and staying weakly above the diagonal. For example, \textit{ururuur}
is a ballot path (starting at the origin and) terminating at the point
$\left( 3,4\right).$
Let $D\left( i,m\right)$ be the number of ballot paths
terminating at $\left( i,m\right)$
which do not contain the pattern (substring) \textit{urru}.

It is not difficult to see that
$D(i,m)$ satisfies the recurrence relation
$$D(i,m)-D(i-1,m)=D(i,m-1)-D(i-2,m-1)+D(i-3,m-1)$$
under the initial conditions
$D(i,i) = D(i-1,i)$ and $D\left( 0,0\right) = 1.$
Set  $D(i,m)=s_{i}(m)/i!.$ The previous recurrence relation may be written as
\begin{equation}
s_{i}(m)-i \, s_{i-1}(m)=s_{i}(m-1)-\left( i\right) _{2}s_{i-2}(m-1)+\left(
i\right) _{3}s_{i-3}(m-1),
\label{ccc5}
\end{equation}
with initial condition $s_{i}(i)=i \, s_{i-1}(i).$ Replace $m$ by $x$ in Equation~\eqref{ccc5}.
The umbral counterpart of the
{\it Sheffer identity}\footnote{{\it Sheffer identity:} A sequence $\{s_i(x)\}$ of Sheffer polynomials satisfies $s_i(x+y)=\sum_{k=0}^i \binom{i}{k} s_k(x) s_{i-k}(y).$} \cite{Niederhausen} is $\varsigma_{x+y} \equiv \varsigma_{x} + y \boldsymbol{.} \gamma^{*}$, which yields $\varsigma_{x} \equiv
\varsigma_{x-1} + \gamma ^{\ast }$
for $y=-1.$ Since $s_{i}(x)-i \, s_{i-1}(x)\simeq (\varsigma _{x}-\chi )^{i},$ we may replace $\varsigma
_{x}$ by $\varsigma _{x-1}+\gamma ^{\ast }$ so that
\begin{equation}
(\varsigma _{x}-\chi )^{i}\simeq(\varsigma _{x-1}+\gamma ^{\ast }-\chi )^{i}.
\label{(eq1)}
\end{equation}
Since
\begin{equation*}
(\varsigma _{x-1}+\gamma ^{\ast }-\chi )^{i} \simeq \varsigma_{x-1}^{i}+ i \, \sum_{j=0}^{i-1}
{\binom{{i-1}}{j}}\varsigma _{x-1}^{i-j-1}\left[
\frac{(\gamma ^{\ast })^{j+1}}{j+1}-(\gamma ^{\ast })^{j}\right],
\end{equation*}
from Equivalence \eqref{(eq1)} and Equation~\eqref{ccc5}, we infer
$$f(\gamma^{*},t)=1+t+\sum_{k\geq 3}t^{k}=(1-t^{2}+t^{3})/(1-t).$$
The initial condition $s_{i}(i) = i \, s_{i-1}(i)$ is equivalent to looking for an
umbra $\alpha $ satisfying
\begin{equation}
(\alpha + i \, \boldsymbol{.}\gamma ^{\ast })^{i}\simeq i \, (\alpha + i \, \boldsymbol{.}
\gamma^{\ast })^{i-1}.  \label{ccc7bis}
\end{equation}
Assume $\alpha \equiv \bar{u}+\zeta .$ In this case the initial
condition~\eqref{ccc7bis} gives\break
$(\zeta + i \, \boldsymbol{.}\gamma ^{\ast })^{i} \simeq \varepsilon^{i}$,
and then $\zeta^i \simeq  - i \, [\gamma^{{\cop}}]^i$ for all positive
integers~$i.$
The solution of the recurrence relation \eqref{ccc5} is
\begin{flushright}
$\displaystyle{\frac{s_i(x)}{i!} \simeq \frac{(\bar{u} + \zeta + x \boldsymbol{.}
\gamma^*)^i}{i!}}. \qquad \qquad \qquad  \qquad \qed$
\end{flushright}
\end{example}

Two special classes of Sheffer polynomials are often employed in applications:
polynomial sequences of binomial type and Appell sequences. In Section~3,
we have already seen that polynomial sequences of binomial type are umbrally represented
by the polynomial umbra $x \boldsymbol{.} \beta \boldsymbol{.} \gamma,$ which is a Sheffer umbra for $(\varepsilon, \gamma).$

An {\it Appell sequence} is a sequence of polynomials $\{p_{i}(x)\}$
satisfying the identity
$$
\frac{{\rm d}}{{\rm d}x} \, p_{i}(x) = i \, p_{i-1}(x), \quad \hbox{for}\,\, i = 1,2,\ldots.
$$
This sequence of polynomials is represented by a Sheffer umbra for $(\alpha,\chi).$
The resulting umbra $\alpha + x \boldsymbol{.} u$ is called {\it Appell umbra}.
Examples of Appell polynomials are the {\it Bernoulli polynomials}
$q_i(x)=E[(\iota + x \boldsymbol{.} u)^i]$ introduced in Section~2,
the {\it Euler polynomials} $$q_i(x)=E\left[ x \boldsymbol{.} u + \left( \frac{1}{2}[\xi - u] \right)^i \right],$$
with $\xi$ the Euler umbra, and the time-space harmonic polynomials $Q_i(x,t)=E[(x + t \boldsymbol{.} \alpha)^i]$ introduced in Section~3.
\subsection{Lagrange inversion formula}
The Lagrange inversion formula gives the coefficients of the compositional inverse of a formal power
series and thus of the compositional inverse of an umbra~$\gamma.$ Following Rota, Shen and Taylor \cite{RotaTaylorShen},
the proof we propose relies on the circumstance that any sequence of binomial type can be represented by Abel polynomials $\{x \, ( x - i \,
a)^{i-1}\}$, $a \in {\mathbb R}.$ We use {\it umbral Abel polynomials} $\{x \, ( x - i \boldsymbol{.} \gamma)^{i-1}\}.$
An exhaustive treatment of umbral Abel polynomials can be found in \cite{Petrullo}.  In contrast to
\cite{RotaTaylorShen}, the employment of the singleton umbra reduces the proof of the Lagrange inversion formula to
a simple computation. We need a preliminary definition.
\begin{definition}
If $\{g_i\}$ is a sequence umbrally represented by an umbra~$\gamma,$ then its {\it derivative umbra} $\gamma_{\scriptscriptstyle D}$
has moments $g_{{\scriptscriptstyle D},i}=i \, g_{i-1}$ for all positive
integers $i.$
\end{definition}
The derivative umbra has moments obtained from a symbolic derivation with respect to $\gamma,$ that is, $i \, g_{i-1} = \E[ i \, \gamma^{i-1}]=\E[\partial_{\gamma} \gamma^{i}].$
\begin{theorem}
[\sc Abel representation of binomial sequences] \cite{Niederhausen} \label{Abel} If
$\gamma$ is an umbra equipped  with a compositional inverse, then
\begin{equation}
(x \boldsymbol{.} \gamma_{\scriptscriptstyle D}^{\ast})^{i} \simeq
x \, (x-i\boldsymbol{.} \gamma)^{i-1},\quad \hbox{for}\,\,
i=1,2,\dots. \label{(Abel)}%
\end{equation}
\end{theorem}
Theorem~\ref{Abel} includes the well-known Transfer Formula \cite{Roman}.
\begin{theorem}[\sc Lagrange inversion formula] \cite{Niederhausen} \label{(InvL1)}
$$({\gamma_{\scriptscriptstyle D}}^{{\cop}})^i \simeq (-i \boldsymbol{.} \gamma)^{i-1}, \quad \hbox{for}\,\, i=1,2,\ldots.$$
\end{theorem}
\begin{proof}
In~\eqref{(Abel)}, replace $x$ by the singleton umbra $\chi$ and observe that
$$(\chi \boldsymbol{.} \beta \boldsymbol{.} {\gamma_{\scriptscriptstyle D}}^
{{\cop}})^i \simeq ({\gamma_{\scriptscriptstyle D}}^{{\cop}})^i \simeq \chi \, ( \chi - i \, \boldsymbol{.}
\gamma)^{i-1},$$
due to $\chi \punt \beta \equiv u$ (see Proposition~\ref{(cumbell)}). The result follows by applying the binomial expansion
to $\chi \, ( \chi - i \boldsymbol{.} \gamma)^{i-1}$ and by recalling that $\chi^{k+1} \simeq 0$ for
$k=1,2,...,i-1.$
\end{proof}
As $-i \boldsymbol{.} \gamma \equiv i \boldsymbol{.} (- 1 \boldsymbol{.}  \gamma),$
moments of ${\gamma_{\scriptscriptstyle D}}^{{\cop}}$
can be computed by using the inverse of $\gamma.$
The umbra $\gamma{\scriptscriptstyle D}$ has first moment equal to $1.$
It is possible to generalize Theorem~\ref{(InvL1)} to
umbrae having first moment different from~$1,$ see \cite{Niederhausen}.
\subsection{Generalized Bell polynomials}
By Theorem~\ref{Abel}, the umbral analog of the well-known \textit{Abel identity}\footnote{The {\it Abel identity} is:
$(x+y)^{i} = \sum_{k = 0}^i \binom{i}{k} (x + k \, a)^{i-k} y(y - k \, a)^{k-1}$, $a \in {\mathbb R}.$}
is:
\begin{equation}
(x+y)^{i} \simeq\sum_{k = 0}^i \binom{i}{k} (x + k \boldsymbol{.} \gamma)^{i-k} y(y-k \boldsymbol{.} \gamma)^{k-1}.
\label{Abel1}
\end{equation}
The polynomials
$${\mathbb B}^{(\gamma)}_{i,k}(x) = \binom{i}{k} (x + k \boldsymbol{.} \gamma)^{i-k},$$
given in Equation~\eqref{Abel1}, are called {\it generalized Bell umbral polynomials} since $E[{\mathbb B}^{(\gamma)}_{i,k}(0)]=$ $ B_{i,k}\left(g_{{\scriptscriptstyle D},1}; \ldots \right.$ $\left. \ldots; g_{{\scriptscriptstyle D},i-k+1}\right),$
where $B_{i,k}$ are the partial Bell exponential polynomials given in Equation~\eqref{(alfapartition)}, and $\{g_{{\scriptscriptstyle D},i}\}$ are the moments of the umbra $\gamma_{\scriptscriptstyle D}.$
Paralleling the complete exponential Bell polynomials, the generalized complete Bell umbral polynomials are
$${\mathbb Y}^{(\gamma)}_{i}(x) = (x + \beta \boldsymbol{.} \gamma_{\scriptscriptstyle D})^{i} \simeq \displaystyle{\sum_{k = 0}^i} \binom{i}{k} (x + k \boldsymbol{.} \gamma)^{i-k} = \displaystyle{\sum_{k = 0}^i} {\mathbb B}^{(\gamma)}_{i,k}(x).$$
Generalized Bell polynomials are useful to give a unifying representation of the coefficients of Sheffer polynomials, of Riordan
arrays, and of connection constants, as we show in the following. Assume $\E[\gamma]=1.$
Generalizations to the case $\E[\gamma] \ne 1$ are given in \cite{Niederhausen}. We need a preliminary definition.
\begin{definition} \label{10.7}
If $\{g_i\}$ is a sequence umbrally represented by an umbra~$\gamma,$ then its {\it primitive umbra} $\gamma_{\scriptscriptstyle P}$
has moments $g_{{\scriptscriptstyle P},i} = g_{i+1}/(i+1)$ for all positive
integers $i.$
\end{definition}
The primitive umbra has moments obtained from a symbolic integration with respect to $\gamma,$ that is,
$$\frac{g_{i+1}}{i+1} = \E\left(\frac{\gamma^{i+1}}{i+1}\right) =
\E\left(\int \gamma^{i} {\rm d} \gamma\right).$$
For $0 \leq k \leq i,$ let $c_{i,k} $ denote the $k$-th coefficient of the $i$-th Sheffer polynomial $s_i(x).$
\begin{theorem}[\sc Coefficients of Sheffer polynomials] \cite{Niederhausen}
If $\{s_i(x)\}$ is a polynomial sequence
umbrally represented by a Sheffer umbra for the pair $(\alpha, \gamma),$ then
$c_{i,k} = \E\left[{\mathbb B}^{(\gamma_{\scriptscriptstyle P})}_{i,k}(\alpha)\right]$
for all positive integers~$i$ and $0 \leq k \leq i.$
\end{theorem}
\begin{theorem}[\sc Riordan array] \cite{Niederhausen}
The elements of an exponential Riordan array for the pair $(f(\alpha,t),f^{\cop}(\gamma,t)-1)$
are umbrally represented by $d_{i,k} =  \E\left[{\mathbb B}^{(\gamma_{\scriptscriptstyle P})}_{i,k}(\alpha)\right]$
for all positive integers $i$ and $0 \leq k \leq i.$
\end{theorem}
So, the theory of Riordan arrays and that of Sheffer sequences are the two sides of the same coin.
\begin{theorem}[\sc Connection constants] \cite{Niederhausen}
If $\{s_{i}(x)\}$ is a polynomial sequence umbrally represented by a Sheffer umbra for $(\alpha,\gamma)$ and $\{r_{i}(x)\}$ is a polynomial sequence umbrally represented by a Sheffer umbra for $(\xi,\zeta),$ then $b_{i,k} = \E\left[{\mathbb B}^{(\phi_{\scriptscriptstyle P})}_{i,k}(\varrho)\right]$ with $\{b_{i,k}\}$ given in \eqref{const} and
$\varrho \equiv  (-1 \boldsymbol{.} \alpha \boldsymbol{.} \beta \boldsymbol{.} \gamma + \xi \boldsymbol{.} \beta \boldsymbol{.} \zeta) \boldsymbol{.} \zeta^*,$ $\phi \equiv \zeta \boldsymbol{.} \gamma^{*}.$
\end{theorem}

\section{Fast symbolic computation}
There are many packages devoted to numerical/graphical statistical toolsets but not doing algebraic/symbolic computations. Quite often, the packages filling this gap are not open source. One of the softwares largely used in the statistical community is {\tt R}.\footnote{Available at {\small ({\tt http://www.r-project.org/})}} {\tt R} is a much stronger numerical programming environment, and the procedures including symbolic software are not yet specifically oriented for statistical calculations. The availability of a widely spread open source symbolic platform will be of great interest, especially with interface capabilities to external programs. The conceptual aspects related to symbolic methods involve more strictly algebraic and combinatorial issues, and the topics introduced in the previous sections are an example.

What we regard as symbolic computation is now evolving towards a universal algebraic language which aims at combining syntactic elegance and computational efficiency. The classical umbral calculus is no doubt an old topic. However, the version here reviewed has not only the virtue of improving the syntactic elegance of the umbral theory but also to offer a different and deeper viewpoint, for example on the role played by cumulants in working with compound Poisson processes.

In this section, we show how this different and deeper viewpoint allows us to reach computational efficiency as by-product. These examples should push statisticians to climb the steep learning curve of a new method and mathematicians to believe what Terence Tao claimed in his blog: \lq\lq A different choice of foundations can lead to a different way of thinking about the subject, and thus to ask a different set of questions and to discover a different set of proofs and solutions. Thus it is often of value to understand multiple foundational perspectives at once, to get a truly stereoscopic view of the subject.\rq\rq

Two examples are proposed here. The first concerns efficient algorithms aiming at the computation of experimental measurements of cumulants, known in the literature as {\it polykays}. The sampling behavior of polykays is much simpler than sample moments
\cite{McCullagh}, but their employment was not so widespread in the statistical community due to the past computational complexity in recovering their expression. Only recently, the symbolic method of moments has contributed in speeding up the procedures for their computations \cite{Statcomp}. The second example shows how to recover classical, Boolean, and free cumulants from moments and vice-versa,
by using only one family of umbral polynomials. Quite surprisingly this algorithm is an application of the symbolic representation of Sheffer sequences making any of the existing procedures obsolete. This part has a consistent theoretical background devoted to the symbolic representation of classical, Boolean and free cumulants, their generalization to Abel-type cumulants, and some open problems which are still under investigation.
\subsection{$k$-statistics and polykays}
In dealing with computations,
the approach used to compute $U$-statistics and estimators of multivariate moments
is not always efficient. So, further manipulations are required.
As an example, we show how to generate $k$-statistics, which are $U$-statistics for cumulants,
by using a suitable compound Poisson r.v. The $i$-th $k$-statistic $k_i$ is the unique symmetric unbiased estimator of the cumulant
$c_i$ of a given statistical distribution \cite{Kendall}, that is, $E[k_i] = c_i.$
\par
A frequently asked question is: why is it so relevant to get efficiency in these calculations? Usually higher order objects require enormous amounts of data
in order to get a good accuracy. Nevertheless, there are different areas, such as astronomy \cite{Prasad}, astrophysics \cite{Ferreira}, biophysics \cite{Muller}, and neuroscience \cite{Staude}, in which there is need to compute high order $k$-statistics in order to detect a Gaussian population. Undoubtedly,
a challenging target is to have efficient procedures to deal with the involved huge amount of algebraic and symbolic
computations.
\par
The theory of $k$-statistics has a long history beginning with Fisher
\cite{Fisher}, who rediscovered the theory of half-invariants of Thiele
\cite{Thiele1}. Fisher introduced $k$-statistics (single and multivariate) as new symmetric
functions of a random sample, with the aim of estimating cumulants without
using moment estimators.  Dressel \cite{Dressel} developed a theory of more
general functions,
revisited later by Tukey \cite{Tukey1}, who called them
{\it poly\-kays}. The whole subject is described by Stuart and Ord
\cite{Kendall} in great detail. In the 1980s, tensor notation was
used by  Speed \cite{Speed}
and extended to polykays and single $k$-statistics. This extension
uncovers the coefficients defining
polykays to be values of the M\"obius function on the lattice of
set partitions. As a consequence, Speed used the set-theoretic
approach to symmetric functions introduced by Doubilet \cite{Doubilet}. In
the same period, McCullagh \cite{McCullagh} simplified the
tensor notation of Kaplan \cite{Kaplan} by introducing the notion of generalized
cumulants. Symbolic operators for expectation and the derivation
of unbiased estimators for multiple sums were introduced by
Andrews {\it et al.} \cite{Andrews}.

A classical way to recover $k$-statistics is to use Theorem~\ref{ttt}
in~\eqref{id:cum vs mom}. The resulting $U$-statistics are expressed in terms of
augmented symmetric polynomials. Then, to get $k$-statistics, we have to express
augmented symmetric polynomials in terms of power sum symmetric polynomials.
\begin{theorem} \cite{Bernoulli} \label{kstat} For $i \leq n$, we have\footnote{For the notation, see Subsection~2.1.}
\begin{equation}
(\chi {\boldsymbol .} \alpha)^i \simeq \sum_{\lambda \vdash i} \frac{(\chi
{\boldsymbol .}  \chi)^{\nu_{\lambda}}} {(n)_{\nu_{\lambda}}} \, d_{\lambda}
\, \sum_{\pi \in \Pi_{\nu_{\lambda}}} (\chi {\boldsymbol .}  \chi)^{ {\boldsymbol .}
\pi} (n {\boldsymbol .}  \alpha)_{S_{\pi}}, \label{(kstatfin1)}
\end{equation}
where $S_{\pi}$ is the subdivision of the multiset $P_\lambda = \{\alpha^{(r_1)}, {\alpha^2}^{(r_2)}, \ldots\}$ corresponding to the partition $\pi \in \Pi_{\nu_{\lambda}}.$
\end{theorem}
A significant computational cost is required to find set partitions or multiset subdivisions,
although these procedures are already optimized \cite{CompStat}.
By using the symbolic method, an improvement in the performance is achievable.
The main idea is to recover $k$-statistics as cumulants of compound Poisson r.v.'s
replacing the augmented symmetric polynomials by exponential polynomials~\eqref{(xpart)}.
The implementation is then sped up since there is no need to compute set partitions.
\par
To this aim, we introduce the {\it multiplicative inverse} of an umbra.
\begin{definition} \label{multumbrae}
The {\it multiplicative inverse} of an umbra $\alpha$ is the umbra~$\gamma$ satisfying
$\alpha \gamma \equiv u.$
\end{definition}
In dealing with a saturated umbral calculus, the multiplicative
inverse of an umbra is not unique, but any two multiplicatively
inverse umbrae of the umbra $\alpha$ are similar. From Definition~\ref{multumbrae}, we have
$a_n g_n = 1$ for all non-negative integers $n,$ i.e., $g_n =
1 / a_n,$ where $\{a_n\}$ and $\{g_n\}$ are umbrally represented by the umbrae $\alpha$ and
$\gamma$, respectively. In the following, the multiplicative
inverse of an umbra $\alpha$ will be denoted by the symbol $1/\alpha.$
The following proposition shows how to express cumulants of an umbra by suitable polynomials linked to
moments of compound Poisson r.v.'s.
\begin{theorem} \cite{Bernoulli} \label{th1}
If $c_i(y)=E[( n {\boldsymbol .}  \chi {\boldsymbol .}  y {\boldsymbol .}  \beta {\boldsymbol .}
\alpha)^i],$ then
$$(\chi {\boldsymbol .}  \alpha)^i \simeq c_i \left( \frac{\chi {\boldsymbol .}  \chi}{n
{\boldsymbol .}  \chi} \right),\quad \hbox{for}\,\, i=1,2,\ldots. $$
\end{theorem}
Thus, moments of the $\alpha$-cumulant umbra can be computed by evaluating
the umbral polynomials $c_i(y)$, after having replaced the indeterminate $y$ by the umbra $\chi {\boldsymbol .}  \chi/ n {\boldsymbol .}  \chi.$ The polynomials $c_i(y)$ are moments of a summation of $n$ uncorrelated
cumulant umbrae of $y {\boldsymbol .}  \beta {\boldsymbol .} \alpha,$
the latter in its turn being a polynomial $\alpha$-partition
umbra representing a compound Poisson r.v.\ of parameter $y.$ The last step is to express the polynomials $c_i(y)$ in terms of power sum symmetric polynomials $n \boldsymbol{.} \alpha^r$ for $r \leq n.$
\begin{theorem} \cite{Bernoulli} \label{thm3}
If
$$p_n(y)=\sum_{k=1}^n (-1)^{k-1} (k-1)! \, S(n,k) \, y^k,$$
where $S(n,k)$ are the Stirling numbers of the second kind, then
\begin{equation}
(\chi\boldsymbol{.}\alpha)^i \simeq \sum_{\lambda \vdash i} d_{\lambda} \,
p_{\lambda} \, \left( \frac{\chi\boldsymbol{.}\chi}{n\boldsymbol{.}\chi} \right)
(n\boldsymbol{.}\alpha)^{r_1} (n\boldsymbol{.}\alpha^2)^{r_2} \cdots ,
\label{(7)}
\end{equation}
with $p_{\lambda}(y)=[p_1(y)]^{r_1} [p_2(y)]^{r_2} \cdots.$
\end{theorem}
Comparing Equivalence~\eqref{(7)} with Equivalence~\eqref{(kstatfin1)}, the reduction of
complexity is evident since no set partitions are involved.
\par
Theorem~\ref{thm3} has been generalized to polykays.
Polykays are symmetric statistics $k_{r, \ldots,\, t}$ satisfying
$$E[k_{r, \ldots,\, t}]= c_r \cdots c_t,$$
where $c_r, \ldots, c_t$ are cumulants.
For simplicity, in the following we just deal with two subindices $k_{r,t},$ the generalization to more than two being straightforward.
As a product of uncorrelated cumulants, $(\chi {\boldsymbol .}  \alpha)^r (\chi^{\prime} {\boldsymbol .}  \alpha^{\prime})^t$ is the umbral counterpart of the polykay $k_{r,t},$ with $\chi, \chi^{\prime}$ uncorrelated singleton umbrae, and $\alpha, \alpha^{\prime}$ uncorrelated and similar umbrae. For polykays, the polynomial
$$p_{r,t}(y) = \sum_{(\lambda \vdash r, \, \eta \vdash t)} y^{\nu_{\lambda}+\nu_{\eta}} \,
(n)_{\nu_{\lambda} + \nu_{\eta}} \, d_{\lambda + \eta} \, a_{\lambda + \eta}$$
plays the same role as the polynomial $c_i(y)$ in Theorem~\ref{th1}.
The following theorem is the analog of
Theorem~\ref{thm3} for polykays.
\begin{theorem} \cite{Bernoulli} \label{thm4}
If $q_{r,t}$ is the umbral polynomial obtained via $p_{r,t}(y)$ by
replacing $y^{\nu_{\lambda}+\nu_{\eta}}$ by
$$
\frac{(\chi\boldsymbol{.}\chi)^{\nu_{\lambda}}(\chi^{\prime}\boldsymbol{.}\chi^{\prime})^{\nu_{\eta}}}{(n\boldsymbol{.}\chi)^{\nu_{\lambda}+\nu_{\eta}}}
\frac{d_{\lambda} d_{\eta}}{d_{\lambda + \eta}},$$
then $(\chi\boldsymbol{.}\alpha)^r (\chi^{\prime}\boldsymbol{.}\alpha^{\prime})^t \simeq q_{r,t}.$
\end{theorem}
Now, let us consider a polynomial umbra $\rho_{y,k}$ whose first $k$ moments are all equal to
$y,$ that is, $\rho^i_{y,k} \simeq (\chi\boldsymbol{.} y\boldsymbol{.} \beta)^i$ for $i=0,1,2,\ldots,k,$
and all the remaining moments are zero. If $k=\max\{r,t\},$ then
$$[n\boldsymbol{.}(\rho_{y,k} \, \alpha)]^{r+t} \simeq \sum_{\lambda \vdash (r+t)} \, d_{\xi} \,
(\chi\boldsymbol{.}\rho_{y,k})_{\xi} \, (n\boldsymbol{.}\alpha)^{s_1} \, (n\boldsymbol{.}\alpha^2)^{s_2} \cdots,$$
which in turn gives the polynomials $p_{r,t}(y)$ in terms of power sum symmetric polynomials, as
$[n\boldsymbol{.}(\rho_{y,k} \, \alpha)]^{r+t} \simeq p_{r,t}(y).$ This device together with Theorem~\ref{thm4} speeds up the computation of polykays.

The resulting algorithm can be fruitfully employed also in the estimation of products of moments. Indeed, unbiased estimators of products of moments
can be computed by using polykays, as we show in the following.
Let us consider a set $\{\alpha_1,\alpha_2,\ldots,\alpha_n\}$ of $n$
uncorrelated umbrae similar to $\alpha.$ Set $\alpha^{ \boldsymbol{.}
  \pi} = \alpha_{i_1}^{|B_1|} \alpha_{i_2}^{|B_2|} \cdots
\alpha_{i_k}^{|B_k|},$
where $\pi=\{B_1,B_2, \ldots, B_k\}$ is a partition of
$[n]$ and $i_1, i_2, \ldots,i_k$ are distinct integers chosen in $[n].$ We have
$E[\alpha^{\boldsymbol{.} \pi}]=a_{\lambda},$ where $\lambda$ is the type of the partition $\pi.$ The formula
giving products of moments in terms of polykays is \cite{Bernoulli}
\begin{equation}
\alpha^{\boldsymbol{.} \pi} \simeq \displaystyle{\sum_{{\substack{\tau \in \Pi_n \\ \tau \leq \pi}}}} (\chi \boldsymbol{.} \alpha)^{\boldsymbol{.} \tau}.
\label{prima}
\end{equation}
By using the M\"obius inversion
\begin{equation}
(\chi \boldsymbol{.} \alpha)^{\boldsymbol{.} \pi} \simeq \displaystyle{\sum_{{\substack{\tau \in \Pi_n \\ \tau \leq \pi}}}} \mu(\tau, \pi) \,
\alpha^{\boldsymbol{.} \pi},
\label{seconda}
\end{equation}
we get polykays in terms of products of moments. Equivalences~\eqref{prima} and~\eqref{seconda}
have also a different meaning. Taking the evaluation of both sides, formulas connecting classical cumulants and moments are obtained in terms of multiplicative functions. These formulas are
$$a_{\pi} = \displaystyle{\sum_{\tau \leq \pi}} c_{\tau}
\text{ for all $\pi$}, \quad
\text{if and only if}\quad
c_{\pi} = \displaystyle{\sum_{\tau \leq \pi}} \mu(\tau,\pi)
a_{\tau}\text{ for all $\pi$},$$
where $a_{\pi}=a_{|B_1|} a_{|B_2|} \cdots a_{|B_k|}$ (the same for $c_{\pi})$, and
$$\mu(\tau,\pi) = (-1)^{n-r} (2!)^{r_3} (3!)^{r_4} \cdots ((n-1)!)^{r_n}$$
is the M\"obius function for the classical partition lattice,
with $n=|\tau|$, $r=|\pi|$, and $\lambda(\tau, \pi) = (1^{r_1},2^{r_2}, \ldots,
n^{r_n})$ the class of $[\tau,\pi].$

Since all these formulas can be generalized to
correlated umbral monomials (see Definition~\ref{4.1}), the next section
is devoted to the multivariate umbral calculus and its applications in probability and statistics.
Since any umbra is a composition umbra, the key to generalize this symbolic calculus to $k$-tuples
of umbral monomials is the multivariate Fa\`a di Bruno formula,
which is related to the composition of multivariate formal power series.
For the fast algorithms computing multivariate $k$-statistics and multivariate polykays by using the symbolic method, the reader is referred to \cite{Statcomp}. In Appendix~1, Table~\ref{table:3} and~\ref{table:4}, we show the advantage of these fast algorithms compared to
existing procedures.
%

\subsection{Computing cumulants through generalized Abel polynomials}
\subsubsection{Classical, Boolean, and free cumulants}
In Section~4, cumulants have been introduced as a sequence of numbers providing an alternative to
moments of a distribution. L\'evy processes are an example because their symbolic
representation involves cumulants of $X_1.$ Since cumulants linearize the convolution of probability
measures in classical, Boolean, and free probability, they have one more advantage: cumulants are a tool
to recognize independent r.v.'s. The linearity of classical convolutions
corresponds to tensor independent r.v.'s \cite{Feller}. The linearity of Boolean convolutions
corresponds to Boolean independent r.v.'s for which the factorization rule holds, but without the commutative
property of the resulting moments. The Boolean convolution \cite{Speicher}
was constructed starting from the notion of partial cumulants, which were extensively
used in the context of stochastic differential equations. The linearity of free convolutions corresponds
to free r.v.'s \cite{Voiculescu}.
The latter belong to the non-commutative probability theory introduced
by Voiculescu \cite{Voiculescu1} with a view to tackle problems in the theory of
operator algebras and random matrix theory.   The combinatorics underlying
classical, Boolean, and free cumulants is closely related to the algebra of multiplicative
functions (see Example~\ref{algmul}). Roughly speaking, one can work with classical, Boolean, and free cumulants
using the lattice of set, interval, and non-crossing partitions,\footnote{The combinatorics underlying non-crossing partitions has been
first studied by Kreweras \cite{Kreweras} and Poupard \cite{Poupard}. Within free probability,
non-crossing partitions are extensively used by Speicher \cite{Speicher}.} respectively.

Since the linearity of cumulants gives a quick access to test whether a given probability measure
is a convolution, the availability of a fast procedure for computing cumulants ---
classical, Boolean, and free --- from the sequence of moments turns out to be useful.
The symbolic method of moments replaces the combinatorics of partition lattices with
the algebra of umbral Abel polynomials given in Theorem~\ref{Abel}. Indeed, all these families of cumulants
share the same parametrization in terms of umbral Abel polynomials. The result is an
algorithm providing just one formula for computing all types of cumulants in terms of
moments and vice-versa.

The aim of this section is to introduce this parametrization, to describe the resulting
algorithm and some problems that this parametrization lives open. Moreover
we define an Abel-type cumulant and explain why the connection among classical, Boolean, and
free cumulants cannot be encoded in a straight way in the formal power series language, as
pointed out also in \cite{Nica}.

Let us consider again the umbral Abel polynomials $\{x (x - i \boldsymbol{.} \alpha)^{i-1}\}.$
\begin{definition} \label{6.6}
The umbral polynomial ${\mathfrak a}_{i}^{(m)} (x, \alpha) = x \, ( x + m \boldsymbol{.} \alpha)^{i-1},$ with $m \in {\mathbb Z},$
is called the {\it generalized umbral Abel polynomial}.
\end{definition}
We have already used generalized umbral Abel polynomials in Section~4 for classical cumulants.
Indeed, from the latter~\eqref{(ric)}, moments of the $\alpha$-cumulant umbra are umbrally equivalent to
${\mathfrak a}_{i}^{(-1)} (\alpha, \alpha)$
\begin{equation}
\kappa_{\scriptscriptstyle \alpha}^i \simeq  \alpha (\alpha - 1 \boldsymbol{.} \alpha)^{i-1}
\simeq {\mathfrak a}_{i}^{(-1)} (\alpha, \alpha).
\label{(cummom)}
\end{equation}
The former~\eqref{(ric)} inverts Equivalence~\eqref{(cummom)} and can be restated as
$$
\alpha^i \simeq  \kappa_{\scriptscriptstyle \alpha} (\kappa_{\scriptscriptstyle \alpha} + \alpha)^{i-1}
\simeq {\mathfrak a}_{i}^{(1)} (\kappa_{\scriptscriptstyle \alpha}, \alpha).
$$
The connection between these two types of generalized umbral Abel polynomials holds in greater generality as the following proposition shows.
\begin{theorem}[\sc First Abel inversion theorem] \cite{tesiPetrullo}
\label{FirstAbel}
\begin{multline}
\alpha^i \simeq {\mathfrak a}_{i}^{(m)} (\gamma, \gamma) \simeq \gamma \,
(\gamma + m \boldsymbol{.} \gamma)^{i-1} \\
\text{if and only if}\quad
\gamma^i \simeq {\mathfrak a}_{i}^{(-m)} (\alpha, \gamma)
\simeq \alpha \, (\alpha - m \boldsymbol{.} \gamma)^{i-1}.
\end{multline}
\end{theorem}
A natural question arising here is the following: does the sequence \break $\{{\mathfrak a}_{i}^{(-m)} (\alpha, \alpha)\},$
resulting from~\eqref{(cummom)} for positive integers $m,$ have the three main algebraic properties of cumulants?

Let us consider the umbra $\bar{\alpha} \equiv \alpha \, \bar{u}$ whose moments are $\E[\bar{\alpha}^i] = i! \, a_i,$
and the sequence $\{{\mathfrak a}_{i}^{(-2)} (\bar{\alpha},
\bar{\alpha})\}.$ By the First Abel Inversion Theorem, given in Theorem~\ref{FirstAbel}, we have
\begin{multline}
\bar{\eta}_{\alpha}^i \simeq  {\mathfrak a}_{i}^{(-2)} (\bar{\alpha},
\bar{\alpha})  \simeq \bar{\alpha} \, (\bar{\alpha} - 2 \boldsymbol{.}
\bar{\alpha})^{i-1} \\
\text{if and only if}\quad
\bar{\alpha}^i \simeq {\mathfrak a}_{i}^{(2)} (\bar{\eta}_{\alpha}, \bar{\alpha}) \simeq \bar{\eta}_{\alpha} \, (\bar{\eta}_{\alpha} +2 \boldsymbol{.} \bar{\alpha})^{i-1}.
\label{(parBool)}
\end{multline}
The umbra $\bar{\eta}_{{\scriptscriptstyle \alpha}}$ is called the {\it $\bar{\alpha}$-Boolean cumulant umbra}. The three main algebraic properties of cumulants can be stated as follows:
\begin{enumerate}
\item[{\it i)}] ({\it additivity})  $\bar{\eta}_{\scriptscriptstyle \xi} \equiv \bar{\eta}_{\scriptscriptstyle \alpha} \stackrel{\boldsymbol{.}}{+}\bar{\eta}_{\scriptscriptstyle \gamma}$
if and only if $ -1 \boldsymbol{.} \bar{\xi} \equiv -1 \boldsymbol{.} \bar{\alpha} \stackrel{\boldsymbol{.}}{+} -1
\boldsymbol{.} \bar{\gamma};$
\item[{\it ii)}] ({\it homogeneity}) $\bar{\eta}_{c {\scriptscriptstyle \alpha}} \equiv c \bar{\eta}_{{\scriptscriptstyle \alpha}}, \, c \in {\mathbb R};$
\item[{\it iii)}] ({\it semi-invariance})
$\bar{\eta}_{\scriptscriptstyle \xi} \equiv \bar{\eta}_{\scriptscriptstyle \alpha} \stackrel{\boldsymbol{.}}{+} c \, \chi$ if and only if $ -1 \boldsymbol{.} \bar{\xi} \equiv -1 \boldsymbol{.} \bar{\alpha} \stackrel{\boldsymbol{.}}{+} -1
\boldsymbol{.} (c \bar{u}).$
\end{enumerate}
Its moments are $\{\E[{\mathfrak a}_{i}^{(-2)} (\bar{\alpha}, \bar{\alpha})]\}$ whose expansion in terms of moments is
$$
\E[\bar{\alpha} (\bar{\alpha} - 2 \boldsymbol{.} \bar{\alpha})^{i-1}] = \sum_{\lambda \vdash i} d_{\lambda} (-1)^{\nu_{\lambda} - 1} \nu_{\lambda}! \, \bar{a}_{\lambda}, \quad \hbox{\rm with} \,\, \bar{a}_{\lambda}= r_1! \, a_1^{r_1}\, r_2! \, a_2^{r_2} \cdots
$$
by Definition~\ref{6.6}.

\begin{theorem} \cite{dps} If $\bar{\eta}_{{\scriptscriptstyle \alpha}}$ is the {\it $\bar{\alpha}$-Boolean cumulant umbra}, then
\begin{equation}
\bar{\eta}_{{\scriptscriptstyle \alpha}} \equiv \bar{u}^{\cop} \boldsymbol{.} \beta \boldsymbol{.} \bar{\alpha} \quad \hbox{and}\quad \bar{\alpha} \equiv \bar{u} \boldsymbol{.} \beta \boldsymbol{.} \bar{\eta}_{{\scriptscriptstyle \alpha}}.
\label{boocum}
\end{equation}
\end{theorem}
Comparing Equivalences~\eqref{boocum} with the equivalences characterizing classical cumulants,
\begin{equation}
\kappa_{{\alpha}}  \equiv \chi \boldsymbol{.} \alpha \equiv u^{\cop} \boldsymbol{.} \beta \boldsymbol{.} \alpha \qquad \hbox{and} \qquad  \alpha \equiv \beta \boldsymbol{.} \kappa_{{\alpha}}  \equiv u \boldsymbol{.} \beta \boldsymbol{.} \kappa_{{\alpha}},
\label{(class1)}
\end{equation}
the analogy is evident. The same analogy has a counterpart within formal power series, as the following remark shows.
\begin{remark}
Let us observe that the ordinary formal power series
$$f(z) = 1 + \sum_{i\geq 1} f_i z^i$$
is the generating function of the umbra $\bar{\alpha},$ that is, $f(\bar{\alpha},z)=f(z).$
If we consider the umbra
$\gamma$ satisfying $f(\bar{\gamma},z)= g(z) = 1 + \sum_{i \geq 1} g_i
z^i$, then the auxiliary umbra $\bar{\alpha}\boldsymbol{.} \beta
\boldsymbol{.} \bar{\gamma}$ has generating function $f[g(t)-1].$ This umbra is the
counterpart of the convolution of multiplicative functions on the lattice of {\it interval partitions}.
Indeed, let us consider the lattice of interval partitions $(\mathcal{I}_i, \leq).$
A partition $\pi$ of $\Pi_i$ is called an {\it interval partition} if
each block $B_j$ of $\pi$ is an interval $[a_j,b_j]$ of $[i]$, that is,
$B_j=[a_j,b_j]=\{x \in [j]\,|\,a_j\leq x\leq b_j\},$ where $a_j,b_j\in
[i]$. The type of each interval $[\sigma,\tau]$ in
$\mathcal{I}_i$ is the same as the one given in Example~\ref{algmul}. A convolution
$\diamond$ is defined on the algebra of multiplicative
functions by
\begin{equation}\label{def:conv''}
({\mathfrak g} \diamond {\mathfrak f})(\sigma,\pi)
:= \sum_{\scriptscriptstyle{\sigma\leq\tau\leq\pi \atop
\tau\in\mathcal{I}_i}} {\mathfrak g}(\sigma,\tau) {\mathfrak f}(\tau,\pi),
\end{equation}
with ${\mathfrak f}$ and ${\mathfrak g}$ multiplicative functions.
Then ${\mathfrak h} = {\mathfrak g}\diamond {\mathfrak f}$ is a multiplicative function and
$$
h_i = \sum_{\tau\in\mathcal{I}_i} f_{\ell(\tau)} g_{\tau}.\,
$$

The convolution~\eqref{def:conv''} of two multiplicative functions  corresponds to a composition
umbra since $E[(\bar{\alpha} {\boldsymbol .} \beta {\boldsymbol .} \bar{\gamma})^i] = \sum_{\pi\in\mathcal{I}_i} a_{|\pi|} \, g_{\pi}.$
\end{remark}
\begin{remark}
There is an interesting parallelism between the umbra having all classical cumulants equal to $1$
and the umbra having all Boolean cumulants equal to $1.$ The Bell umbra $\beta$ is the unique umbra
(up to similarity) having $\{1\}_{i\geq 1}$ as sequence of classical cumulants. Its $i$-th
moment is the $i$-th {\it Bell number}, that is, the number of partitions of
an $i$-set. The umbra $\alpha$ satisfying $\bar{\alpha}\equiv(2\bar{u})_{\scriptscriptstyle D}$
is the unique umbra (up to similarity) having $\{1\}_{i\geq 1}$ as sequence of
Boolean cumulants. The $i$-th moment of this umbra is $2^{i-1},$ that is, the number of
interval partitions $|\mathcal{I}_i|.$
\end{remark}
Free cumulants are closely related to the Lagrange inversion formula.
They were introduced by Voiculescu \cite{Voiculescu} as semi-invariants via the R-transform. Indeed, let us consider a non-commutative r.v.\ $X$, i.e., an element of a unital non-commutative algebra ${\mathbb A}$. Suppose $\phi : {\mathbb A} \rightarrow {\mathbb C}$ is a unital linear functional. The $i$-th moment of $X$ is the complex number $m_i=\phi(X^i)$, while
its generating function is the formal power series $M(z)= 1 + \sum_{i \geq 1} m_i z^i$. The {\it non-crossing} (or {\it free}) {\it cumulants of\/} $X$ are the
coefficients $r_i$ of the ordinary power series $R(z) = 1 + \sum_{i
  \geq 1} r_i z^i$ which satisfies the identity
\begin{equation}
M(z)=R[z M(z)].
\label{(eqq)}
\end{equation}
If we set $M(z)=f(\bar{\alpha},z)$ and $R(z)=f(\K_{\scriptscriptstyle   \alpha},z)$, then $f(\bar{\alpha}_{{\scriptscriptstyle D}}, z) = 1 +
z M(z)$, and Equation~\eqref{(eqq)} gives
$\bar{\alpha} \equiv \K_{\scriptscriptstyle \alpha} \boldsymbol{.} \beta \boldsymbol{.} \bar{\alpha}_{{\scriptscriptstyle D}}$,
which justifies the following definition.

\begin{definition} \label{(def3)}
The umbra $\K_{\scriptscriptstyle \alpha}$ satisfying $(-1 \boldsymbol{.} \K_{\alpha})_{\scriptscriptstyle D} \equiv \bar{\alpha}_{\scriptscriptstyle D}^{\cop}$ is called the {\it free cumulant umbra
of $\alpha.$}
\end{definition}
Denote the auxiliary umbra ${(\ab)}_{\scriptscriptstyle P}$ by $\LL_{\bar{\alpha}}.$ We call this umbra
{\it Lagrange involution of $\alpha$}. Then, from Definition~\ref{(def3)}, we have
$$\K_{\scriptscriptstyle \alpha}\equiv-1 \boldsymbol{.} \LL_{\bar{\alpha}},$$
which is an alternative definition of the free cumulant umbra. The three main algebraic properties of cumulants can be stated
as follows:
\begin{enumerate}
\item[{\it i)}] ({\it additivity}) $\K_{\scriptscriptstyle {\xi}} \equiv \K_{\scriptscriptstyle {\alpha}}
\stackrel{\boldsymbol{.}}{+} \K_{\scriptscriptstyle {\gamma}} \,\,
\text{if and only if} \,\, -1 \boldsymbol{.} \LL_{\bar{\xi}} \equiv -1 \boldsymbol{.} \LL_{\bar{\alpha}} \stackrel{\boldsymbol{.}}{+}
-1 \boldsymbol{.} \LL_{\bar{\gamma}};$
\item[{\it ii)}] ({\it homogeneity}) $\K_{c {\scriptscriptstyle  \alpha}}\equiv c\K_{{\scriptscriptstyle  \alpha}}, \, c \in {\mathbb R};$
\item[{\it iii)}] ({\it semi-invariance}) $\K_{\scriptscriptstyle {\xi}} \equiv \K_{\scriptscriptstyle {\alpha}}
\stackrel{\boldsymbol{.}}{+} c \chi \,\, \text{if and only if} \,\, -1 \boldsymbol{.} \LL_{\bar{\xi}} \equiv\break -1 \boldsymbol{.} \LL_{\bar{\alpha}} \stackrel{\boldsymbol{.}}{+} -1 \boldsymbol{.} \LL_{c\bar{u}}.$
\end{enumerate}
From Definition~\ref{(def3)}
we have
$$\K_{\scriptscriptstyle \alpha} \equiv \bar{\alpha} \boldsymbol{.} \beta \boldsymbol{.} \bar{\alpha}_{\scriptscriptstyle D}^{\cop}\quad \text{and}
\quad \bar{\alpha} \equiv  \K_{\scriptscriptstyle \alpha} \boldsymbol{.} \beta \boldsymbol{.} (-1 \boldsymbol{.}  \K_{\scriptscriptstyle \alpha})_{\scriptscriptstyle D}^{\cop}.$$ If we compare these two equivalences with those corresponding to the $\alpha$-cumulant umbra~\eqref{(class1)} and the $\alpha$-Boolean cumulant umbra~\eqref{boocum}, respectively,
the difference is quite obvious. Similarly, while the characterization
of classical cumulants and Boolean cumulants is closely related to the composition of
generating functions, this is not true for free cumulants. To be more precise, there
does not exist a generating function $U(t)$ such that, if $R(t)$ is the free cumulant
generating function of $M(t)$, then $M(t)=U[R(t)-1],$ see also \cite{dps}.
The same difference arises for convolutions of multiplicative functions
on the lattice of non-crossing partitions, as the following remark shows.
\begin{remark}
A \textit{non-crossing partition} $\pi=\{B_1,B_2,\ldots,B_s\}$ of the set $[i]$ is a partition such that, if $1 \leq h<l<k<m\leq i$,
with $h,k\in B_j$, and $l,m\in B_{j\,'}$, then $j=j\,'$. Let $\mathcal{NC}_i$ denote the set of all non-crossing
partitions of $[i]$. Since $\mathcal{NC}_i \subseteq \Pi_i$, we consider the induced
subposet $(\mathcal{NC}_i,\leq)$ of $(\Pi_i,\leq)$. If $\pi\in\mathcal{NC}_i$, then the decomposition of the intervals
$[\mathbf{0}_i,\pi]= \{\tau\in\mathcal{NC}_i\,|\,\tau\leq\pi\}$ is the same as the one given in Example~\ref{algmul}, while the intervals
$[\pi,\mathbf{1}_i]=\{\tau\in \mathcal{NC}_i\,|\,\pi\leq\tau\}$ have a quite different decomposition, which does not depend on
$\ell(\pi)$ (see \cite{Spe1}). The convolution $\ast$ defined on the multiplicative functions  is given by
$$
({\mathfrak g} \ast {\mathfrak f})(\sigma,\pi)
:= \sum_{\scriptscriptstyle{\sigma\leq\tau\leq\pi \atop
\tau\in\mathcal{NC}_i}}{\mathfrak g}(\sigma,\tau) \, {\mathfrak f}(\tau,\pi).
$$
Hence, if ${\mathfrak h}={\mathfrak g} \ast {\mathfrak f}$, then
$h_i=\sum_{\tau\in\mathcal{NC}_i}g_{\tau}\, {\mathfrak f}(\tau,\mathbf{1}_i).$
Let us observe that the generating function of $f(\LL_{\bar{\alpha}},z)$
is the {\it Fourier transform} $(\mathcal{F}\mathfrak{f})(z)$ of multiplicative functions \cite{Speicher1}.
Indeed, $\LL_{\bar{\omega}}\equiv \LL_{\bar{\gamma}} + \LL_{\bar{\alpha}}$
corresponds to $[\mathcal{F}(\mathfrak{g}\ast
  \mathfrak{f})](z)=(\mathcal{F}\mathfrak{g})(z)(\mathcal{F}\mathfrak{f})(z)$
for all unital multiplicative functions $\mathfrak{f}$
and $\mathfrak{g}$ on the lattice of non-crossing partitions.
\end{remark}
\begin{remark}
The umbra having all free cumulants equal to $1$ has a special meaning. Indeed, the umbra
$\bar{\vartheta}\equiv\bar{u}\boldsymbol{.}{(-1\boldsymbol{.}\bar{u})_{\scriptscriptstyle
D}}^{*}$ is the unique umbra satisfying $\K_{\bar{\vartheta}}\equiv\bar{u}.$ Since
$\vartheta^i \simeq \mathcal{C}_i,$ where $\mathcal{C}_i$ denotes the $i$-th Catalan number,
the umbra $\vartheta$ is called {\it Catalan umbra}. The parallelism with the Bell umbra relies
on the fact that $\mathcal{C}_i = |\mathcal{NC}_i|,$ so the moments of $\vartheta$ give
the numbers of non-crossing partitions \cite{Speicher}.
\end{remark}
\begin{example}[\sc Wigner semicircle distribution] In free
probability, the Wigner semicircle r.v.\footnote{
The Wigner semicircle r.v.\ has probability density function given by a semicircle of radius $R$ centered at $(0, 0)$ and suitably normalized, that is,  $f(x)= 2 \sqrt{R^2-x^2\,}/(\pi R^2) \,$ for $x \in [-R,R]$
and $f(x)=0$ otherwise.} is analogous to the Gaussian r.v.\ in classical probability theory. Indeed, its free
cumulants of degree higher than $2$ are zero as for classical cumulants
of a Gaussian r.v. By using Equivalence~\eqref{(momcomp)}, it is straightforward to prove that
the Wigner semicircle r.v.\ is represented by the umbra $\bar{\vartheta} {\boldsymbol .} \beta {\boldsymbol .}
\bar{\eta},$ where $\vartheta$ is the Catalan umbra and $\eta$ is the umbra introduced in Example~\ref{gaus}.
\end{example}

\begin{theorem}[\sc Abel parametrization] \cite{dps} \label{Abpar}
If $\K_{\scriptscriptstyle\bar{\alpha}}$ is the free
cumulant umbra of $\alpha$, then
\begin{equation}  \bar{\alpha}^i \simeq \K_{\scriptscriptstyle\bar{\alpha}}
(\K_{\scriptscriptstyle\bar{\alpha}} +
i {\boldsymbol .} \K_{\scriptscriptstyle\bar{\alpha}})^{i-1} \quad \hbox{and}
\quad \K_{\scriptscriptstyle\bar{\alpha}}^i \simeq \bar{\alpha}
(\bar{\alpha} - i {\boldsymbol .} \bar{\alpha})^{i-1}.
\label{parfree1}
\end{equation}
\end{theorem}
In terms of generalized Abel polynomials,
Equivalences~\eqref{parfree1} satisfy $\bar{\alpha}^i \simeq {\mathfrak a}_{i}^{(i)}
(\K_{\scriptscriptstyle\bar{\alpha}},
\K_{\scriptscriptstyle\bar{\alpha}})$
and $\K_{\scriptscriptstyle\bar{\alpha}}^i \simeq {\mathfrak
  a}_{i}^{(-i)} (\bar{\alpha}, \bar{\alpha}).$
\begin{example} Complete Bell polynomials
$a_i={\calY}_i(c_1,c_2,\ldots,c_i)$ given in \eqref{id:cum vs Bell} return moments in terms of classical cumulants. For free cumulants, the same role is played by the volume polynomials introduced by Pitman and Stanley \cite{Pitman1}. Let us recall that the $i$-th volume polynomial $V_i(x)$ is the following homogeneous polynomial of degree $i$:
\begin{equation}\label{def:vol pol 1}
V_i(x)=\frac{1}{i!}\sum_{\pbs \, \in \, \hbox{\footnotesize park}(i)}\!\!\!\!x_{\pbs},
\end{equation}
where $x_{\pbs} = x_{p_1} \cdots x_{p_i}$ and $\hbox{park}(i)$ is the set of all {\it parking
functions} of length $i.$\footnote{A \emph{parking function} of length $i$ is a sequence
$\pbs=(p_1, \ldots,p_i)$ of $i$ positive integers, whose non-decreasing arrangement
$\pbs^\uparrow=(p_{j_1}, \ldots,p_{j_i})$ satisfies $p_{j_k}\leq k.$}
The volume polynomial~\eqref{def:vol pol 1} in the uncorrelated umbrae $\alpha_1, \ldots, \alpha_i$
similar to $\alpha$ satisfy $i!\,V_i(\alpha_1, \ldots,\alpha_i) \, \simeq \, \alpha(\alpha+ i \boldsymbol{.}\alpha)^{i-
1},$ see \cite{Petrullo}. In particular, if $\{r_i\}$ denotes the sequence of free cumulants, then
$a_i = i!\, V_i(r_1, \ldots,r_i),$ in parallel with~\eqref{id:cum vs Bell}.
\end{example}
\subsubsection{Abel-type cumulants}
A more general class of cumulants can be defined by using the generalized Abel polynomials. To
the best of our knowledge, a previous attempt to give a unifying approach to cumulants was
given in \cite{Anshelevich}, but the Boolean case does not seem to fit in.
Set ${\mathfrak a}_i^{(-m)}(\alpha) = {\mathfrak a}_i^{(-m)}(\alpha, \alpha).$
\begin{definition}
An {\it Abel-type cumulant\/} $c_{i,m}(\alpha)$ of $\alpha$ is
$$c_{i,m}(\alpha) \simeq {\mathfrak a}_i^{(-m)}(\alpha) \simeq \alpha (\alpha - m \boldsymbol{.} \alpha)^{i-1}$$
if $i$ and $m$ are positive integers.
\end{definition}
Let us consider the infinite matrix
\begin{equation}
C_{\alpha} = \left( \begin{array}{ccc}
{\mathfrak a}_1^{(-1)}(\alpha) &  {\mathfrak a}_1^{(-2)}(\alpha)   & \ldots \\
{\mathfrak a}_2^{(-1)}(\alpha) &  {\mathfrak a}_2^{(-2)}(\alpha)   & \ldots \\
\vdots & \vdots & \\
{\mathfrak a}_n^{(-1)}(\alpha) &  {\mathfrak a}_n^{(-2)}(\alpha)   & \ldots \\
\vdots & \vdots & \\
\end{array} \right).
\label{(matrix)}
\end{equation}
Classical cumulants can be recovered from the first column, and Bool\-ean cumulants from the second column.
The sequence corresponding to the main diagonal of the matrix~\eqref{(matrix)} picks out free cumulants.

\begin{theorem} [\sc Homogeneity property]
$${\mathfrak a}_i^{(-m)}(c \alpha) \simeq c^i {\mathfrak
  a}_i^{(-m)}(\alpha), \quad c \in {\mathbb R}.$$
\end{theorem}
To prove the semi-invariance property, a suitable normalization of
${\mathfrak a}_i^{(-m)}(  \alpha)$ is necessary. For example, for Boolean and free
cumulants, the normalization coefficient is equal to $\{i!\},$ umbrally represented by the Boolean unity $\bar{u}.$
For classical cumulants, no normalization is needed.

Each sequence of cumulants in~\eqref{(matrix)} --- chosen along
columns or along diagonals --- linearizes a suitable convolution of umbrae (i.e., of
moments). More precisely, it is possible to define the convolution $\alpha \stackrel{\scriptscriptstyle (m)}{+} \gamma$ induced by an
integer $m$ as
\begin{equation}
{\mathfrak a}_i^{(-m)}(\alpha \stackrel{\scriptscriptstyle (m)}{+} \gamma) \simeq {\mathfrak a}_i^{(-m)}(\alpha) + {\mathfrak a}_i^{(-m)}(\gamma).
\label{(convolution)}
\end{equation}
These convolutions are commutative. It is possible to prove
the existence of the convolution~\eqref{(convolution)}
induced by any integer $m$ for any pair $\alpha, \gamma$ of umbrae \cite{dps}.
It will be interesting to investigate if further sequences of cumulants known in
the literature can be
detected in the matrix~\eqref{(matrix)}, as for example the binary
cumulants \cite{Sturmfels, Torney}, the tree-cumulants \cite{Zwiernik}
and, more generally, the L-cumulants \cite{Zwiernik1}; or if there
are more sequences of cumulants that can be extracted from the matrix~\eqref{(matrix)}.
\subsubsection{An algorithm to compute cumulants from moments and vice-versa}
A procedure to compute free cumulants from moments has been given in \cite{Bryc} by using nested
derivatives of~\eqref{(eqq)}. The symbolic parame\-trization of cumulants via Abel polynomials
suggests a completely different approach to perform such a computation \cite{DiNardoOliva}.
The main advantage is that the resulting algorithm does not only allow us to compute free cumulants,
but also classical and Boolean cumulants.

The algorithm relies on an efficient expansion of the
umbral  polynomial $\gamma (\gamma + \rho {\boldsymbol .}
\gamma)^{i-1}$ for all positive integers $i,$ with $\rho,\gamma \in {\calA},$ that is,
\begin{equation}
\E\left[ \gamma (\gamma + \rho {\boldsymbol .} \gamma)^{i-1} \right] = \sum_{\lambda
\vdash i} \E[(\rho)_ {\nu_{\lambda}-1}] \, d_{\lambda} \, g_{\lambda},
\label{alg}
\end{equation}
where $\{ g_i\}$ is umbrally represented by the umbra $\gamma.$ To compute the right-hand side of~\eqref{alg}, we need
the factorial moments of $\rho.$ Recall that, if we just know the moments $\{\E[\rho^i]\}$, the factorial moments can be
computed by using the well-known change of basis
$$\E[(\rho )_i] = \sum_{k=1}^i s(i,k) \E[\rho^k],$$
where $\{s(i,k)\}$ are the {\it Stirling numbers of the first kind}.

\medskip
Expansion~\eqref{alg} can be fruitfully used to compute cumulants and moments
as follows:

\medskip
{\it i)} for classical cumulants in terms of moments: we use $\kappa_{{\scriptscriptstyle \alpha}}^i
\simeq\break \alpha (\alpha - 1 {\boldsymbol .} \alpha)^{i-1}$ in \eqref{(ric)}; then in Equation~\eqref{alg} choose
$\rho = -1 {\boldsymbol .} u,$ with $\E[(-1 {\boldsymbol .} u)_i] = (-1)_i = (-1)^i i!,$ and $\gamma= \alpha;$

\smallskip
{\it ii)} for moments in terms of classical cumulants: we use $\alpha^i \simeq\break \kappa_{{\scriptscriptstyle \alpha}}
(\kappa_{{\scriptscriptstyle \alpha}}+\alpha)^{i-1}$ in \eqref{(ric)}; then in Equation~\eqref{alg} choose
$\rho = \beta,$ with $\E[(\beta)_i] = 1,$ and $\gamma=\kappa_{\scriptscriptstyle \alpha};$

\smallskip
{\it iii)} for Boolean cumulants in terms of moments: we use $\bar{\eta}_{\alpha}^i \simeq\break \bar{\alpha} \, (\bar{\alpha}
- 2 \boldsymbol{.} \bar{\alpha})^{i-1}$ in \eqref{(parBool)}; then in Equation~\eqref{alg} choose
$ \rho =-2  {\boldsymbol .} u,$ with $\E[(-2  {\boldsymbol .} u)_i] =  (-1)^i (i+1)!,$ and $\gamma = \bar{\alpha};$

\smallskip
{\it iv)} for moments in terms of Boolean cumulants: we use $\bar{\alpha}^i \simeq\break \bar{\eta}_{\alpha} \, (\bar{\eta}_{\alpha} +2
\boldsymbol{.} \bar{\alpha})^{i-1}$ in~\eqref{(parBool)}; then in Equation~\eqref{alg} choose
$\rho = 2 {\boldsymbol .} \bar{u} {\boldsymbol .} \beta,$ with $\E[(2 {\boldsymbol .} \bar{u} {\boldsymbol .} \beta)_i]
= (i+1)!,$ and $\gamma = \bar{\eta}_ {\scriptscriptstyle\alpha};$

\smallskip
{\it v)} for free cumulants in terms of moments: we use $\K_{\scriptscriptstyle\bar{\alpha}}^i \simeq\break \bar{\alpha}
(\bar{\alpha} - i {\boldsymbol .} \bar{\alpha})^{i-1}$ in~\eqref{parfree1}; then in Equation~\eqref{alg} choose
$\rho =-i {\boldsymbol .} u,$ with $\E[(-i {\boldsymbol .} u)_j]=(-i)_j,$ and $\gamma = \bar{\alpha};$

\smallskip
{\it vi)} for moments in terms of free cumulants: we use $\bar{\alpha}^i \simeq\break \K_{\scriptscriptstyle\bar{\alpha}}
(\K_{\scriptscriptstyle\bar{\alpha}} + i {\boldsymbol .} \K_{\scriptscriptstyle\bar{\alpha}})^{i-1}$ in~\eqref{parfree1};
then in Equation~\eqref{alg} choose $\rho = i {\boldsymbol .} u,$ with $\E[(i {\boldsymbol .} u)_j]=(i)_j,$ and $\gamma =
\K_{\scriptscriptstyle\bar{\alpha}}$.

\medskip
The symbolic  algorithm allows us to compute also Abel-type cumulants by choosing  $\rho = -m {\boldsymbol .} u,$ with $\E[(-m {\boldsymbol .} u)_i]=(-m)_i,$ and $\gamma = \alpha$ in~\eqref{alg}.

\section{Multivariate Fa\`a di Bruno formula}
The multivariate Fa\`a di Bruno formula is a chain rule to compute higher derivatives of multivariable
functions. This formula has been recently
addressed by the following two approaches: combinatorial methods are
used by Constantine and Savits \cite{Const},  and (only in the
bivariate case) by Noschese and Ricci \cite{Ricci}; analytical methods
involving Taylor series are proposed by Leipnik and
Pearce \cite{Leipnik}. We refer to this last paper for a detailed
list of references on this subject and for a detailed account of
its applications. For statistical applications, a good account is the paper of Savits \cite{Savits}.
A comprehensive survey of univariate and multivariate
series approximation in statistics is given in \cite{Kolassa}.

Computing the multivariate Fa\`a di Bruno formula by means of a
symbolic software can be done by recursively applying
the chain rule. Despite its conceptual simplicity, applications
of the chain rule become impractical already for small values, because
the number of additive terms becomes enormous and their computation unwieldy.
Moreover, the output is often untidy and
further manipulations are required to simplify the result.
Consequently, a {\lq\lq compressed\rq\rq} version of the multivariate Fa\`a di Bruno
formula is more attractive, as the
number of variables or the order of derivatives increase.
By using the symbolic method, the {\lq\lq compressed\rq\rq}
version of the multivariate Fa\`a di Bruno formula is nothing else but
a suitable generalization of the well-known multinomial theorem:
\begin{equation}
(x_1 +  \cdots + x_k)^i = \sum_{j_1 +  \cdots + j_k=i}
{\binom i  {j_1,  \ldots, j_k}} x_1^{j_1}   \cdots
x_k^{j_k},
\label{(multthm)}
\end{equation}
where the indeterminates are replaced by auxiliary umbrae.
Suitable choices of these auxiliary umbrae give rise to an efficient computation of the following compositions: univariate
with multivariate, multivariate with univariate, multivariate with the same multivariate, multivariate with different multivariates
in an arbitrary number of components \cite{MultFaa}.

Various attempts to construct a multivariate generalization of the umbral calculus can be found in the literature, see, for example, \cite{Parrish} and \cite{Brown}.

Here, the starting point is the notion of multi-index. If, in the univariate case, the main device is to
replace $a_n$ by $\alpha^n$ via the linear functional $\E,$ in the multivariate case, the main device is to replace
sequences like $\{g_{i_1, \ldots, i_k}\},$ where $\ibs=(i_1, \ldots, i_k) \in \mathbb{N}_0^k$ is a {\it multi-index}, by
a product of powers $\mu_1^{i_1} \cdots \mu_n^{i_k},$ where $\left\{\mu_1, \ldots, \mu_k\right\}$ are umbral monomials in ${\mathbb R}[\calA]$ with not necessarily disjoint support. The $\ibs$-th power of $\mubs=(\mu_1, \ldots,\mu_k)$ is defined by
$\mubs^{\ibs} = \mu_1^{i_1} \cdots \mu_k^{i_k}.$

Products of powers of umbral monomials have already been employed in Theorem~\ref{ttt1} and in Definition~\ref{4.1}, with umbrae indexed by multisets.
Indeed there is a connection between the combinatorics of multisets and the Fa\`a di Bruno formula, as highlighted in \cite{Hardy}.
But, in facing the problem to set up an algorithm to efficiently compute the multivariate Fa\`a di Bruno formula, multi-index notations
have simplified formulae and procedures.
\begin{definition}
A sequence $\{g_{\ibs}\}_{\ibs \in \mathbb{N}_0^k} \in {\mathbb R},$ with
$g_{\ibs} = g_{i_1, \ldots, i_k}$ and $g_{\bf 0} = 1,$ is
umbrally represented by the $k$-tuple $\mubs$ if
$$\E[\mubs^{\ibs}] = g_{\ibs}, \qquad \hbox{for all $\ibs \in \mathbb{N}_0^k.$}$$
\end{definition}
The elements $\{g_{\ibs}\}_{\ibs \in \mathbb{N}_0^k}$ are called {\it
multivariate moments} of $\mubs.$
\begin{definition}
Two $k$-tuples $\nubs$ and $\mubs$ of umbral monomials are said to be {\it similar}, if they
represent the same sequence of multivariate moments. In symbols: $\nubs \equiv \mubs.$
\end{definition}
\begin{definition}
Two $k$-tuples $\nubs$ and $\mubs$ of umbral monomials are said to be {\it uncorrelated\/}
if, for all $\ibs, \jbs \in \mathbb{N}_0^k,$ we have $E[\nubs^{\ibs} \mubs^{\jbs}] = E[\nubs^{\ibs}]
E[\mubs^{\jbs}].$
\end{definition}
If the sequence $\{g_{\ibs}\}_{\ibs \in \mathbb{N}_0^k}$ is umbrally represented by
the $k$-tuple $\mubs$, then the multivariate formal power series
$$
f(\mubs, \tbs) = 1 + \sum_{\ibs \geq 1} g_{\ibs}\frac{\tbs^{\ibs}}{\ibs!}
$$
is the multivariate moment generating function of $\mubs.$
\begin{definition} \label{refrv}
A random vector ${\Xbs}$ is said to be represented by a $k$-tuple $\mubs$ of umbral monomials, if its sequence of multivariate moments $\{g_{\ibs}\}$ is umbrally represented by $\mubs.$
\end{definition}
Thus, a summation of $m$ uncorrelated $k$-tuples similar to $\mubs$ represents a summation of $m$ i.i.d.\ random vectors and can be
denoted by $m \punt \mubs.$ In order to compute its multivariate moments, we need the notion of \textit{composition of a
multi-index} and the notion of \textit{partition of a multi-index}.
\begin{definition} [\sc Composition of a multi-index] A composition $\lambdabs$ of a multi-index $\ibs.$ In symbols $\lambdabs \models \ibs,$ is a matrix $\lambdabs = (\lambda_{ij})$ of non-negative integers, with no zero columns, satisfying $\lambda_{r1}+\lambda_{r2}+\cdots=i_r$ for $r=1,2,\ldots,k.$
\end{definition}
The number of columns of $\lambdabs$ is called the {\it length} of $\lambdabs$ and is denoted by $l(\lambdabs).$
\begin{definition} [\sc Partition of a multi-index] \label{partitiondef} A partition of
a multi-index $\ibs$ is a composition $\lambdabs,$ whose columns are in lexicographic order. In symbols $\lambdabs \mmodels \ibs.$
\end{definition}
As for integer partitions, the notation $\lambdabs = (\lambdabs_{1}^{r_1}, \lambdabs_{2}^{r_2}, \ldots)$
means that in the matrix $\lambdabs$ there are $r_1$ columns equal to $\lambdabs_{1},$ $r_2$ columns equal to $\lambdabs_{2}$, and so on, with $\lambdabs_{1} < \lambdabs_{2} < \cdots$ in lexicographic order. The integer $r_i$ is the multiplicity of $\lambdabs_i.$ We set
${\mathfrak m}(\lambdabs)=(r_1, r_2,\ldots).$
\begin{proposition} \cite{Statcomp} \label{summult} If $\ibs \in {\mathbb N}^k_0,$ then
$$\E[(m \boldsymbol{.} \mubs)^{\ibs}] = \sum_{\lambdabs \mmodels \ibs}
\frac{\ibs!}{{\mathfrak m}(\lambdabs)! \,  \lambdabs!} (m)_{l(\lambdabs)} E[\mubs^{\lambdabs_1}]^{r_1} E[\mubs^{\lambdabs_2}]^{r_2} \cdots,$$
where the sum is over all partitions $\lambdabs$ of the multi-index $\ibs.$
\end{proposition}
Multivariate moments of the auxiliary umbra $(\alpha \boldsymbol{.} \beta \boldsymbol{.} \mubs)^{\ibs}$ can be computed
from Proposition \ref{summult}:
\begin{equation}
\E[(\alpha \boldsymbol{.} \beta \boldsymbol{.} \mubs)^{\ibs}] = \sum_{\lambdabs
\mmodels \ibs} \frac{\ibs!}{m(\lambdabs)! \, \lambdabs!} \,
a_{l(\lambdabs)} E[\mubs^{\lambdabs_1}]^{r_1} E[\mubs^{\lambdabs_2}]^{r_2} \cdots. \label{eq:17}
\end{equation}
To obtain Equation \eqref{eq:17}, the arguments are the same as the ones employed to get Equation \eqref{(momcomp)}.
The umbra $\alpha \boldsymbol{.} \beta \boldsymbol{.} \mubs$ represents a multivariate randomized
compound Poisson r.v., that is, a random sum $S_N = {\boldsymbol X}_1 + \cdots + {\boldsymbol X}_N$
of i.i.d.\ random vectors $\{{\boldsymbol X}_i\}$ indexed by a Poisson r.v.\ $N.$
The right-hand side of Equivalence~\eqref{eq:17} gives moments of $S_N$ if we replace the $\ibs$-th multivariate moment of $\mubs$ by
the $\ibs$-th multivariate moment of $\{{\boldsymbol X}_i\}$ and the moments of $\alpha$ by those of $N.$ This result is the same as
Theorem~4.1 of \cite{Const}, obtained with a different proof and different methods.

A way to compute partitions of a multi-index $\ibs$ is by employing
subdivisions of the multiset $M,$ when $\ibs$ is its vector of
multiplicities.
\begin{example}
The multiset $M=\{\mu_1, \mu_1, \mu_2\}$ corresponds to the multi-index $(2,1),$ and the subdivisions
$$\{\{\mu_1, \mu_1, \mu_2\}\}, \{\{\mu_1, \mu_1\},\{\mu_2\}\}, \{\{\mu_1, \mu_2\}, \{\mu_1\}\},
\{\{\mu_1\}, \{\mu_1\}, \{\mu_2\}\}$$
correspond to the multi-index partitions
$${\binom 2 1}, { \binom {2,0} {0,1}}, {\binom {1,1} {1,0}}, {\binom {1,1,0}  {0,0,1}}.$$
\end{example}
By using the auxiliary umbra $\alpha \boldsymbol{.} \beta \boldsymbol{.} \mubs,$
we can summarize the
different types of composition as follows:
\par \smallskip
{\it a)} ({\it univariate composed with multivariate}) The
composition of a univariate formal power series $f(\alpha,z)$ with a
multivariate formal power series $f(\mubs,\zbs)$ is defined by
$$f(\alpha \boldsymbol{.} \beta \boldsymbol{.} \mubs,\zbs) \, = \,f[\alpha, f(\mubs,\zbs)-1].$$
Equivalence~\eqref{eq:17} gives its $\ibs$-th coefficient.
\par \smallskip
{\it b)} ({\it multivariate composed with univariate}) The
composition of a multivariate formal power series $f(\mubs,\zbs)$
with a univariate formal power series $f(\gamma,z)$ is defined by
$$f[(\mu_1 + \cdots + \mu_k) \boldsymbol{.} \beta \boldsymbol{.} \gamma, z] \, = \,  f[\mubs, (f(\gamma,z)-1, \ldots, f(\gamma,z)-1)].$$
The $i$-th coefficient of $f[(\mu_1 + \cdots + \mu_k) \boldsymbol{.} \beta
\boldsymbol{.} \gamma, z]$ can be computed by replacing $\gamma$ by $(\mu_1 + \cdots +
\mu_k)$ in~\eqref{(momcomp)} and $\alpha$ by $\gamma.$
\par \smallskip
{\it c)} ({\it multivariate composed with multivariate}) The
composition of a multivariate formal power series $f(\mubs,\zbs)$
with a multivariate formal power series $f(\nubs,\zbs)$ is defined by
$$f[(\mu_1 + \cdots + \mu_k) \boldsymbol{.} \beta \boldsymbol{.} \nubs, \zbs] \, = \,
f[\mubs, (f(\nubs,\zbs)-1, \ldots, f(\nubs,\zbs)-1)].$$
The $\ibs$-th coefficient of $f[(\mu_1 + \cdots + \mu_k) \boldsymbol{.} \beta\boldsymbol{.} \nubs, \zbs]$
can be computed by evaluating Equivalence~\eqref{eq:17} after having replaced $\mubs$ by $\nubs$ and
the umbra $\alpha$ by $(\mu_1 + \cdots + \mu_k).$
\par \smallskip
{\it d)} ({\it multivariate composed with multivariate}) The
composition of a multivariate formal power series $f(\mubs,\zbs)$
with a multivariate formal power series $f(\nubs,\zbs_{(m)}),$ with a different
number of indeterminates, is defined by
$$f[(\mu_1 + \cdots + \mu_k) \boldsymbol{.} \beta \boldsymbol{.} \nubs, \zbs_{(m)}] \, = \,
f[\mubs, (f(\nubs,\zbs_{(m)})-1, \ldots, f(\nubs,\zbs_{(m)})-1)].$$

{\it e)} ({\it multivariate composed with different multivariates}) The
composition of a multivariate formal power series $f(\mubs,\zbs)$
with different multivariate formal power series $f(\nubs_i,\zbs_{(m)}),$ with a different
number of indeterminates, is defined by
\begin{multline}
f[\mu_1 \boldsymbol{.} \beta \boldsymbol{.} \nubs_1 + \cdots + \mu_k
    \boldsymbol{.} \beta \boldsymbol{.} \nubs_k, \zbs_{(m)}] \,\\
 = \, f[\mubs, (f(\nubs_1,\zbs_{(m)})-1, \ldots, f(\nubs_k,\zbs_{(m)})-1)].
\label{(mulmul)}
\end{multline}
The $\ibs$-th coefficient of the compositions {\it d)} and {\it e)} can be obtained by using {\it generalized Bell polynomials}.
\begin{definition}
The {\it generalized Bell polynomial\/} of order $\ibs$ is
$$B_{\ibs}^{(\nubs_1,\ldots,\nubs_k)}(x_1, \ldots, x_k) = (x_1 \boldsymbol{.} \beta \boldsymbol{.} \nubs_1 + \cdots + x_k \boldsymbol{.} \beta \boldsymbol{.} \nubs_k)^{\ibs}.$$
\end{definition}
The formal power series whose coefficients are generalized Bell polynomials is
$\exp\left\{\sum_{i=1}^k x_i [f(\nubs_i,\zbs_{(m)}) - 1] \right\}.$ These polynomials give the multivariate
Fa\`a di Bruno formula for all cases listed before:
\par \smallskip
{\it a)} ({\it univariate composed with multivariate})
$B_{\ibs}^{\nubs}(\alpha) = (\alpha \boldsymbol{.} \beta \boldsymbol{.} \mubs)^{\ibs};$
\par \smallskip
{\it b)} ({\it multivariate composed with univariate}) \\
\centerline{$B_{\ibs}^{\gamma}(\mu_1 + \cdots + \mu_k)=[(\mu_1 + \cdots + \mu_k) \boldsymbol{.} \beta \boldsymbol{.} \gamma]^{\ibs};$}
\par \smallskip
{\it c);d)} ({\it multivariate composed with multivariate}) \\
\centerline{$B_{\ibs}^{\nubs}(\mu_1 + \cdots + \mu_k) = [(\mu_1 + \cdots + \mu_k) \boldsymbol{.} \beta \boldsymbol{.} \nubs]^{\ibs};$}
\par \smallskip
{\it e)} ({\it multivariate composed with different multivariates})  \\
\centerline{$B_{\ibs}^{(\nubs_1,\ldots,\nubs_k)}(\mu_1, \ldots, \mu_k) \simeq (\mu_1 \boldsymbol{.} \beta \boldsymbol{.}
\nubs_1 + \cdots + \mu_k \boldsymbol{.} \beta \boldsymbol{.} \nubs_k)^{\ibs}.$}
A {\tt MAPLE} algorithm for computing the evaluation of $B_{\ibs}^{(\nubs_1,\ldots,\nubs_k)}(\mu_1,$ $ \ldots,\break \mu_k)$
is given in \cite{MultFaa}. The main idea is to apply the multivariate version of \eqref{(multthm)} to get
\begin{multline}
{\mathbb E}\left[B_{\ibs}^{(\nubs_1,\ldots,\nubs_k)}(\mu_1, \ldots, \mu_k) \right] \\
 =  \sum_{(\ibs_1, \ldots, \ibs_k): \sum_{j=1}^k \ibs_j = \ibs}  \binom{\kbs}{\ibs_1, \ldots, \ibs_k} {\mathbb E} \left[ (\mu_1 \punt \beta \punt \nubs_1)^{\ibs_1} \cdots (\mu_k \punt \beta \punt \nubs_k)^{\ibs_k} \right],
\label{(expansion)}
\end{multline}
to expand powers $(\mu_i \boldsymbol{.} \beta \boldsymbol{.} \nubs_i)^{k_i},$ to multiply the resulting expansions by
$(\mu_1 \punt \beta \punt \nubs_1)^{\ibs_1} \cdots (\mu_d \punt \beta \punt \nubs_d)^{\ibs_d}$,
and then to apply evaluation lowering powers to indexes.
\subsection{Multivariate cumulants}
Given the sequence $\{g_{\ibs}\}_{\ibs \in \mathbb{N}_0^n},$ its {\it multivariate cumulants} are the coefficients
of the formal power series in the left-hand side of the following equation
$$ 1 +  \sum_{\ibs \geq 1} g_{\ibs}\frac{\tbs^{\ibs}}{\ibs!} =
\exp \left( \sum_{\jbs \geq 1} c_{\jbs}\frac{\tbs^{\jbs}}{\ibs!} \right).$$
\begin{proposition}
The auxiliary umbra $\chi \boldsymbol{.} \mubs$ represents the {\it multivariate cumulants} of $\mubs.$
\end{proposition}
\begin{proof}
The result follows from {\it a)}, by choosing as umbra $\alpha$ the umbra $\chi \boldsymbol{.} \chi,$
and observing that  $\chi \boldsymbol{.} \mubs \equiv (\chi \boldsymbol{.} \chi) \boldsymbol{.} \beta \boldsymbol{.} \mubs$
and
\begin{equation*}
f(\chi \boldsymbol{.} \chi,t)=1 + \log(t+1).\qedhere
\end{equation*}
\end{proof}
The multivariate analog of the well-known semi-invariance property of cumulants
is $\chi \boldsymbol{.} (\mubs + \nubs) \equiv \chi \boldsymbol{.} \mubs \dot{{}+{}} \chi \boldsymbol{.} \nubs,$
where $\mubs$ and $\nubs$ are uncorrelated $n$-tuples of umbral monomials, and
the \textit{disjoint sum} $\mubs \dot{{}+{}} \nubs$ is defined by the multivariate moments
$\E[(\mubs \dot{{}+{}} \nubs)^{\ibs}] = \E[\mubs^{\ibs}] + \E[\nubs^{\ibs}].$
Thus, multivariate cumulants linearize convolutions of uncorrelated $n$-tuples
of umbral monomials.
By using Equation~\eqref{eq:17} with $\alpha$ replaced by $\chi \boldsymbol{.} \chi,$
that is, by expanding the $\ibs$ multivariate moments of $(\chi \boldsymbol{.} \chi) \boldsymbol{.} \beta \boldsymbol{.} \mubs \equiv
\chi \boldsymbol{.} \mubs,$ we get multivariate cumulants in terms of multivariate moments. Still using Equation~\eqref{eq:17}
with $\alpha$ replaced by $u$ and $\mubs$ replaced by $\chi \punt \mubs,$ that is, expanding the $\ibs$-th multivariate moment of $u \boldsymbol{.} \beta \boldsymbol{.} (\chi \punt \mubs) \equiv \mubs,$ we get multivariate moments in terms of multivariate cumulants.
The same equations are given in \cite{McCullagh} using tensor notations. Umbrae indexed by multisets are employed in \cite{Bernoulli}.
\subsection{Multivariate L\'evy processes}
Let $\Xbs = \left\{\Xbs_t\right\}$ be a L\'evy process on ${\mathbb R}^k,$ with
$\Xbs_t = \left(X_1^{(t)}, \ldots, X_k^{(t)}\right),$ where $\left\{X_j^{(t)}\right\}_{j=1}^k$
are univariate r.v.'s all defined on the same probability space for all $t \geq 0. $

In Section~3, we have shown that every L\'evy process $X=\{X_t\}$ on ${\mathbb R}$ is
umbrally represented by the family of auxiliary umbrae $\{t \boldsymbol{.} \alpha\}_{t \geq 0}$
defined by $\E\left[(t \boldsymbol{.} \alpha)^i\right] = E\left[X_t^i\right],$ for all
non-negative integers $i.$ This result can be generalized to the multivariate case as follows.
\begin{definition}\label{7.11}
A L\'evy process $\Xbs=\{\Xbs_t\}_{t \geq 0}$ on ${\mathbb R}^k$ is umbrally represented by
$\{t \boldsymbol{.} \mubs\}_{t \geq 0},$ if $\mubs$
is a $k$-tuple of umbral monomials representing the $k$-tuple of random vectors $\Xbs_1 = (X_1^{(1)},\ldots, X_k^{(1)})$ with the property that
$$E[\mubs^{\ibs}] = E\left[\left\{X_1^{(1)}\right\}^{i_1} \cdots \left\{X_k^{(1)}\right\}^{i_k} \right],$$
for all $\ibs = (i_1, \ldots, i_k) \in \mathbf{N}_0^k.$
\end{definition}
The following theorem gives the multivariate
analog of the L\'evy--Khintchine formula \eqref{(lkf)}.
\begin{theorem}[\sc Multivariate L\'evy--Khintchine formula] \cite{Sato} \label{lkfm}
If $\{\Xbs_t\}$ is a multivariate L\'evy process, then $E\left[e^{\Xbs_t \zbs'}\right]=[\phi_{\Xbs_1}(\zbs)]^t$ with%
\footnote{Here, the notation $\zbs^{\prime}$ denotes the transpose of the row vector $\zbs.$}
\begin{multline}
\phi_{\Xbs_1}(\zbs) =
\exp\left\{\left[\frac{1}{2} \zbs \Sigma \zbs^{\prime}
+ \mbs\zbs^{\prime}\right.\right.\\
\left.\left. + \int_{{\mathbb R}^k}\left(e^{ \zbs\xbs^{\prime}} - 1 -  \, \zbs \, \xbs^{\prime}
\ind_{\{|\xbs|\leq 1\}}(\xbs) \right)\nu(d\xbs)\right]\right\},
\label{mgf}
\end{multline}
where $\mbs \in {\mathbb R}^k$ and $\Sigma$ is a symmetric, positive-definite $k \times k$ matrix.
\end{theorem}
The triple $(\mbs, \Sigma, \nu)$ is called L\'evy triple, $\Sigma$ is called \emph{covariance matrix}, and
$\nu$ is the \emph{L\'evy measure} on ${\mathbb R}^k.$ As in the univariate case, if we set
$\mbs_2 = \int_{{\mathbb R}^k} \xbs \, \ind_{\{|\xbs| > 1\}}(\xbs) \, \nu(d\xbs)$ and
$\mbs_1 = \mbs + \mbs_2$ in Equation~\eqref{mgf}, then
\begin{align}
\phi_{\Xbs_1}(\zbs) & = \exp\left\{\left[\frac{1}{2}\zbs\Sigma \zbs^{\prime} +
\mbs_1\zbs^{\prime} + \int_{{\mathbb R}^k}(e^{\zbs\xbs^{\prime}} - 1 -\zbs\xbs^{\prime})\nu(d\xbs)\right]\right\}.
\label{LK}
\end{align}
Due to the above equation, a multivariate L\'evy process is represented by a family of auxiliary umbrae, as given
in the following theorem.
\begin{theorem} \label{lkfm1} A multivariate L\'evy process $\{\Xbs_t\}_{t \geq 0}$ is represented by the family of auxiliary umbrae
\begin{equation}
 \left\{t \boldsymbol{.} \beta \boldsymbol{.} \left(\etabs C^{\prime} \dot{{}+{}} \chi \boldsymbol{.} \mbs_1 \dot{{}+{}}
\gammabs\right)\right\}_{t \geq 0},\label{levymulti}
\end{equation}
where $\etabs = (\eta_1, \ldots,\eta_k)$ is a $k$-tuple of uncorrelated umbrae satisfying
$f(\etabs,\zbs) = 1 +\zbs\zbs^{\prime}/2,$ the matrix $C$ is the square root of the covariance matrix $\Sigma$ in Equation~\eqref{LK},
and $\gammabs$ is a $k$-tuple of umbral monomials satisfying
$$
f(\gammabs, \zbs) = 1 + \int_{{\mathbb R}^k}
\left(e^{\zbs\xbs^{\prime}} - 1 - \zbs\xbs^{\prime}\right) \nu(d\xbs).
$$
\end{theorem}
As observed previously, every auxiliary umbra, such as $t \boldsymbol{.} \beta \boldsymbol{.} \nubs$, is the symbolic
version of a multivariate compound Poisson process of parameter $t.$  Therefore, for $t$ fixed, the symbolic representation~\eqref{levymulti} of a multivariate L\'evy process is a multivariate compound Poisson r.v.\ of parameter $t$, and the $k$-tuple $(\etabs C^{\prime}  \dot{{}+{}} \chi \boldsymbol{.} \mbs_1 \dot{{}+{}} \gammabs)$ umbrally represents the multivariate cumulants of $\Xbs_1.$

The auxiliary umbra $\chi \boldsymbol{.} \mbs_1$ does not have a probabilistic counterpart. If $\mbs$ is not equal to the zero vector,
this parallels the well-known difficulty to interpret the L\'evy  measure as a probability measure in Subsection~4.2.
\subsection{Multivariate Hermite polynomials.}
In Equation~\eqref{LK},\break assume the L\'evy measure $\nu$ to be zero. Then the moment
generating function of $\{\Xbs_t\}_{t \geq 0}$ is $E\left[e^{\Xbs_t \zbs'}\right]=[\phi_{\Xbs_1}(\zbs)]^t$ with
\begin{equation}
\phi_{\Xbs_1}(\zbs) = \exp\left\{\left(\frac{1}{2}\zbs\Sigma\zbs^{\prime}  +
\mbs_1 \zbs^{\prime} \right)\right\}.
\label{mgfmb1}
\end{equation}
\begin{theorem}[\sc Wiener process] \cite{Sato} \label{bm}
The stochastic process\break $\{\Xbs_t\}_{t \geq 0}$ having moment generating function \eqref{mgfmb1}
is a multivariate Wiener process, that is, $\Xbs_t = \mbs_1 \, t + C  {\boldsymbol B}_t,$  where $C$ is a
$k \times k$ matrix whose determinant is non-zero satisfying $\Sigma= C C^{\prime}$, and $\{{\boldsymbol B}_t\}_{t \geq 0}$ is the multivariate standard Brownian
motion in ${\mathbb R}^k.$
\end{theorem}
When $\mbs_1 = \boldsymbol{0},$ the Wiener process reduces to a multivariate non-standard Brownian motion. From Theorem~\ref{bm}, we see that
the Wiener process $\Xbs_t$ is a special multivariate L\'evy process, and the following result is a corollary of Theorem~\ref{lkfm1}.
\begin{corollary} \label{lkfm2}
The Wiener process $\{\Xbs_t\}_{t \geq 0}$ is represented by the family of auxiliary umbrae
$$\{t \boldsymbol{.} \beta \boldsymbol{.} (\etabs C^{\prime} \dot{{}+{}} \chi \boldsymbol{.} \mbs_1)\}_{t \geq 0}.$$
\end{corollary}

There is a special family of multivariate polynomials which is closely related to the multivariate non-standard Brownian motion: the
(multivariate) {\it Hermite polynomials}. The $\ibs$-th Hermite polynomial
$H_{\ibs}(\boldsymbol{x}, \Sigma)$ is defined by
$$H_{\ibs}(\boldsymbol{x}, \Sigma) = (-1)^{|\ibs|} \frac{D_{\boldsymbol{x}}^{(\ibs)} \phi(\boldsymbol{x}; \boldsymbol{0}, \Sigma)}{{\phi(\boldsymbol{x}; \boldsymbol{0}, \Sigma)}},$$
where $\phi(\boldsymbol{x}; \boldsymbol{0}, \Sigma)$ denotes the multivariate Gaussian
density with mean~$\boldsymbol{0}$
and covariance matrix $\Sigma$ of full rank $k.$
The polynomials $H_{\ibs}(\boldsymbol{x} \Sigma^{-1}, \Sigma^{-1})$ are orthogonal with respect to
$\phi(\boldsymbol{x}; \boldsymbol{0}, \Sigma),$ where $\Sigma^{-1}$ is the inverse matrix of $\Sigma,$
and is denoted by $\tilde{H}_{\ibs}(\boldsymbol{x}, \Sigma).$
\begin{theorem} \cite{Whiters}
If $\boldsymbol{Z} \sim N(\boldsymbol{0}, \Sigma)$ and $\boldsymbol{Y} \sim N(\boldsymbol{0}, \Sigma^{-1})$
are multivariate Gaussian vectors with mean $\boldsymbol{0}$ and covariance matrix $\Sigma$ and $\Sigma^{-1}$,
respectively, then, for all $\ibs \in \mathbb{N}_0^k$, we have
\begin{equation}
H_{\ibs}(\boldsymbol{x}, \Sigma) = E[( \boldsymbol{x} \Sigma^{-1} + {\mathfrak i}
  \boldsymbol{Y})^{\ibs}], \,\, \text{and} \,\,
\tilde{H}_{\ibs}(\boldsymbol{x}, \Sigma) = E[(\boldsymbol{x} + {\mathfrak i} \boldsymbol{Z})^{\ibs}],
\label{(herfin)}
\end{equation}
where ${\mathfrak i}$ is the imaginary unit.
\end{theorem}
In the literature, Equations~\eqref{(herfin)} are called  {\it moment representation} of the Hermite polynomials.
Quite recently, many authors have obtained moment representations for various families of polynomials, especially in the univariate
case, see references in \cite{Dinardooliva4}. The multivariate symbolic method allows us to widen the field of
applicability of this moment representation. For the multivariate Hermite polynomials, we have
\begin{multline}
H_{\ibs}(\boldsymbol{x}, \Sigma) = \E \left[(- 1 \boldsymbol{.} \beta
\boldsymbol{.} \nubs + \boldsymbol{x} \Sigma^{-1})^{\ibs}\right],\\
 \text{and} \,\, \tilde{H}_{\ibs}(\boldsymbol{x}, \Sigma) = \E \left[(-1 \boldsymbol{.} \beta \boldsymbol{.}
\mubs + \boldsymbol{x})^{\ibs} \right],
\label{(herfinumb)}
\end{multline}
with $\nubs$ a $k$-tuple of umbral monomials satisfying $f(\nubs, \tbs)= 1 + \frac{1}{2} \tbs \Sigma^{-1} \tbs^{\prime}$
and $\mubs$ a $k$-tuple of umbral monomials satisfying $f(\mubs, \tbs)= 1 + \frac{1}{2} \tbs \Sigma \tbs^{\prime}.$
In Section~5, we have shown that Appell polynomials are umbrally represented by the polynomial umbra
$x {\boldsymbol .} u + \alpha.$ It is well known that univariate Hermite polynomials are Appell polynomials. Equations~\eqref{(herfinumb)} then show that also multivariate Hermite polynomials are of Appell type.
Let us stress that the umbra $-1 \boldsymbol{.} \beta$ enables us to obtain a simple expression for multivariate Hermite polynomials,
without the employment of the imaginary unit as done in Equations~\eqref{(herfin)}. Moreover, since the moments of
$\boldsymbol{Z} \sim N(\boldsymbol{0},\Sigma)$  are umbrally represented by the umbra
$\beta \boldsymbol{.} \mubs,$ the $k$-tuple $\mubs \equiv \chi \boldsymbol{.} (\beta \boldsymbol{.} \mubs)$ in
$\tilde{H}_{\ibs}(\boldsymbol{x}, \Sigma)$ umbrally represents the multivariate cumulants of $\boldsymbol{Z}.$
\begin{proposition} \label{tsh} $\E \left\{\tilde{H}_{\ibs}[t {\boldsymbol .} \beta
(\etabs C^{\prime}), \, \Sigma] \right\}=0$ for all $t \geq 0.$
\end{proposition}
\begin{proof}
The result follows by observing that the $k$-tuple $\mubs$ in (\ref{(herfinumb)}) satisfies $\mubs \equiv \etabs C^{\prime}$ and
\begin{equation*}
\tilde{H}_{\ibs}[t {\boldsymbol .} \beta (\etabs C^{\prime}) , \, \Sigma] = \E \left\{ \left[-t {\boldsymbol .} \beta {\boldsymbol .} \mubs +
t {\boldsymbol .} \beta (\etabs C^{\prime}) \right]^{\ibs} \right\}.
\qedhere
\end{equation*}
\end{proof}

For the family of polynomials $\{\tilde{H}^{(t)}_{\ibs}(\xbs, \, \Sigma)\}_{t \geq 0},$ Proposition \ref{tsh} parallels the main property of time-space harmonic polynomials introduced in Section~3.1. Since, in the literature, to the best of our knowledge, a theory of multivariate time-space harmonic polynomials has not yet been given, we believe that the setting introduced here could be a fruitful way to build this theory with respect to multivariate L\'evy processes.

More properties on multivariate Hermite polynomials can be found in \cite{MultFaa}. We end this subsection by recalling that
multivariate Hermite polynomials are special generalized Bell polynomials.

\begin{proposition}
If $\nubs_{\boldsymbol{x}}$ is a $k$-tuple of umbral polynomials satisfying $f(\nubs_{\boldsymbol{x}}, \tbs)=
1+ f(\nubs, \tbs + \boldsymbol{x}) - f(\nubs, \tbs),$ with $f(\nubs, \tbs)= 1 + \frac{1}{2} \tbs \Sigma^{-1} \tbs^{\prime}$,
then
$$
H_{\ibs}(\boldsymbol{x}, \Sigma) = (-1)^{|\ibs|} \E \left[ (-1 \boldsymbol{.} \beta \boldsymbol{.} \nubs_{\boldsymbol{x}})^{\ibs} \right].
$$
\end{proposition}
Finally we remark that the efficient computation of the multivariate Hermite polynomials as generalized Bell polynomials
can help us in constructing an efficient algorithm for a multivariate Edgeworth approximation of multivariate density functions \cite{Kolassa}.
\subsection{Multivariate Bernoulli polynomials.}
By the symbolic me\-thod, it is possible to introduce multivariate Bernoulli polynomials as powers of polynomials whose coefficients
involve multivariate L\'evy processes. According to \cite{Dinardooliva4}, this means to give a moment representation for these polynomials.
\begin{definition}
The \emph{multivariate Bernoulli umbra} $\bio$ is the $k$-tuple $(\iota, \ldots, \iota),$
with $\iota$ the Bernoulli umbra.
\end{definition}
\begin{definition}
If $\vbs \in \mathbb{N}_0^k,$ the {\it multivariate Bernoulli polynomial\/} of order $\vbs$ is $B_{\vbs}^{(t)}(\xbs)=\E[(\xbs + t {\boldsymbol .} \bio)^{\vbs}],$ where $\bio$ is the multivariate Bernoulli umbra and $t \geq 0.$
\end{definition}
As corollary, the multivariate Bernoulli polynomials are
$$B_{\vbs}^{(t)}(\xbs)= \sum_{\jbs \leq \vbs} \binom{\vbs}{\jbs} \, \xbs^{\jbs} \, \E[(t {\boldsymbol .} \bio)^{\vbs - \jbs}].$$
The symbolic representation of the coefficients of $B_{\vbs}^{(t)}(\xbs)$ by $\E[(t {\boldsymbol .} \bio)^{\vbs}]$ simplifies the calculus and is computationally efficient, see \cite{Dinardooliva4}.
\begin{proposition} \cite{Dinardooliva4} \label{tshBer} $\mathcal{B}_{\vbs}^{(t)}(- t {\boldsymbol .} \bio) = \mathcal{B}_{\vbs}^{(t)}[t {\boldsymbol .} (-1 {\boldsymbol .} \bio)] = 0.$
\end{proposition}
Let us observe that, when in $B_{\vbs}^{(t)}(\xbs)$ we replace $\xbs$ by the auxiliary umbra $-t {\boldsymbol .} \mubs,$ this means to replace $\xbs^{\jbs}$ by $\E[(-t {\boldsymbol .} \mubs)^{\jbs}].$
The auxiliary umbra $- t {\boldsymbol .} \bio$ is the symbolic version of a multivariate L\'evy process.
Indeed, $-1 {\boldsymbol .} \bio$ is the umbral counterpart of a $k$-tuple identically distributed with
$(U,\ldots, U),$ where $U$ is a uniform r.v.\ on the interval $(0,1),$
and corresponds to the random vector $\X_1 = (X_1^{(1)}, \ldots, X_k^{(1)})$ given in Definition~\ref{7.11}.
By generalizing the definition of conditional evaluation given in Section~3 to the multivariate
case, the multivariate Bernoulli polynomials should become time-space harmonic
with respect to the multivariate L\'evy processes $\{-t {\boldsymbol .} \bio\}_{t \geq 0}.$ Indeed, Proposition~\ref{tshBer}
shows that these polynomials share one of the main properties of time-space harmonic polynomials: when the vector of indeterminates
is replaced by the corresponding L\'evy process, their overall mean is zero.
Similar considerations apply for multivariate Euler polynomials,  see \cite{Dinardooliva4} for details and further connections between these two families of polynomials.
\section{Appendix 1: Tables of computation times}
In Table \ref{table:1}, computation times are given for computing the mean of statistics
such as the second expression in~\eqref{(vrbick1)}. The first column refers to the {\tt Mathematica} routine proposed in
\cite{Vrbick}. The second column refers to a {\tt Maple} routine relying on the symbolic method of moments.
\begin{table}[ht]
\caption{\footnotesize{Comparison of computation times}.} 
\centering 
\begin{tabular}{l c c} 
\hline\hline 
$\qquad   [1^i 2^j 3^k \cdots]$  &  {\tt SIP} &  {\tt MAPLE}  \\ [0.5ex] 
\hline 
$[5^3 \, 9 \, 10][1 \, 2 \, 3 \, 4 \, 5]$             &    0.7     &      0.1 \\
$[5^3 \, 8 \, 9 \, 10][1 \, 2 \, 3 \, 4 \, 5]$           &    5.6     &      0.4 \\
$[6 \, 7 \, 8 \, 9 \, 10][1 \, 2 \, 3 \, 4 \, 5]$           &    2.2     &      0.1 \\
$[6 \, 7 \, 8 \, 9 \, 10][1 \, 2][3 \, 4 \, 5]$       &    3.1     &      0.4 \\
$[6 \, 7][8 \, 9 \, 10][1 \, 2][3 \, 4 \, 5]$      &    4.7     &      1.3 \\
$[5 \, 6 \, 7 \, 8 \, 9 \, 10][1 \, 2 \, 3 \, 4 \, 5]$         &   16.7     &      0.3 \\
$[5 \, 6 \, 7 \, 8 \, 9 \, 10][1 \, 2 \, 3 \, 4 \, 5 \, 6]$       &  348.7     &      1.5 \\
$[6 \, 7 \, 8 \, 9 \, 10][6\, 7][3 \, 4 \, 5][1\, 2]$   & 125.6     &     16.4 \\   \hline
\hline 
\end{tabular}
\label{table:1} 
\end{table}
\smallskip

Table \ref{table:2} lists computation times (in seconds) observed by
using the polynomial expansion in~\eqref{alg} and Bryc's procedure \cite{Bryc},
both implemented in {\tt MAPLE}, release 7, for the conversion from free
cumulants to moments.\footnote{The output is in the same
form as the one given by Bryc's procedure.}
\begin{table}[ht]
\caption{\footnotesize{Comparison of computation times needed to compute free
cumulants in terms of moments. Tasks performed on Intel (R)
Pentium~(R), CPU $3.00$ GHz, $512$ MB RAM}.} 
\centering 
\begin{tabular}{l c c} 
\hline\hline 
$i$ & {\tt MAPLE} (umbral) & {\tt MAPLE} (Bryc)\\
\hline
15 & 0.015 & 0.016\\
18 & 0.031 & 0.062\\
21 & 0.078 & 0.141\\
24 & 0.172 & 0.266\\
27 & 0.375 & 0.703\\
\hline
\end{tabular}
\label{table:2} 
\end{table}
\bigskip

Table \ref{table:3} shows computation times for three different software\break packages.
The first one has been implemented in
{\tt Mathema\-tica} and\break refers to the procedures in \cite{Andrews}, here called
{\tt AS} algorithms, see also\break {\tt http://fisher.utstat.toronto.edu/david/SCSI/chap.3.nb.}
The second one is the package {\tt Math\-Statica} from \cite{MathStatica}.
The third package, {\tt Fast} algorithms, has been implemented
in {\tt Maple 10.x} by using the results of Section~6. The source code
is given in \cite{Dinardofast}. Let us remark that, for all
procedures, the results are in the same output form
and have been performed by the authors on the same platform. To
the best of our knowledge, there is no
{\tt R} implementation for $k$-statistics and polykays.
\begin{table}[ht]
\caption{\footnotesize{Computation times in sec.\ for
$k$-statistics and polykays. Missing computation times mean \lq\lq
greater than $20$ hours\rq\rq}.} 
\centering 
\begin{tabular}{c c c c c c} 
\hline\hline 
$k_{t, \ldots,\, l}$ & {\tt AS} Algorithms & {\tt MathStatica} &
{\tt Fast\/} algorithms  \\ [0.5ex] 
\hline 
$k_5$    &  0.06    & 0.01    &  0.01  \\
$k_7$    &  0.31    & 0.02    &  0.01  \\
$k_9$    &  1.44    & 0.04    &  0.01  \\
$k_{11}$ &  8.36    & 0.14    &  0.01  \\
$k_{14}$ & 396.39   & 0.64    &  0.02  \\
$k_{16}$ & 57982.40 & 2.03    &  0.08  \\
$k_{18}$ &  -       & 6.90    &  0.16  \\
$k_{20}$ &   -      & 25.15   &  0.33  \\
$k_{22}$ &   -      & 81.70   &  0.80  \\
$k_{24}$ &   -      & 359.40  &  1.62  \\
$k_{26}$ &   -      & 1581.05 &  2.51  \\
$k_{28}$ &   -      & 6505.45 &  4.83  \\
$k_{3,2}$&  0.06    & 0.02    &  0.01  \\
$k_{4,4}$&  0.67    & 0.06    &  0.02  \\
$k_{5,3}$&  0.69    & 0.08    &  0.02  \\
$k_{7,5}$&  34.23   & 0.79    &  0.11  \\
$k_{7,7}$&  435.67  & 2.52    &  0.26  \\
$k_{9,9}$&    -     & 27.41   &  2.26  \\
$k_{10,8}$&   -     & 30.24   &  2.98  \\
$k_{4,4,4}$& 34.17  & 0.64    &  0.08  \\\hline 
\end{tabular}
\label{table:3} 
\end{table}
\medskip

Table \ref{table:4} is the same as Table~\ref{table:3}, but for multivariate $k$-statistics and
multivariate polykays. Missing computation times in the second column are due to the fact that no procedures are devoted to multivariate polykays in {\tt MathStatica} (release~$1$).
\begin{table}[ht]
\caption{\footnotesize{Computation times in sec.\ for
multivariate $k$-statistics and multivariate polykays. In the first column, missing computation
times mean \lq\lq greater than $20$ hours\rq\rq}.} 
\centering 
\begin{tabular}{c c c c} 
\hline\hline 
$k_{t_1 \ldots\, t_r;\, l_1 \ldots l_m}$ & {\tt AS} Algorithms &
{\tt MathStatica} & {\tt Fast\/}-algorithms \\ [0.5ex] 
$k_{3 \, 2}$        &   0.20   &  0.02 & 0.01  \\
$k_{4 \, 4}$        &  18.20   &  0.14 & 0.02  \\
$k_{5 \, 5}$        &  269.19  &  0.56 & 0.08  \\
$k_{6 \, 5}$        & 1023.33  &  1.02 & 0.16  \\
$k_{6 \, 6}$        &    -     &  2.26 & 0.33  \\
$k_{7 \, 6}$        &    -     &  4.06 & 0.59  \\
$k_{7 \, 7}$        &    -     &  8.66 & 1.23  \\
$k_{8 \, 6}$        &    -     &  7.81 & 1.16  \\
$k_{8 \, 7}$        &    -     & 15.89 & 2.59  \\
$k_{8 \, 8}$        &    -     & 30.88 & 5.53  \\
$k_{3 \, 3 \, 3}$   & 1211.05  & 0.92  & 0.44  \\
$k_{4 \, 3 \, 3}$   &    -     & 2.09  & 0.34  \\
$k_{4 \, 4 \, 3}$   &    -     & 4.98  & 1.02  \\
$k_{4 \, 4 \, 4}$   &    -     & 13.97 & 2.78   \\
$k_{1 \, 1;\, 1 \, 1}$ & 0.05    &   -   & 0.01 \\
$k_{2 \, 1;\, 1 \, 1}$ & 0.20    &   -   & 0.01 \\
$k_{2 \, 2;\, 1 \, 1}$ & 1.30    &   -   & 0.03 \\
$k_{2 \, 2;\, 2 \, 1}$ & 6.50    &   -   & 0.06 \\
$k_{2 \, 2;\, 2 \, 2}$ & 34.31   &   -   & 0.11 \\
$k_{2 \, 1;\, 2 \, 1;\, 2 \, 1}$   & 81.13  &  - &  0.17 \\
$k_{2 \, 2;\, 1 \, 1;\, 1 \, 1}$   & 30.34  &  - &  0.14 \\
$k_{2 \, 2;\, 2 \, 1;\, 1 \, 1}$   & 127.50 &  - &  0.22 \\
$k_{2 \, 2;\, 2 \, 1;\, 2 \, 1}$   & 406.78 &  - &  0.47 \\
$k_{2 \, 2;\, 2 \, 2;\, 1 \, 1}$   & 467.88 &  - &  0.55 \\
$k_{2 \, 2;\, 2 \, 2;\, 2 \, 1}$   & 1402.55&  - &  1.14 \\
$k_{2 \, 2;\, 2 \, 2;\, 2 \, 2}$   & 3787.41&  - &  2.96 \\
\hline
\end{tabular}
\label{table:4} 
\end{table}
\section{Appendix 2: Families of time-space harmonic polynomials}
\begin{center}
{\sl Hermite polynomials}
\end{center}
\smallskip

The generalized Hermite polynomials $\left\{H_k^{(s^2 t)}(x)\right\}$ \cite{Roman}  are
time-space harmonic polynomials with respect to the Wiener process with zero drift and variance $s^2,$ whose umbral counterpart is the family of auxiliary umbrae
$\{t \boldsymbol{.} \beta \boldsymbol{.} (s \eta)\}_{t \geq 0}.$
Indeed,
$$H_k^{(s^2 t)}(x) = \E[(x - t \boldsymbol{.} \beta \boldsymbol{.} (s \eta))^k], \qquad k = 1,2,\ldots,$$
with $s \in {\mathbb R}^+$ and $\eta$ the umbra introduced in Example~\ref{gaus}.

\medskip
\begin{center}
{\sl Poisson--Charlier polynomials}
\end{center}
\smallskip

The Poisson--Charlier polynomials $\left\{\tilde{C}_k(x, \lambda t)\right\}$ \cite{Roman} are
time-space harmonic with respect to the Poisson process $\{N_t\}_{t \geq 0}$ with intensity parameter $\lambda > 0,$ whose umbral counterpart is the family of auxiliary umbrae
$\{(t \lambda) \boldsymbol{.} \beta\}_{t \geq 0}.$ Indeed,
$$\tilde{C}_k(x, \lambda t) = \sum_{j = 1}^k \, s(k,j) \, \E[(x - (t \lambda) \boldsymbol{.} \beta)^j], \qquad k = 1,2,\ldots,$$
where $\{s(k,j)\}$ are the Stirling numbers of the first kind.

\medskip
\begin{center}
{\sl L\'evy--Sheffer systems}
\end{center}
\smallskip

The L\'evy--Sheffer system \cite{schoutens1} is a sequence of polynomials\break $\{V_k(x,t)\}_{t \geq 0}$
time-space harmonic with respect to the L\'evy process umbrally represented by
$\{-t \boldsymbol{.} \alpha \boldsymbol{.} \beta \boldsymbol{.} \gamma^{{\scriptscriptstyle <-1>}}\}_{t \geq 0},$ since
$$V_k(x,t) = \sum_{i=0}^k  \E[(x - t \boldsymbol{.} \alpha \boldsymbol{.} \beta \boldsymbol{.} \gamma^{{\scriptscriptstyle <-1>}})^i] \,
B_{k,i}(g_1, \ldots, g_{k-i+1}), \ \, k = 1,2,\ldots,$$
where $B_{k,i}$ are partial Bell exponential polynomials and $\{g_i\}$ is the sequence umbrally represented by $\gamma.$

\medskip
\begin{center}
{\sl Laguerre polynomials}
\end{center}
\smallskip

The Laguerre polynomials $\{\mathcal{L}_k^{t-k}(x)\}_{t \geq 0}$ \cite{Roman} are time-space harmonic with respect to the gamma process $\{G_t (1, 1)\}_{t \geq 0},$ with scale parameter~$1$ and shape parameter $1,$ whose umbral counterpart is the family of auxiliary umbrae $\{t \boldsymbol{.} \bar{u}\}_{t \geq 0}.$ Indeed,
$$(-1)^k \, k! \, \mathcal{L}_k^{t-k}(x) = \E[(x - t \boldsymbol{.} \bar{u})^k], \qquad k = 1,2,\ldots,$$
where $\bar{u}$ is the Boolean unity.

\medskip
\begin{center}
{\sl Actuarial polynomials}
\end{center}
\smallskip

The actuarial polynomials $\{g_k(x, \lambda t)\}_{t \geq 0}$ \cite{Roman} are time-space harmonic with respect to the gamma process $\{G_t (\lambda, 1)\}_{t \geq 0},$
with scale parameter $\lambda >0$ and shape parameter $1,$
whose umbral counterpart is the family of auxiliary umbrae $\{(\lambda t) \boldsymbol{.} \bar{u}\}_{t \geq 0}.$  Indeed,
$$g_k(x, \lambda t) = \sum_{i = 1}^k \E[(x - (\lambda t) \boldsymbol{.} \bar{u})^i] \, B_{k,i}(m_1, \ldots, m_{k - i - 1}), \qquad k = 1,2,\ldots,$$
where $B_{k,i}$ are partial Bell exponential polynomials and $m_i =\break \E[(\chi \boldsymbol{.} (-\chi))^i],$ for $i = 1, \ldots, k.$

\medskip
\begin{center}
{\sl Meixner polynomials}
\end{center}
\smallskip

The Meixner polynomials of the first kind $\{M_k(x,t,p)\}_{t \geq 0}$ \cite{Roman} are time-space harmonic with respect to the Pascal process
$\{Pa(p,t)\}_{t \geq 0},$ whose umbral counterpart is the family of auxiliary umbrae\break
$\{t \boldsymbol{.} \bar{u} \boldsymbol{.} d \boldsymbol{.} \beta\}_{t \geq 0}.$
Indeed,
$$(-1)^k \, (t)_k \, M_k(x,t,p) = \sum_{i = 1}^k \E[(x - t \boldsymbol{.} \bar{u} \boldsymbol{.} d \boldsymbol{.} \beta)^i] \, B_{k,i}(m_1, \ldots, m_{k - i - 1}),$$
where $B_{k,i}$ are partial Bell exponential polynomials and for $i = 1, \ldots, k$
$$m_i = \E\left[\left(\chi \boldsymbol{.} \left(-1 \boldsymbol{.} \chi + \frac{\chi}{p} \right)\right)^i \right].$$

\medskip
\begin{center}
{\sl Bernoulli polynomials}
\end{center}
\smallskip

The Bernoulli polynomials $\{B_k(x,n)\}$ \cite{Roman} are time-space harmonic with respect to the random walk $\{n \boldsymbol{.} (-1 \boldsymbol{.}
\iota)\}_{n \geq 0}.$ Indeed,
$$B_k(x,n) = \E\left[\left(x - n \boldsymbol{.} \left(-1 \boldsymbol{.} \iota \right)\right)^k \right], \qquad k = 1,2,\ldots.$$
The $n$-th increment $M_n$ in Example~\ref{3.13} is represented by the umbra $-1 \boldsymbol{.} \iota$ corresponding to a uniform r.v.\ on the interval $[0,1].$

\medskip
\begin{center}
{\sl Euler polynomials}
\end{center}
\smallskip

The Euler polynomials $\{\mathcal{E}_k (x, n)\}$ \cite{Roman} are time-space harmonic with respect to the random walk
$\{n \boldsymbol{.} \left[ \frac{1}{2} \left( - 1 \boldsymbol{.} \xi + u \right) \right]\}_{n \geq 0}.$ Indeed,
$$\mathcal{E}_k (x, n) = \E\left[\left(x - n \boldsymbol{.} \left[ \frac{1}{2} \left( - 1 \boldsymbol{.} \xi + u \right) \right]\right)^k \right], \qquad k = 1,2,\ldots.$$
The $n$-th increment $M_n$ in Example~\ref{3.13} is represented by the umbra $ \frac{1}{2} \left( - 1 \boldsymbol{.} \xi \right.$ $\left. + \, u \right)$ corresponding to a Bernoulli r.v.\ with parameter
$1/2.$

\medskip
\begin{center}
{\sl Krawtchouk polynomials}
\end{center}
\smallskip

The Krawtchouk polynomials $\{\mathcal{K}_k (x, p, n)\}$ \cite{Roman} are time-space harmonic with respect to the random walk $\{n \boldsymbol{.} (-1 \boldsymbol{.} \upsilon)\}_{n \geq 0}.$
Indeed,
$$\frac{n!}{(n - k)!} \, \mathcal{K}_k (x, p, n) = \sum_{i = 1}^k \E[(x + n \boldsymbol{.} \upsilon)^i] \, B_{k,i}(m_1, \ldots, m_{k - i + 1}),$$ where $B_{k,i}$ are partial Bell exponential polynomials,
$$m_i = \E\left[\left(\chi \boldsymbol{.} \left(-1 \boldsymbol{.} \chi - \frac{\chi}{d} \right)\right)^i\right]$$
for $i = 1, \ldots, k,$ and $p/q=d,$ $p+q=1.$ The $n$-th increment $M_n$ in Example~\ref{3.13} is represented by the umbra $-1 \boldsymbol{.} \upsilon$ corresponding to a Bernoulli r.v.\ with parameter $p.$

\medskip
\begin{center}
{\sl Pseudo--Narumi polynomials}
\end{center}
\smallskip

The pseudo--Narumi polynomials $\{N_k (x, a n)\},$ with $a \in
\mathbb{N}$ \cite{Roman},\break
are time-space harmonic with respect to the
random walk\break
$\{(an). (-1 \boldsymbol{.} \iota)\}_{n \geq 0},$ with $a \boldsymbol{.} (-1 \boldsymbol{.} \iota)$ the umbral counterpart of a sum
of $a \in \mathbb{N}$ r.v.'s with uniform distribution on the interval $[0,1].$ Indeed,
$$k! \, N_k (x, a n) = \sum_{i = 1}^k \E[(x - (a n) \boldsymbol{.} (-1 \boldsymbol{.} \iota))^i] B_{k,i}(m_1, \ldots, m_{k - i +1}),$$
where $B_{k,i}$ are partial Bell exponential polynomials and $m_i =\break \E[(u^{\scriptscriptstyle <-1>})^i],$
for $i = 1, \ldots, k.$
\section{Acknowledgments}
I would like to thank Domenico Senato and Henry Crapo for helpful discussions and suggestions during the writing of this survey.
I would like to also thank Pasquale Petrullo and Rosaria Simone for carefully reading the paper, leading to an improvement of its technical
quality. A special thank to Christian Krattenthaler whose remarks and criticisms have stimulated the author in looking for a simplified
exposure devoted to readers coming in new to the matter.
{}
\end{document}